\documentclass[11pt,a4paper,titlepage,twoside
]{book}

\usepackage[utf8x]{inputenc}
\usepackage[T1]{fontenc}
\usepackage{kpfonts}

\usepackage[a4paper,includeheadfoot,pdftex,textwidth=16cm,textheight=24cm,
bottom=3.6cm]{geometry}
\usepackage[svgnames]{xcolor}
\usepackage{graphicx}

\usepackage[bookmarks=true,
pdfborder={0 0 1},
colorlinks=true,urlcolor=blue,citecolor=Purple, 
linkcolor=NavyBlue,hypertexnames=false]{hyperref}

\usepackage{enumitem}
\setlist{parsep=0pt} 
\setlist[itemize,enumerate]{nolistsep,itemsep=3pt,topsep=5pt} 
\setlist{leftmargin=5mm}

\usepackage{fancybox}
\usepackage[Lenny]{fncychap}
\usepackage{fancyhdr}
\usepackage{amsmath}
\usepackage{amsfonts}
\usepackage{amssymb}
\usepackage{amsthm}
\usepackage{ upgreek }

\usepackage{bbm}

\usepackage{mathtools}
\usepackage{mdframed}
\usepackage{tikz}
\usetikzlibrary{matrix,arrows,calc}
\usepgflibrary{shapes}
\usepgflibrary{fpu}
\usepackage{chessboard}

\usepackage[margin=10pt,font=small,labelfont=bf,
labelsep=endash]{caption}

\definecolor{ThmColor}{rgb}{0.93,0.93,0.995}
\definecolor{DefColor}{rgb}{0.92,0.96,0.985}
\definecolor{RemColor}{rgb}{0.96,0.93,0.96}
\definecolor{ExoColor}{rgb}{0.905,0.995,0.905}

\mdfdefinestyle{thmstyle}{backgroundcolor=ThmColor,nobreak,innertopmargin=0pt}
\mdfdefinestyle{defstyle}{backgroundcolor=DefColor,nobreak,innertopmargin=0pt} 
\mdfdefinestyle{remstyle}{backgroundcolor=RemColor,innertopmargin=0pt} 
\mdfdefinestyle{exostyle}{backgroundcolor=ExoColor,innertopmargin=0pt} 

\mdtheorem[style=thmstyle]{theorem}%
{Th\'eor\`eme}[section]

\mdtheorem[style=thmstyle]{proposition}[theorem]{Proposition}[section]
\mdtheorem[ntheorem,style=thmstyle]{corollary}[theorem]{Corollaire}[section]
\mdtheorem[ntheorem,style=thmstyle]{lemma}[theorem]{Lemme}[section]

\mdtheorem[ntheorem,style=defstyle]{definition}[theorem]{D\'efinition}[section]
\mdtheorem[ntheorem,style=defstyle]{notation}[theorem]{Notation}[section]
\mdtheorem[ntheorem,style=defstyle]{assumption}[theorem]{hypoth\`ese}[section]

\mdtheorem[ntheorem,style=remstyle]{example}[theorem]{Exemple}[section]
\mdtheorem[ntheorem,style=remstyle]{remark}[theorem]{Remarque}[section]

\mdtheorem[ntheorem,style=exostyle]{exercise}[theorem]{Exercice}[section]

\newcommand{\CM}{cha\^ine de Markov}
\newcommand{\CCM}{Cha\^ine de Markov}
\newcommand{\CMs}{cha\^ines de Markov}
\newcommand{\reaches}{\rightsquigarrow}
\newcommand{\Tc}{T_{\text{c}}}
\newcommand{\myquote}[1]{\guillemotleft\;#1\;\guillemotright}

\usepackage{cleveref}
\crefname{exercise}{exercise}{exercises}

\usepackage{autonum} 

\tikzset{myxshift/.style = {shift = {(#1, 0)}}}
\tikzset{myyshift/.style = {shift = {(0, #1)}}}
\newcommand{\pos}[2]{
   \def\posx{{#1}}
   \def\posy{{#2}}
}

\newcommand{\urntikz}
  {
  \begin{scope}[myxshift = \posx]
  \begin{scope}[myyshift = \posy]
  \draw[thick,-]    
    (-1.1,1.0) -- (-1.1,0.2)
    (-1.1,0.2) arc (180:270:0.2)
    (-0.9,0.0) -- (-0.3,0.0)
    (-0.3,0.0) arc (-90:0:0.2)
    (-0.1,0.2) -- (-0.1,1.0)
  ;
  \end{scope}
  \end{scope}
  }

\input{sarajevo.sty}

\DeclareMathOperator{\pgcd}{pgcd}       
\newcommand{\vone}{\mathbf{1}}
\newcommand{\myvrule}[3]{\vrule height #1 depth #2 width #3}

\begin{document}

\pagestyle{empty}
\newgeometry{margin=1in}

\hypersetup{pageanchor=false}

\thispagestyle{empty}

\vspace*{1cm}
\begin{center}

{\Huge\bfseries\scshape
Processus al\'eatoires et applications \\[1mm]
 -- Algorithmes MCMC et vitesse de convergence \\[1mm]
}

\vspace*{12mm}
{\large Nils Berglund}\\[2mm]
{\large Institut Denis Poisson -- UMR 7013}\\[2mm]
{\large Universit\'e d'Orl\'eans, Universit\'e de Tours, CNRS}

\vspace*{12mm}
{\Large Notes de cours}\\[4mm]
\vspace*{12mm}


\vspace*{27mm}
--- Version du 9 d\'ecembre 2024 ---\\[2mm]

\end{center}
\hypersetup{pageanchor=true}

\cleardoublepage

\pagestyle{fancy}
\fancyhead[RO,LE]{\thepage}
\fancyhead[LO]{\nouppercase{\rightmark}}
\fancyhead[RE]{\nouppercase{\leftmark}}
\cfoot{}
\setcounter{page}{1}
\pagenumbering{roman}
\restoregeometry

\tableofcontents

\cleardoublepage

\setcounter{page}{1}
\pagenumbering{arabic}


\part[Cha\^ines de Markov \`a espace d\'enombrable]{Cha\^ines de Markov\\ \`a espace d\'enombrable}
\label{part:cm_denombrable} 


\chapter{Exemples de cha\^ines de Markov}
\label{chap:cm_exemple} 


\section{Textes al\'eatoires}
\label{sec:ex_textes} 

Les \CMs\ ont \'et\'e introduites au d\'ebut du vingti\`eme si\`ecle par 
le math\'ematicien russe Andrey Markov, dans le but d'\'etudier des suites 
de variables al\'eatoires non ind\'ependantes. L'une des premi\`ere applications 
\'etait l'analyse de la distribution de voyelles dans des romans. 

Dans un \'etat d'esprit similaire, voici trois \myquote{textes}\ g\'en\'er\'es 
de mani\`ere al\'eatoire~:

\begin{enumerate}
\item[A.] 
\begin{mdframed}[innerleftmargin=7mm,innertopmargin=10pt,innerbottommargin=10pt]
{\sf
YxUV,luUqHCLvE?,MRiKaoiWjyhg nEYKrMFD!rUFUy.qvW;e:FflN.udbBdo!, \\
ZpGwTEOFcA;;RrSMvPjA'Xtn.vP?JNZA;xWP, Cm?;i'MzLqVsAnlqHyk,ghDT  \\
:PwSwrnJojRhVjSe?dFkoVRN!MTfiFeemBXITdj m.h d'ea;Jkjx,XvHIBPfFT \\
s I'SLcSX;'X!S, ODjX.eMoLnQttneLnNE!qGRgCJ:BuYAauJXoOCCsQkLcyPO \\
MulKLRtSm;PNpFfp'PfgvIJNrUr t l aXtlA?;TPhPxU:,ZmVGr,,'DIjqZDBY \\
DrkPRiKDYRknDhivt;, LYXDuxNKpjegMvrtfz:JpNTDj'LFmHzXxotRM u.iya \\
UUrgZRcA QmCZffwsNWhddBUPAhJIFJvs.CkKFLJoXef;kCnXrv'uWNcpULYsnl \\
Kg OURmysAnxFjHawwsSpM H;PWPsMaFYLMFyvRWOjbdPlLQIaaspNZkuO'Ns.l \\
jEXO,lxQ'GS;n;H:DH:VWJN :t'JMTUVpKCkVZ'NyKJMGiIbQFXEgDEcWxMBiyo \\
ybRIWIAC deMJnnL;SBAZ?:.UuGnC:B.!lBUT,pT?tyHHLlCvN, mKZgwlMJOJd \\
HHobua;KU.;kADVM?jr'v.SCq:hZLR;lqkmLkhn:ajhBM,gKexDAro,HlczWTv \\
cFmNPt.MudUWPO, sTrWlJdgjoiJd.:d;CpJkJCW;FIRnpMGa;umFysOMAqQtmT \\
pPaYZKtOFYppeE.KFX?SuvcbaDrQ XECelD;cfoQKf?'jCTUaISS;fV:gqoWfSq \\
k:Tf!YuPBANtKhewiNg'ImOFs:UhcExmBjsAaMhBf UVP, 'dcFk;gxJMQGyXI; \\
nVwwfWxS:YXQMELEIObTJiilUYSlOsg.gCqlrN:nEU:irHM'nOLXWUbJLTU re' \\
kk vAwMgt'KgWSxwxqJe,z'OBCrnoIshSCDlZirla,rWNPkc?UgZm GOBX.QylY \\
jOtuF
}
\end{mdframed}

\item[B.]	
\begin{mdframed}[innerleftmargin=7mm,innertopmargin=10pt,innerbottommargin=10pt]
{\sf
nsunragetnetelpnlac.  pieln tJmends d e.imnqu  caa  aneezsconns re.tc oml d
 e  c, paeisfuaul irt ssna l df.ieulat a ese t hre edn ro  m eeel
slsplotasstp etuoMeiiseeaenemzeaeuqpeer  enuoco  sfehnnir p ts 'mpisu qrd
iraLp nFetesa,opQeey rieeaduset Mu\-uisecG il e m  ru daeiafasousfnircot i
eeedracev ever.nsn iaeulu!,mtel lpa rdbjdide  tolr'murunlr bteaaua
ieasilureseuavrmoce ntvqm qnurnaunsa.mraayVarinanr  eumsu cnponf ciuo
.pssre  elreeY snrrq aani psu oqoddaiaaomrssloe'avia,loei va
eroltrsurdeduuoe ffusir 'th'niIt has,slluoooe tee  ?eoxaea slsii i u
edtvsear e,Mesatnd o o rvdocaeagiua  apugiqn rclt  smtee.te, gceade etsn e
v in eag ent so  ra te,  oi seGndd  i eeet!dii e  ese nanu d sp ul afeen
aqelonens ssisaaoe cs     eectadegotuudlru  i  'c, uuuuts 'tt , dir
atermdmuciqedn  esovsioieieerxdroie mqso,es rrvteen,r dtei xcalrionuaae e
vtmplsz miuqa   u aboir br gmcdexptedn pEua't vm vnic eeren ereaa,eegeta u
rss nlmxomas ea nsbnt s,eEpeteae teiasbo cd ee tu em ue quee en, sd
eeneepeot 
}
\end{mdframed}

\item[C.]
\begin{mdframed}[innerleftmargin=7mm,innertopmargin=10pt,innerbottommargin=10pt]
{\sf
cesalu'act, bouleuivoie melarous die ndant leuvoiblue poit pesois
deuntaciroverchu llie e lle s r lerchar, laisueuayaissabes vet s cuetr i
as, rdetite se d'iretie, de.. nendoules, le pablur e d ! copomouns ppait
limmix a r aux urars laie Le r lercret ce c. n'are four nsirepapole pa vr
s, nte le efit. itesit, le faun e ju estatusuet usoin prcilaisanonnout ssss
l tosesace cole sientt, dent pontrtires. e, l mentoufssss chat Laneus c
Chontrouc Ce e. Et deses j'ecci uleus mmon s mauit paga lanse l cont
ciquner e c Cha s l'a Jes des s'erattrlunt es de sacouen erends. ve e quns
som'a aisajouraite eux lala pour ! a levionible plaint n ss, danetrc ponce
con du lez, l danoit, dirvecs'u ce ga vesai : chleme eesanl Pa chiontotes
anent fomberie vaud'untitez e esonsan t a ! bondesal'is Ilaies, vapa e !
Lers jestsiee celesu unallas, t. ces. ta ce aielironi mmmileue cecoupe et
dennt vanen A la ajole quieet, scemmu tomtemotit me aisontouimmet Le s
Prage ges peavoneuse ! blec douffomurrd ntis.. rur, ns ablain i pouilait
lertoipr ape. leus icoitth me e e, poiroia s. ! atuepout somise e la as
}
\end{mdframed}

\end{enumerate}

Il est clair qu'aucun de ces textes n'a de signification. Toutefois,
le texte B.\ semble moins arbitraire que le texte A., et C.\ para\^\i
t moins \'eloign\'e d'un texte fran\c cais que B. Il suffit pour cela
d'essayer de lire les textes \`a haute voix. 

Voici comment ces textes ont \'et\'e g\'en\'er\'es. Dans les trois cas, on
utilise le m\^eme alphabet de 60 lettres (les 26 minuscules et majuscules, 
quelques signes de ponctuation et l'espace). 

\begin{enumerate}
\item	Pour le premier texte, on a simplement tir\'e au hasard, de
mani\`ere ind\'ependante et avec la loi uniforme, des lettres de
l'alphabet. 

\item	Pour le second texte, on a tir\'e les lettres de mani\`ere
ind\'ependante, mais pas avec la loi uniforme. Les probabilit\'es des
diff\'erentes lettres correspondent aux fr\'equences de ces lettres dans
un texte de r\'ef\'erence fran\c cais (en l’occurrence, un extrait du {\sl
Colonel Chabert}\/ de Balzac). Les fr\'equences des diff\'erentes
lettres du texte al\'eatoire sont donc plus naturelles, par exemple la
lettre {\sf e} appara\^\i t plus fr\'equemment (dans $13\%$ des cas) que la
lettre {\sf z} ($0.2\%$). 

\item	Pour le dernier texte, enfin, les lettres n'ont pas \'et\'e
tir\'ees de mani\`ere ind\'ependante, mais d\'ependant de la lettre
pr\'ec\'edente. Dans le m\^eme texte de r\'ef\'erence que
pr\'e\-c\'edemment, on a d\'etermin\'e avec quelle fr\'equence la lettre
{\sf a} est suivie de {\sf a} (jamais), {\sf b} (dans $3\%$ des cas), et
ainsi de suite, et de m\^eme pour toutes les autres lettres. Ces
fr\'equences ont ensuite \'et\'e
choisies comme probabilit\'es de transition lors de la g\'en\'eration du
texte. 
\end{enumerate}

Ce proc\'ed\'e peut facilement \^etre am\'elior\'e, par exemple en faisant
d\'ependre chaque nouvelle lettre de plusieurs lettres pr\'ec\'edentes. 
Mais m\^eme avec une seule lettre pr\'ec\'edente, il est remarquable que 
les textes engendr\'es permettent assez facilement de reconna\^\i tre la
langue du texte de r\'ef\'erence, comme en t\'emoignent ces deux exemples:

\begin{enumerate}
\item[D.]
\begin{mdframed}[innerleftmargin=7mm,innertopmargin=10pt,innerbottommargin=10pt]
{\sf 
deser Eld s at heve tee opears s cof shan; os wikey coure tstheevons irads;
Uneer I tomul moove t nendoot Heilotetateloreagis his ud ang l ars thine
br, we tinond end cksile: hersest tear, Sove Whey tht in t ce tloour ld t
as my aruswend Ne t nere es alte s ubrk, t r s; penchike sowo
Spotoucthistey psushen, ron icoowe l Whese's oft Aneds t aneiksanging t
ungl o whommade bome, ghe; s, ne. torththilinen's, peny. d llloine's anets
but whsto a It hoo tspinds l nafr Aneve powit tof f I afatichif m as tres,
ime h but a wrove Les des wined orr; t he ff teas be hende pith hty ll ven
bube. g Bube d hitorend tr, Mand nd nklichis okers r whindandy, Sovede brk
f Wheye o edsucoure, thatovigh ld Annaix; an eer, andst Sowery looublyereis
isthalle Base whon ey h herotan wict of les, h tou dends m'dys h Wh
on'swerossictendoro whaloclocotolfrrovatel aled ouph rtrsspok,
ear'sustithimiovelime From alshis ffad, Spake's wen ee: hoves aloorth
erthis n t Spagovekl stat hetubr tes, Thuthiss oud s hind t s potrearall's
ts dofe 
}\footnote{Texte de r\'ef\'erence: Quelques sonnets de Shakespeare.}
\end{mdframed}

\item[E.]
\begin{mdframed}[innerleftmargin=7mm,innertopmargin=10pt,innerbottommargin=10pt]
{\sf
dendewoch wich iere Daf' lacht zuerckrech, st, Gebr d, Bes.
jenditerullacht, keie Un! etot' in To sendenus scht, ubteinraben Qun Jue
die m arun dilesch d e Denuherelererufein ien. seurdan s ire Zein. es min?
dest, in. maur as s san Gedein it Ziend en desckruschn kt vontimelan. in,
No Wimmmschrstich vom delst, esichm ispr jencht sch Nende Buchichtannnlin
Sphrr s Klldiche dichwieichst. ser Bollesilenztoprs uferm e mierchlls aner,
d Spph! wuck e ing Erenich n sach Men. Sin s Gllaser zege schteun d,
Gehrstren ite Spe Kun h Umischr Ihngertt, ms ie. es, bs de! ieichtt f;
Ginns Ihe d aftalt veine im t'seir; He Zicknerssolanust, fllll. mmichnennd
wigeirdie h Zierewithennd, wast naun Wag, autonbe Wehn eietichank We
dessonindeuchein ltichlich bsch n, Ichritienstam Lich uchodigem Din eieiers
die it f tlo nensseicichenko Mechtarzaunuchrtzubuch aldert; l von. fteschan
nn ih geier Schich Geitelten Deichst Fager Zule fer in vischtrn; Schtih Un
Hit ach, dit? at ichuch Eihra! Hich g ure vollle Est unvochtelirn An 
}\footnote{Texte de r\'ef\'erence: Un extrait du {\sl Faust}\/ de Goethe.}
\end{mdframed}
\end{enumerate}

Cela donne, inversement, une m\'ethode assez \'economique permettant \`a
une machine de d\'eterminer automatiquement dans quelle langue un texte est
\'ecrit. C'est un exemple tr\`es simplifi\'e d'intelligence artificielle, 
ex\'ecutant une t\^ache d'apprentissage profond.


\section{Mod\`ele d'urnes d'Ehrenfest}
\label{sec:ex_Ehrenfest} 

Ce mod\`ele d'urnes a \'et\'e introduit en 1907 par Paul et Tatjana Ehrenfest, 
dans le but de comprendre le \myquote{paradoxe}\ de 
l'irr\'eversibilit\'e.
Il s'agit du probl\`eme suivant. Un syst\`eme microscopique, constitu\'e de
mol\'ecules qui s'en\-tre\-choquent, ob\'eit, du moins en m\'ecanique classique, 
aux lois de Newton. Ces lois sont \emph{r\'eversibles}, ce qui signifie que 
si l'on parvenait \`a filmer les mol\'ecules pendant un intervalle de temps, 
et qu'on passait le film \`a l'envers, cette \'evolution renvers\'ee ob\'eirait 
encore aux lois de Newton. Par cons\'equent, rien ne permettrait de dire quel film
est pass\'e \`a l'endroit ou \`a l'envers.

Dans notre monde macroscopique, en revanche, les ph\'enom\`enes sont pour la 
plupart \emph{irr\'eversibles}. Un verre qui tombe se brise, mais on n'observe 
jamais des morceaux de verre s'assembler spontan\'ement. Une goutte de colorant 
dans de l'eau se dilue au cours du temps, mais on ne voit jamais le colorant dilu\'e 
se concentrer en un point. Comment se peut-il qu'un syst\`eme r\'eversible \`a 
l'\'echelle microscopique se comporte de mani\`ere irr\'eversible \`a notre 
\'echel\-le macroscopique~?

\`A un niveau un peu moins macroscopique, consid\'erons deux r\'ecipients, l'un rempli 
d'un gaz, et l'autre vide. Les deux r\'ecipients sont mis en contact, et au temps $0$, 
on ouvre une vanne permettant aux mol\'ecules du gaz de se r\'epartir entre les deux 
r\'ecipients. On observe alors la pression du gaz s'\'equilibrer entre les deux 
r\'ecipients, mais on ne s'attend pas \`a voir toutes les mol\'ecules spontan\'ement 
revenir dans un r\'ecipient. 

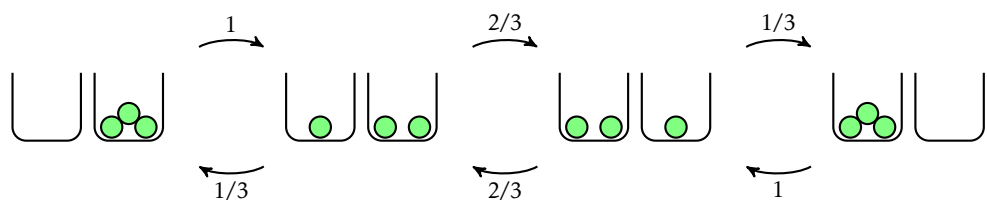
\begin{figure}
\vspace{-3mm}
\begin{center}
\begin{tikzpicture}[->,>=stealth',auto,scale=0.9,node
distance=3.0cm, thick,main node/.style={circle,scale=0.7,minimum size=0.4cm,
fill=green!50,draw,font=\sffamily}]
  
  \pos{0}{0} \urntikz
  \pos{1.2}{0} \urntikz
  
  \node[main node] at(0.35,0.2) {};
  \node[main node] at(0.85,0.2) {};
  \node[main node] at(0.6,0.4) {};
  
  \pos{4}{0} \urntikz
  \pos{5.2}{0} \urntikz

  \node[main node] at(4.35,0.2) {};
  \node[main node] at(4.85,0.2) {};
  \node[main node] at(3.4,0.2) {};

  \pos{8}{0} \urntikz
  \pos{9.2}{0} \urntikz

  \node[main node] at(7.15,0.2) {};
  \node[main node] at(7.65,0.2) {};
  \node[main node] at(8.6,0.2) {};

  \pos{12}{0} \urntikz
  \pos{13.2}{0} \urntikz

  \node[main node] at(11.15,0.2) {};
  \node[main node] at(11.65,0.2) {};
  \node[main node] at(11.4,0.4) {};
  
  \node[minimum size=2.2cm] (0) at (0.1,0.5) {};
  \node[minimum size=2.2cm] (1) at (4.1,0.5) {};
  \node[minimum size=2.2cm] (2) at (8.1,0.5) {};
  \node[minimum size=2.2cm] (3) at (12.1,0.5) {};
  
  \path[shorten >=.3cm,shorten <=.3cm,every
        node/.style={font=\sffamily\footnotesize}]
    (0) edge [bend left,above] node {$1$} (1)
    (1) edge [bend left,above] node {$2/3$} (2)
    (2) edge [bend left,above] node {$1/3$} (3)
    (3) edge [bend left,below] node {$1$} (2)
    (2) edge [bend left,below] node {$2/3$} (1)
    (1) edge [bend left,below] node {$1/3$} (0)
    ;
\end{tikzpicture}
\end{center}
\vspace{-7mm}
 \caption[]{Le mod\`ele d'urnes d'Ehrenfest, dans le cas de $3$ boules.}
 \label{fig_ehrenfest}
\end{figure}

Le mod\`ele des urnes d'Ehrenfest est un mod\`ele al\'eatoire repr\'esentant 
cette situation. On consid\`ere $N$ boules r\'eparties sur deux urnes. \`A chaque 
pas de temps, on choisit l'une des $N$ boules uniform\'ement au hasard, et on 
la change d'urne (\figref{fig_ehrenfest}). Soit $X_n$ le nombre de boules 
dans l'urne de gauche au $n$i\`eme pas de temps. On a alors 
\begin{equation}
 X_{n+1} = 
 \begin{cases}
  X_n + 1 & \text{avec probabilit\'e $1 - \frac{X_n}{n}$\;,} \\
  X_n - 1 & \text{avec probabilit\'e $\frac{X_n}{n}$\;.}
 \end{cases}
\end{equation} 
La probabilit\'e de cette transition ne d\'epend que de $X_n$, pas des \'etats aux temps 
pr\'ec\'edents, et est ind\'ependante des transitions pr\'ec\'edentes. 

Il s'agit d'un exemple de \CM\ sur $\set{0,1,\dots,N}$, qui a des propri\'et\'es 
garantissant que la loi de $X_n$ converge vers une loi limite (qui s'av\`ere 
\^etre une loi binomiale). De plus, on peut calculer le \defwd{temps de r\'ecurrence
moyen} vers l'\'etat de d\'epart, $X_0 = N$~: il est \'egal \`a $2^N$. Ceci donne 
une r\'eponse au paradoxe de l'irr\'eversibilit\'e~: s'il est effectivement possible 
qu'un \'ev\'enement qui contredit cette irr\'eversibilit\'e arrive (toutes les 
boules retournent dans l'urne de d\'epart), le temps n\'ecessaire pour l'observer 
est extr\^emement grand. D\'ej\`a pour $N=1000$, on a 
\begin{equation}
 2^N = 2^{1000} = (2^{10})^{100} > (10^3)^{100} = 10^{300}\;.
\end{equation} 
M\^eme pour un pas de temps d'une nanoseconde ($10^{-9}$ secondes), ce temps 
est de $10^{291}$ secondes. Une ann\'ee comporte environ $3\cdot 10^7$ secondes, 
donc il faudra attendre en moyenne plus de $10^{283}$ ans pour voir toutes les 
mol\'ecules dans le r\'ecipient de gauche, ce qui est largement sup\'erieur \`a l'\^age 
estim\'e de notre univers. Si $N$ est comparable au nombre d'Avogadro, ce temps 
de r\'ecurrence est encore beaucoup plus grand. 


\section{Marches al\'eatoires}
\label{sec:ex_MA} 

Les marches al\'eatoires constituent un exemple relativement simple,
et n\'eanmoins tr\`es important de \CMs\ sur un ensemble
d\'enombrable infini. Dans ce cas, en effet, $\cX=\Z^d$ est un r\'eseau
infini, de dimension $d\in\N^*$. Souvent, on consid\`ere que la
\CM\ d\'emarre en $X_0=0$. Ensuite, elle choisit \`a chaque instant
l'un des $2d$ sites voisins, selon une loi fix\'ee d'avance. 

Une \defwd{marche al\'eatoire}\/ sur $\Z^d$ est donc une \CM\ \`a
valeurs dans $\Z^d$, de distribution initiale telle que $\prob{X_0 = 0} = 1$, et de
probabilit\'es de transition satisfaisant 
\begin{equation}
\label{rw1}
\pcond{X_{n+1} = y}{X_n = x} = 0
\qquad
\text{si $x=y$ ou $\norm{x-y}>1$\;.}
\end{equation}
La marche est dite \defwd{sym\'etrique}\/ si 
\begin{equation}
\label{rw2}
\pcond{X_{n+1} = y}{X_n = x} = \frac1{2d}
\qquad
\text{pour $\norm{x-y}=1$\;.}
\end{equation}
Les trajectoires de la marche al\'eatoire sont des suites de points de
$\Z^d$ \`a distance $1$, qu'on a coutume d'identifier \`a la ligne
bris\'ee reliant ces points (\figref{fig_rw2d}). 

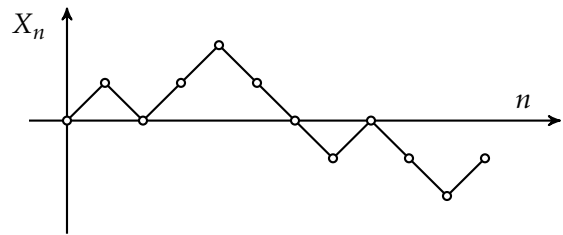
\begin{figure}
\begin{center}
\begin{tikzpicture}[-,scale=0.5,auto,node
distance=1.0cm, thick,main node/.style={draw,circle,fill=white,minimum
size=3pt,inner sep=0pt}]

  \path[->,>=stealth'] 
     (-1,0) edge (13,0)
     (0,-3) edge (0,3)
  ;

  \node at (12.0,0.5) {$n$};
  \node at (-1.0,2.5) {$X_n$};

  \draw (0,0) node[main node] {} 
  -- (1,1) node[main node] {} 
  -- (2,0) node[main node] {} 
  -- (3,1) node[main node] {} 
  -- (4,2) node[main node] {} 
  -- (5,1) node[main node] {} 
  -- (6,0) node[main node] {} 
  -- (7,-1) node[main node] {} 
  -- (8,0) node[main node] {} 
  -- (9,-1) node[main node] {} 
  -- (10,-2) node[main node] {} 
  -- (11,-1) node[main node] {} 
  ;
\end{tikzpicture}
\end{center}
\vspace{-5mm}
 \caption[]{Une r\'ealisation d'une marche al\'eatoire
unidimensionnelle.}
 \label{fig_marche1}
\end{figure}

Notons que $X_n$ est la somme de $n$ variables al\'eatoires ind\'ependantes, 
de m\^eme loi uniforme sur les $2d$ voisins de $0$ dans $\Z^d$. Ceci permet 
d'appliquer des th\'eor\`emes limites tels que le th\'eor\`eme central 
limite \`a l'\'etude de $X_n$ pour $n$ grand. En particulier, l'esp\'erance 
de $X_n$ est nulle pour tout $n$, et sa variance est proporionnelle \`a $n$. 

\begin{figure}
\begin{center}
\begin{tikzpicture}[-,scale=0.5,auto,node
distance=1.0cm, thick,main node/.style={draw,circle,fill=white,minimum
size=3pt,inner sep=0pt}]

  \path[->,>=stealth'] 
     (-4,0) edge (8,0)
     (0,-5) edge (0,3)
  ;

  \draw[very thick] (0,0) node[main node,thick] {} 
  -- (0,1) node[main node,thick] {} 
  -- (1,1) node[main node,thick] {} 
  -- (1,0) node[main node,thick] {} 
  -- (2,0) node[main node,thick] {} 
  -- (2,-1) node[main node,thick] {} 
  -- (1,-1) node[main node,thick] {} 
  -- (1,-2) node[main node,thick] {} 
  -- (2,-2) node[main node,thick] {} 
  -- (2,-3) node[main node,thick] {} 
  -- (1,-3) node[main node,thick] {} 
  -- (0,-3) node[main node,thick] {} 
  -- (-1,-3) node[main node,thick] {} 
  -- (-2,-3) node[main node,thick] {} 
  -- (-2,-2) node[main node,thick] {} 
  -- (-1,-2) node[main node,thick] {} 
  -- (-1,-3) node[main node,thick] {} 
  -- (-1,-4) node[main node,thick] {} 
  -- (0,-4) node[main node,thick] {} 
  -- (0,-3) node[main node,thick] {} 
  -- (1,-3) node[main node,thick] {} 
  -- (1,-4) node[main node,thick] {} 
  -- (2,-4) node[main node,thick] {} 
  -- (3,-4) node[main node,thick] {} 
  -- (4,-4) node[main node,thick] {} 
  -- (5,-4) node[main node,thick] {} 
  -- (5,-3) node[main node,thick] {} 
  -- (5,-2) node[main node,thick] {} 
  -- (4,-2) node[main node,thick] {} 
  -- (4,-3) node[main node,thick] {} 
  -- (5,-3) node[main node,thick] {} 
  -- (6,-3) node[main node,thick] {} 
  ;
\end{tikzpicture}
\end{center}
\vspace{-5mm}
 \caption[]{Une trajectoire d'une marche al\'eatoire en dimension $d=2$.}
 \label{fig_rw2d}
\end{figure}
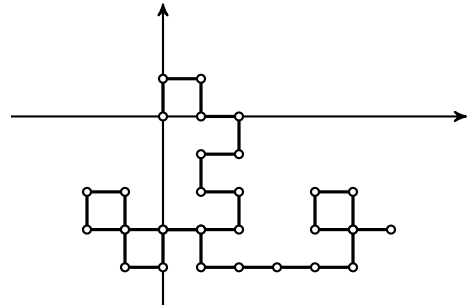

Par exemple, en dimension $d=1$, on trouve 
\begin{equation}
 \prob{X_n = x} = \frac1{2^n}\binom{n}{\frac{n+x}2}
\qquad
\forall x\in\set{-n,-n+2,\dots,n-2,n}\;.
\end{equation} 
\`A une transformation affine pr\`es, $X_n$ suit une loi binomiale (plus pr\'ecis\'ement, 
$(X_n + n)/2$ suit une loi binomiale). Son esp\'erance 
est nulle, et sa variance est \'egale \`a $n$. Ceci implique en particulier que 
la marche va finir par atteindre n'importe quel point de $\Z$ si l'on attend assez 
longtemps. Par ailleurs, $\prob{X_n = x}$ tend vers $0$ lorsque $n$ tend vers l'infini, 
pour tout $x$ fix\'e. La loi de $X_n$ n'admet donc pas de loi limite. Des propri\'et\'es 
similaires sont vraies pour la marche al\'eatoire sym\'etrique sur $\Z^d$. 


\section{Mod\`ele d'Ising}
\label{sec:ex_Ising} 

Le mod\`ele d'Ising (ou de Lenz--Ising), fut introduit en 1920 par le 
physicien Wilhelm Lenz, et \'etudi\'e en dimension $1$ par son 
\'etudiant Ernst Ising. 
Comme le mod\`ele d'Ehrenfest, ce mod\`ele vient de la physique, plus
particuli\`erement de la physique statistique. Il est cens\'e d\'ecrire un
ferro-aimant, qui a la propri\'et\'e de s'aimanter spontan\'ement \`a
temp\'erature suffisamment basse. 

On consid\`ere une partie (connexe)
$\Lambda$ du r\'eseau $\Z^d$ ($d$ \'etant la dimension du syst\`eme, par
exemple $3$), contenant $N$ sites. A chaque site, on attache un \myquote{spin}\ 
(une sorte d'aimant \'el\'ementaire), prenant valeurs $+1$ ou
$-1$. Un choix d'orientations de tous les spins s'appelle une
configuration, c'est donc un \'el\'ement de l'espace de configuration
$\cX=\set{-1,1}^\Lambda$ (\figref{fig_ising}). A une configuration $x\in\cX$, on associe
l'\'energie 
\begin{equation}
\label{intro1}
H(x) = -\sum_{\langle i,j\rangle\in\Lambda} x_ix_j 
- h \sum_{i\in\Lambda}x_i\;.
\end{equation}
Ici, la notation $\langle i,j\rangle$ indique que l'on ne somme que sur les paires de
spins plus proches voisins du r\'eseau, c'est--\`a--dire \`a une distance
$1$. Le premier terme est donc d'autant plus grand qu'il y a de spins
voisins diff\'erents. Le second terme d\'ecrit l'interaction avec un champ
magn\'etique ext\'erieur $h$. Il est d'autant plus grand qu'il y a de
spins oppos\'es au champ magn\'etique. 

\begin{figure}
\begin{center}
\begin{tikzpicture}[thick,auto,node distance=0.5cm,every
node/.style={font=\sffamily\LARGE}]

  \draw [fill=yellow!30] (-0.3,-0.3) rectangle (3.8,2.3);

  \node[blue] (00) {$-$};
  \node[red]  (10) [right of=00] {$+$};
  \node[red]  (20) [right of=10] {$+$};
  \node[blue] (30) [right of=20] {$-$};
  \node[blue] (40) [right of=30] {$-$};
  \node[blue] (50) [right of=40] {$-$};
  \node[blue] (60) [right of=50] {$-$};
  \node[red]  (70) [right of=60] {$+$};

  \node[red]  (01) [above of=00] {$+$};
  \node[blue] (11) [right of=01] {$-$};
  \node[blue] (21) [right of=11] {$-$};
  \node[red]  (31) [right of=21] {$+$};
  \node[blue] (41) [right of=31] {$-$};
  \node[red]  (51) [right of=41] {$+$};
  \node[blue] (61) [right of=51] {$-$};
  \node[red]  (71) [right of=61] {$+$};

  \node[blue] (02) [above of=01] {$-$};
  \node[blue] (12) [right of=02] {$-$};
  \node[red]  (22) [right of=12] {$+$};
  \node[blue] (32) [right of=22] {$-$};
  \node[red]  (42) [right of=32] {$+$};
  \node[red]  (52) [right of=42] {$+$};
  \node[blue] (62) [right of=52] {$-$};
  \node[red]  (72) [right of=62] {$+$};

  \node[red]  (03) [above of=02] {$+$};
  \node[blue] (13) [right of=03] {$-$};
  \node[red]  (23) [right of=13] {$+$};
  \node[red]  (33) [right of=23] {$+$};
  \node[blue] (43) [right of=33] {$-$};
  \node[blue] (53) [right of=43] {$-$};
  \node[blue] (63) [right of=53] {$-$};
  \node[red]  (73) [right of=63] {$+$};

  \node[blue] (04) [above of=03] {$-$};
  \node[red]  (14) [right of=04] {$+$};
  \node[blue] (24) [right of=14] {$-$};
  \node[red]  (34) [right of=24] {$+$};
  \node[red]  (44) [right of=34] {$+$};
  \node[blue] (54) [right of=44] {$-$};
  \node[red]  (64) [right of=54] {$+$};
  \node[blue] (74) [right of=64] {$-$};

  \end{tikzpicture}
\end{center}
\vspace{-5mm}
 \caption[]{Une configuration du mod\`ele d'Ising en dimension $d=2$.}
 \label{fig_ising}
\end{figure}
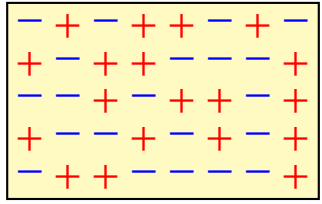

Un principe de base de la physique statistique dit que si un syst\`eme est
en \'equilibre thermique \`a temp\'erature $T$, alors il se trouve dans la
configuration $x$ avec probabilit\'e proportionnelle \`a $\e^{-\beta
H(x)}$ (appel\'ee \defwd{mesure de Gibbs}), o\`u $\beta=1/(k_{\text{B}}T)$, avec 
$k_{\text{B}}$ une constante physique appel\'ee \defwd{constante de Boltzmann}. 
A temp\'erature faible, le syst\`eme privil\'egie les configurations de basse 
\'energie, alors que lorsque la temp\'erature tend vers l'infini, toutes les configurations
deviennent \'equiprobables. 

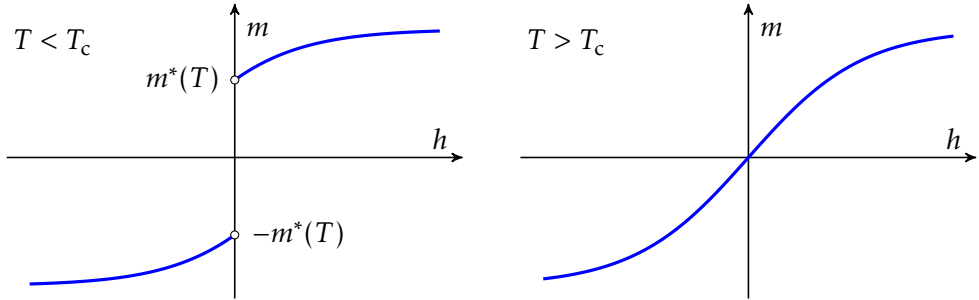
\begin{figure}
\begin{center}
\begin{tikzpicture}[>=stealth',main node/.style={circle,minimum
size=3pt,inner sep=0pt,fill=white,draw},x=3cm,y=1.7cm, 
declare function={m(\x) = tanh(2*\x); mm(\x) = tanh(2*\x +0.7);}]

\draw[->,semithick] (-1,0) -> (1,0);
\draw[->,semithick] (0,-1.1) -> (0,1.2);

\draw[blue,very thick,-,smooth,domain=0.0:0.9,samples=50,/pgf/fpu,
/pgf/fpu/output format=fixed] plot (\x, {mm(\x)});

\draw[blue,very thick,-,smooth,domain=0.0:0.9,samples=50,/pgf/fpu,
/pgf/fpu/output format=fixed] plot (-\x, {-mm(\x)});

\node[] at (0.9,0.15) {$h$};
\node[] at (0.1,1.0) {$m$};

\node[main node] at (0.0, {mm(0)}) {};
\node[main node] at (0.0, {-mm(0)}) {};

\node[] at (-0.23,{mm(0)}) {$m^*(T)$};
\node[] at (0.28,{-mm(0)}) {$-m^*(T)$};

\node[] at (-0.8,0.9) {$T < \Tc$};
\end{tikzpicture}
\hspace{5mm}
\begin{tikzpicture}[>=stealth',main node/.style={circle,minimum
size=0.25cm,fill=blue!20,draw},x=3cm,y=1.7cm, 
declare function={m(\x) = tanh(2*\x); mm(\x) = tanh(2*\x +0.7);}]

\draw[->,semithick] (-1,0) -> (1,0);
\draw[->,semithick] (0,-1.1) -> (0,1.2);

\draw[blue,very thick,-,smooth,domain=-0.9:0.9,samples=100,/pgf/fpu,
/pgf/fpu/output format=fixed] plot (\x, {m(\x)});

\node[] at (0.9,0.15) {$h$};
\node[] at (0.1,1.0) {$m$};

\node[] at (-0.8,0.9) {$T > \Tc$};
\end{tikzpicture}
\end{center}
\vspace{-5mm}
 \caption[]{Aimantation du mod\`ele d'Ising en fonction du champ 
 magn\'etique ext\'erieur $h$, \`a gauche pour $T < \Tc$, et \`a 
 droite pour $T > \Tc$.}
 \label{fig_ising2}
\end{figure}

L'\defwd{aimantation totale} de l'\'echantillon est donn\'ee par la variable
al\'eatoire 
\begin{equation}
\label{intro2}
m(x) = \sum_{i\in\Lambda} x_i\;,
\end{equation}
et son esp\'erance vaut 
\begin{equation}
\label{intro3}
\expec m = 
\dfrac{\displaystyle\sum_{x\in\cX} m(x)
\e^{-\beta H(x)}}
{\displaystyle\sum_{x\in\cX}\e^{-\beta
H(x)}}\;.
\end{equation}
L'int\'er\^et du mod\`ele d'Ising est qu'on peut montrer l'existence d'une
\defwd{transition de phase}, en dimension $d$ sup\'erieure ou \'egale \`a $2$.
Dans ce cas il existe une \defwd{temp\'erature critique} $\Tc$ en-dessous de laquelle
l'aimantation varie de mani\`ere discontinue en fonction de $h$ dans la
limite $N\to\infty$. Pour des temp\'eratures sup\'erieures \`a la valeur
critique, l'aimantation est continue en $h$. Plus pr\'ecis\'ement (\figref{fig_ising2}),
\begin{itemize}
\item   l'aimantation est toujours strictement positive si $h > 0$, 
et strictement n\'egative si $h < 0$;
\item   si $T \geqs \Tc$, alors l'aimantation tend vers $0$ lorsque $h \to 0$, 
que ce soit par valeurs positives ou n\'egatives;
\item   en revanche, si $T < \Tc$, l'aimantation tend vers une valeur 
strictement positive $m^*(T)$ lorsque $h$ tend vers $0$ par valeurs positives, 
et vers $-m^*(T)$ lorsque $h$ tend vers $0$ par valeurs n\'egatives.
\end{itemize}
La quantit\'e $m^*(T)$ s'appelle l'\defwd{aimantation spontan\'ee} 
du syst\`eme. Elle tend contin\^ument vers $0$ lorsque $T$ tend vers $\Tc$ par
la gauche. 

L'existence de l'aimantation spontan\'ee est importante pour de nombreux 
dispositifs de stockage de donn\'ees (disques durs, m\'emoires flash). Lorsque 
des donn\'ees sont sauvegard\'ees sur un tel dispositif, un champ magn\'etique 
est appliqu\'e localement afin de cr\'eer une aimantation, qui persiste lorsque 
le champ retombe \`a z\'ero. Des donn\'ees sous forme binaire sont ainsi 
repr\'esent\'ees par des domaines d'aimantation diff\'erentes, et cette information 
peut \^etre r\'ecup\'er\'ee par la suite, tant que l'aimant n'est pas port\'e 
\`a une temp\'erature d\'epassant $\Tc$. 

\begin{figure}
\centerline{
\includegraphics*[clip=true,width=70mm]{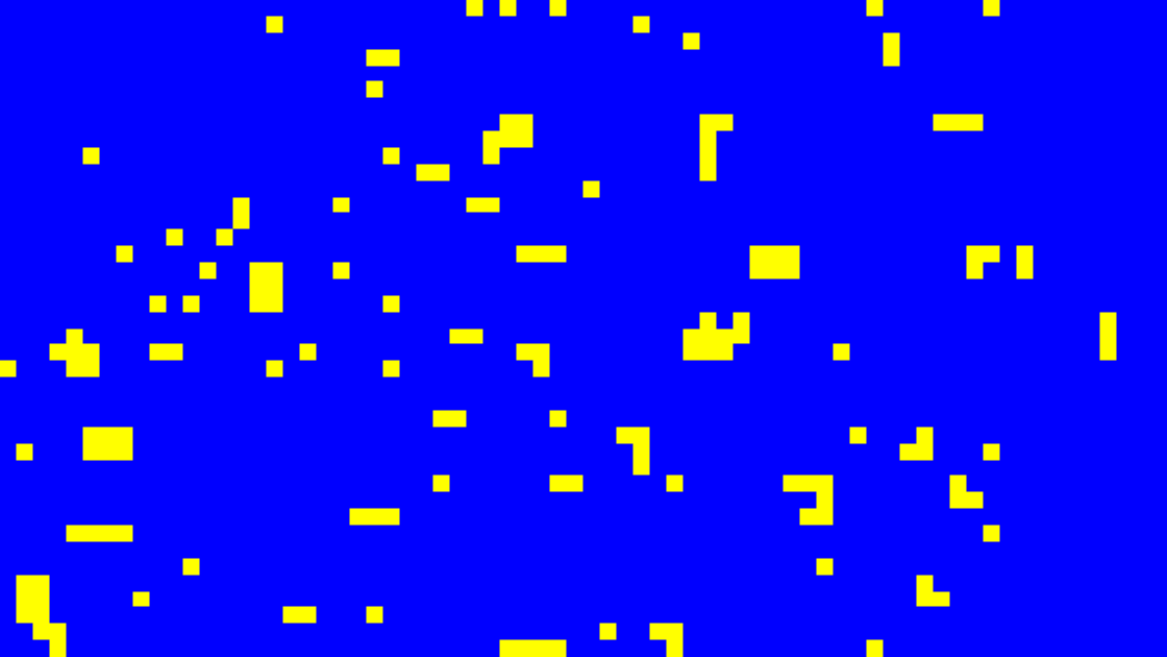}
\hspace{0.1mm}
\includegraphics*[clip=true,width=70mm]{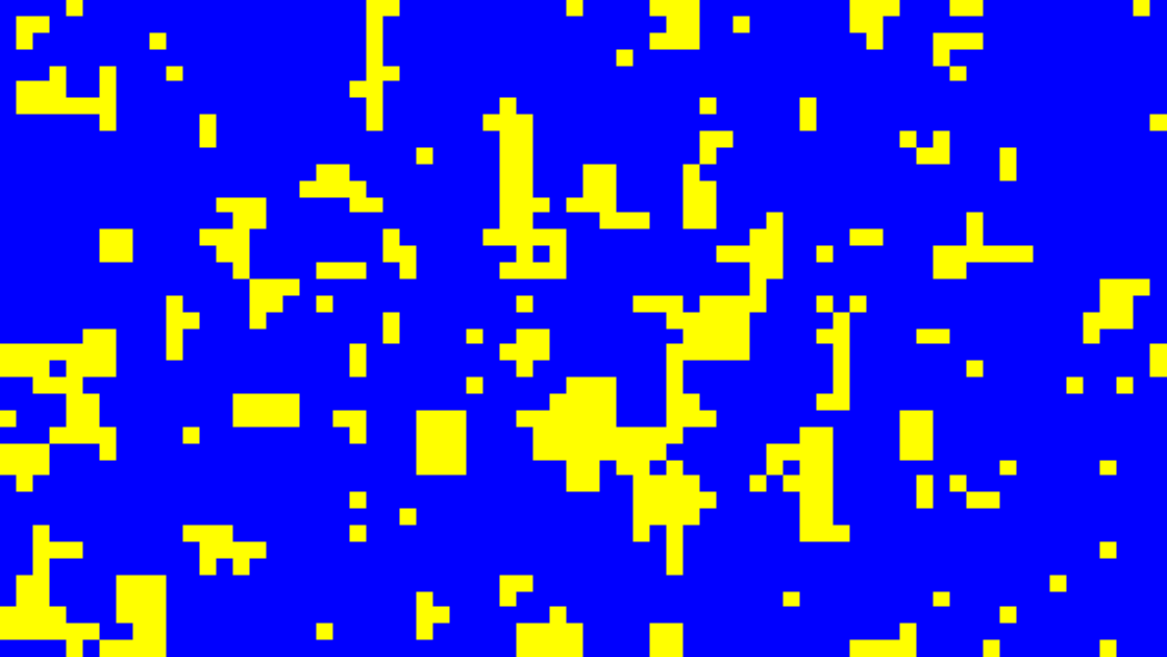}
}
\vspace{2mm}
\centerline{
\includegraphics*[clip=true,width=70mm]{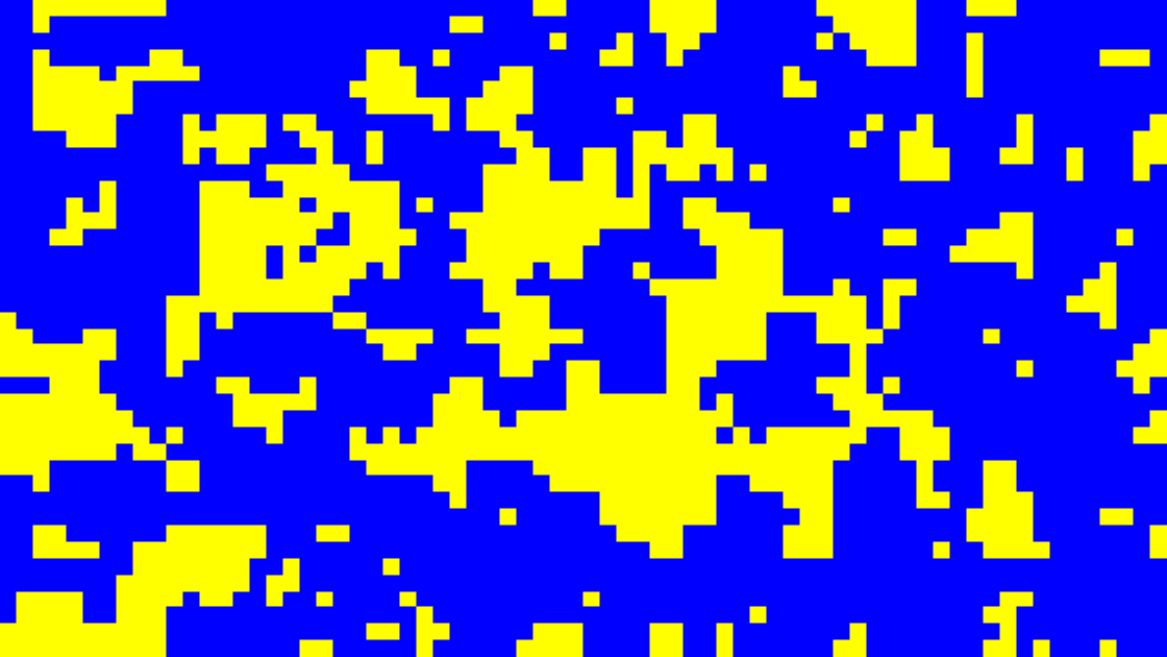}
\hspace{0.1mm}
\includegraphics*[clip=true,width=70mm]{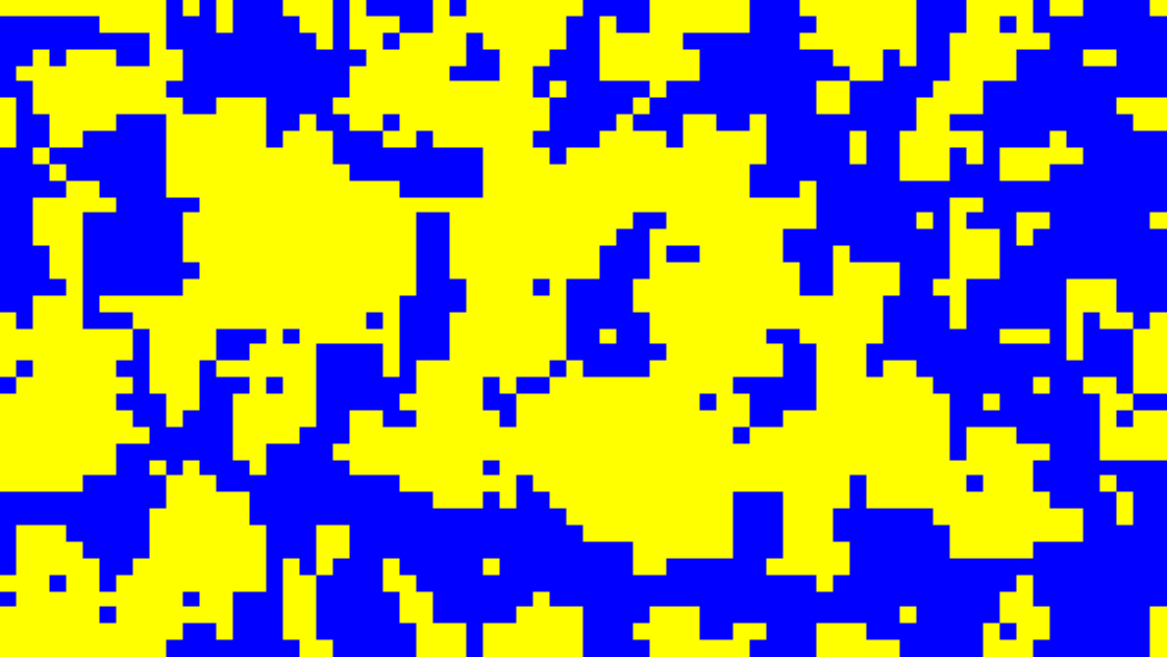}
}
\caption[]{Exemple de simulation d'une dynamique de Glauber. Evolution 
au cours du temps pour $h=1$ et $\beta=0.6$, avec tous les spins initialement
\'egaux \`a $-1$ (bleu). Le champ $h$ positif favorise les spins \'egaux \`a
$+1$ (jaunes).}
 \label{fig_glauber}
\end{figure}

Si l'on veut d\'eterminer num\'eriquement l'aimantation, il suffit en
principe de calculer la somme~\eqref{intro3}. Toutefois, cette somme
comprend $2^N$ termes, ce qui cro\^it tr\`es rapidement avec la taille du
syst\`eme. Par exemple pour un cube de $10\times10\times10$ spins, le
nombre de termes vaut $2^{1000}$, ce qui est de l'ordre de
$10^{300}$. Un ordinateur calculant $10^{10}$ termes par seconde
mettrait beaucoup plus que l'\^age de l'univers \`a calculer la somme. 

Une alternative est d'utiliser un algorithme dit de Metropolis. Au lieu
de parcourir toutes les configurations possibles de $\cX$, on n'en
parcourt qu'un nombre limit\'e, de mani\`ere bien choisie, \`a l'aide
d'une \CM. Pour cela, on part d'une configuration
initiale $x$, puis on transforme cette configuration en retournant un
spin choisi au hasard. Plus pr\'ecis\'ement, on n'op\`ere cette
transition qu'avec une certaine probabilit\'e, qui d\'epend de la
diff\'erence d'\'energie entre les configurations de d\'epart et
d'arriv\'ee. L'id\'ee est que si les probabilit\'es de transition sont
bien choisies, alors la \CM\ va \'echantillonner l'espace de
configuration de telle mani\`ere qu'il suffira de lui faire parcourir une
petite fraction de toutes les configurations possibles pour obtenir une
bonne approximation de l'aimantation $\expec{m}$. Les questions sont alors 
\begin{enumerate}
\item	De quelle mani\`ere choisir ces probabilit\'es de transition~?
\item	Combien de pas faut-il effectuer pour approcher $\expec{m}$ avec
une pr\'ecision donn\'ee~?
\end{enumerate}
R\'epondre \`a ces deux questions est l'un des objectifs principaux 
de ce cours.


\chapter{Rappels sur les cha\^ines de Markov}
\label{chap:cm_rappels} 

Nous rappelons dans ce chapitre quelques notions de base de la th\'eorie des 
\CMs, souvent sans d\'emonstration. La plupart des d\'emonstrations peuvent se 
trouver dans n'im\-por\-te quel bon cours sur les \CMs, comme par exemple~\cite{Durrett1}.


\section{D\'efinitions, notations}
\label{sec:rap_notation} 

Soit $\cX$ un ensemble d\'enombrable, fini ou infini. 

\begin{definition}[Mesure de probabilit\'e, matrice stochastique]
\label{def:matrice_stoch} 
\begin{itemize}
\item  Une mesure de probabilit\'e $\nu$ sur $\cX$ est un ensemble $(\nu(x))_{x\in\cX}$ 
de nombres r\'eels positifs ou nuls satisfaisant 
\begin{equation}
\label{eq:mproba} 
 \sum_{x\in\cX} \nu(x) = 1\;.
\end{equation} 
\item   Une \defwd{matrice stochastique} sur $\cX$ est un ensemble $P = (p_{xy})_{x,y\in\cX}$ 
de nombres r\'eels positifs ou nuls satisfaisant 
\begin{equation}
\label{eq:mstoch} 
 \sum_{y\in\cX} p_{xy} = 1 \qquad \forall x\in\cX\;.
\end{equation} 
\end{itemize}
\end{definition}

Remarquons que puisque les $\nu(x)$ sont positifs ou nuls, la condition~\eqref{eq:mproba}
implique qu'ils sont n\'ecessairement tous dans l'intervalle $[0,1]$. Il en va de m\^eme 
pour les $p_{xy}$. 

\begin{definition}[Cha\^ine de Markov]
On se donne une matrice stochastique $P$ sur $\cX$, et une mesure de probabilit\'e $\nu$
sur $\cX$.
Une \defwd{\CM} (homog\`ene en temps) sur $\cX$, de loi initiale $\nu$ et de matrice de transition $P$, est une suite $(X_n)_{n\geqs0}$ de variables al\'eatoires \`a valeurs dans $\cX$, 
telles que $\prob{X_0 = x} = \nu(x)$ pour tout $x\in\cX$, et satisfaisant 
la \defwd{propri\'et\'e de Markov}
\begin{align}
\pcond{X_n = y}{X_0 = x_0, X_1 = x_1, \dots, X_{n-1} = x_{n-1}}
&= \pcond{X_n = y}{X_{n-1} = x_{n-1}} \\
&= p_{x_{n-1}y}
\end{align}
pour tout $n\geqs1$ et tout choix de $x_0, \dots, x_{n-1}, y\in\cX$. 
\end{definition}

Une cons\'equence imm\'ediate de cette d\'efinition est la suivante. 

\begin{proposition}[Probabilit\'e de trajectoires et loi de $X_n$]
\label{prop:proba_traj} 
Soit $(X_n)_{n\geqs0}$ une \CM\ de loi initiale $\nu$ et de matrice de transition $P$. 
Alors, pour tout $n\geqs0$ et tout choix de $x_0, \dots, x_n\in\cX$, 
\begin{equation}
\label{eq:proba_traj} 
 \prob{X_0 = x_0, X_1 = x_1, \dots, X_n = x_n}
 = \nu(x_0)p_{x_0x_1} \dots p_{x_{n-1}x_n}\;.
\end{equation} 
De plus, pour tout $n\geqs1$ et tout $y\in\cX$, on a 
\begin{equation}
\label{eq:proba_nu_y} 
 \prob{X_n = y} 
 = \sum_{x_0\in\cX} \dots \sum_{x_{n-1}\in\cX}
 \nu(x_0)p_{x_0x_1} \dots p_{x_{n-2}x_{n-1}}p_{x_{n-1}y}\;.
\end{equation} 
\end{proposition}

Dans la suite, les notations suivantes vont s'av\'erer pratiques. 

\begin{itemize}
\item  On \'ecrira $\probin{\nu}{X_n = y}$ au lieu de $\prob{X_n = y}$ 
pour insister sur le fait que la loi initiale est $\nu$. 

\item   De mani\`ere similaire, on \'ecrira $\expecin{\nu}{X_n}$ pour l'esp\'erance 
de $X_n$, partant de la loi $\nu$. 

\item   Soit $\delta_x$ la mesure de probabilit\'e sur $\cX$ donn\'ee par 
\begin{equation}
 \delta_x(y) = 
 \begin{cases}
  1 & \text{si $y = x$\;,}\\
  0 & \text{sinon\;.}
 \end{cases}
\end{equation}
Alors, on \'ecrira souvent $\probin{x}{\cdot}$ et $\expecin{x}{\cdot}$ au lieu de 
$\probin{\delta_x}{\cdot}$ et $\expecin{\delta_x}{\cdot}$.

\item   Il sera pratique de voir les mesures de probabilit\'e sur $\cX$ comme des 
vecteurs ligne. De cette fa\c con, \eqref{eq:proba_nu_y} peut s\'ecrire 
\begin{equation}
 \probin{\nu}{X_n = y} = \bigpar{\nu P^n}_y\;.
\end{equation} 
\end{itemize}

\begin{definition}[\CCM\ r\'eversible]
La \CM\ est dite \defwd{r\'eversible} s'il existe une application 
$\alpha:\cX\to[0,\infty)$, non identiquement nulle, telle que 
\begin{equation}
 \alpha(x) p_{xy} = \alpha(y)p_{yx}
 \qquad 
 \forall x,y\in\cX\;.
\end{equation} 
Dans ce cas, $\alpha = (\alpha_x)_{x\in\cX}$ est appel\'e un \defwd{vecteur r\'eversible}.
\end{definition}

Le nom r\'eversible vient de la propri\'et\'e suivante.

\begin{proposition}[Renversement du temps]
Supposons la \CM\ r\'eversible, pour un vecteur r\'eversible $\alpha$ qui est une 
mesure de probabilit\'e. Alors 
\begin{equation}
 \probin{\alpha}{X_0 = x_0, X_1 = x_1, \dots, X_n = x_n}
 =  \probin{\alpha}{X_0 = x_n, X_1 = x_{n-1}, \dots, X_n = x_0}
\end{equation} 
pour tout $n\in\N$, et tout choix de $x_0, x_1, \dots, x_n\in \cX$. 
\end{proposition}
\begin{proof}
Il suit de~\eqref{eq:proba_traj} que
\begin{align}
 \probin{\alpha}{X_0 = x_0, X_1 = x_1, \dots, X_n = x_n}
 &= \alpha(x_0)p_{x_0 x_1}p_{x_1x_2} \dots p_{x_{n-1}x_n} \\
 &= p_{x_1 x_0}\alpha(x_1)p_{x_1x_2} \dots p_{x_{n-1}x_n} \\
 &= \dots \\
 &= p_{x_1 x_0}p_{x_2x_1} \dots p_{x_nx_{n-1}} \alpha(x_n) \\
 &= \alpha(x_n)p_{x_nx_{n-1}}\dots p_{x_2x_1} p_{x_1 x_0}\;.
\end{align}
ce qui est bien \'egal \`a $\probin{\alpha}{X_0 = x_n, X_1 = x_{n-1}, \dots, X_n = x_0}$.
\end{proof}


\section{Cha\^ines de Markov irr\'eductibles}
\label{sec:rap_irred} 

\begin{definition}[\'Etat accessible, \CM\ irr\'eductible]
\begin{itemize}
\item   On dit qu'un \'etat $y\in\cX$ est \defwd{accessible} depuis $x\in\cX$ s'il existe 
$n\geqs0$ tel que 
\begin{equation}
 \probin{x}{X_n = y} > 0\;.
\end{equation} 
Dans ce cas, on \'ecrira $x \reaches y$.

\item   On dit que les \'etats $x$ et $y$ \defwd{communiquent} et on \'ecrit $x \sim y$,
si on a \`a la fois $x\reaches y$ et $y\reaches x$. 

\item   La \CM\ est \defwd{irr\'eductible} si $x \sim y$ pour tout $x, y\in\cX$. 
\end{itemize}
\end{definition}

On v\'erifie facilement que la relation $\reaches$ est \defwd{r\'eflexive} et
\defwd{transitive}~: on a toujours $x\reaches x$, et si $x\reaches y$ et $y\reaches z$, alors on a $x\reaches z$. 

La relation $\sim$ est r\'eflexive, transitive et \defwd{sym\'etrique}~: si $x \sim y$, 
alors $y \sim x$. C'est donc une \defwd{relation d'\'equivalence}. On a donc une 
partition de $\cX$ en \defwd{classes d'\'equivalence}~:
\begin{equation}
 \cX = \bigsqcup_{k\geqs 0} \cX_k\;,
\end{equation} 
o\`u $\sqcup$ signifie la r\'eunion disjointe, et $x \sim y$ si et seulement si
$x$ et $y$ appartiennent \`a la m\^eme classe. 
En particulier, la \CM\ est irr\'eductible si et seulement si elle admet une unique 
classe d'\'equivalence. 

On peut associer \`a une \CM\ un graphe orient\'e, dont les sommets sont les \'el\'ements de $\cX$, 
et dont les ar\^etes sont les couples $(x,y)$ tels que $p_{xy} > 0$ (avec $y\neq x$). 
Si $\cX$ est fini, une mani\`ere de montrer que la \CM\ est irr\'eductible est d'exhiber
un chemin ferm\'e dans ce graphe, c'est-\`a dire une suite $(x_1, \dots, x_m, x_{m+1} = x_1)$, 
contenant tous les \'elements de $\cX$ au moins une fois, et telle que $p_{x_i x_{i+1}} > 0$ 
pour tout $i\in\set{1,\dots,m}$. 

\begin{example}[Marche al\'eatoire sym\'etrique sur $\Z^d$]
La marche al\'eatoire sym\'etrique sur $\Z^d$ est irr\'eductible. En effet, 
pour tout $x, y\in\Z^d$, il existe un chemin reliant $x$ \`a $y$. 
Ce chemin peut \^etre construit en changeant chaque composante de $x$, 
par \'etapes successives, d'une unit\'e \`a la fois, jusqu'\`a atteindre $y$. 
\end{example}

\begin{remark}[Classes ouvertes et ferm\'ees]
Si la \CM\ n'est pas irr\'eductible, alors une classe $\cX_k$ est \defwd{ferm\'ee} si 
pour tout $x\in \cX_k$ et tout $y\notin\cX_k$, $y$ n'est pas accessible depuis $x$. 
Dans ce cas, la restriction de la \CM\ \`a $\cX_k$ est irr\'eductible. Une classe 
qui n'est pas ferm\'ee est dite \defwd{ouverte}. 
\end{remark}


\section{R\'ecurrence}
\label{sec:rap_rec} 

\begin{definition}[Temps de passage]
Soit $x\in\cX$. Le \defwd{temps de passage} (ou \defwd{temps de premier passage}) de la \CM\ 
en $x$ est la variable al\'eatoire 
\begin{equation}
 \tau_x = \inf\setsuch{n\geqs1}{X_n = x}\;,
\end{equation} 
avec la convention $\tau_x = \infty$ si $X_n \neq x$ pour tout $n\geqs1$. 

Dans le cas particulier o\`u la mesure initiale est $\delta_x$, $\tau_x$ 
s'appelle \'egalement \defwd{temps de retour} en $x$. 
\end{definition}

Dans la suite, on \'ecrira 
\begin{equation}
 \probin{\nu}{\tau_x < \infty} = \lim_{n\to\infty} \probin{x}{\tau_x < n} 
 = 1 - \probin{\nu}{\tau_x = \infty}\;.
\end{equation} 
Attention, par convention la limite lorsque $n\to\infty$ ne comprend \emph{jamais}
le terme $n = \infty$. 

\begin{definition}[R\'ecurrence et transience]
\begin{itemize}
\item   Un \'etat $x\in\cX$ est dit \defwd{r\'ecurrent} si $\probin{x}{\tau_x < \infty} = 1$.
\item   Un \'etat non r\'ecurrent est dit \defwd{transient}.
\item   La \CM\ est dite \defwd{r\'ecurrente} si tous ses \'etats sont r\'ecurrents, 
et \defwd{transiente} si tous ses \'etats sont transients. 
\end{itemize}
\end{definition}

Le crit\`ere suivant permet de ramener la question de la r\'ecurrence d'une \CM\ 
\`a celle d'un petit nombre d'\'etats.

\begin{proposition}[R\'ecurrence et communication]
Si les \'etats $x$ et $y$ communiquent, alors $y$ est r\'ecurrent si et seulement si 
$x$ est r\'ecurrent. Par cons\'equent, 
\begin{itemize}
\item  si un \'etat d'une classe $\cX_k$ est r\'ecurrent (respectivement transient), 
alors tous les \'etats de la classe sont r\'ecurrents (respectivement transients);
on dit alors que la classe est r\'ecurrente (respectivement 
transiente);

\item   si la \CM\ est irr\'eductible, et poss\`ede un \'etant r\'ecurrent 
(respectivement transient), alors la \CM\ est r\'ecurrente (respectivement 
transiente). 
\end{itemize}
\end{proposition}
\begin{proof}[\textit{D\'emonstration partielle}]
Nous allons montrer que si $x$ et $y$ sont dans la m\^eme classe r\'ecurrente, alors 
\begin{equation}
\label{rt8}
\probin{x}{\tau_y<\infty} = \probin{y}{\tau_x<\infty} = 1\;.
\end{equation}
Soit $A_M = \bigcup_{m=1}^M \set{X_m=y}$ l'\'ev\'enement \myquote{la \CM\
visite le site $y$ lors des $M$ premiers pas}. Alors 
\begin{equation}
\label{rt8:1}
\lim_{M\to\infty} \fP^x(A_M) = \sum_{m=1}^\infty \probin{y}{\tau_y=m} =
1\;.
\end{equation}
Soit $n_0$ le plus petit entier tel que $\probin{y}{X_{n_0}=x}>0$. Alors
pour tout $M>n_0$, 
\begin{align}
\nonumber
\fP^y\Bigpar{A_M\cap\set{X_{n_0}=x}}
&= \sum_{n=1}^{M-n_0} \probin{y}{X_{n_0}=x, \tau_y=n_0+n} \\
\nonumber
&= \sum_{n=1}^{M-n_0} \probin{y}{X_{n_0}=x, X_1\neq y, \dots, X_{n_0}\neq y}
\probin{x}{\tau_y=n} \\
&\leqs \probin{y}{X_{n_0}=x} \sum_{n=1}^{M-n_0}\probin{x}{\tau_y=n}\;.
\label{rt8:2}
\end{align}
La premi\`ere \'egalit\'e suit du fait que la \CM\ ne peut pas
retourner en $y$ avant $n_0$ et visiter $x$ au temps $n_0$, par
d\'efinition de $n_0$. Nous faisons maintenant tendre $M$ vers l'infini
des deux c\^ot\'es de l'in\'egalit\'e. Le membre de gauche tend vers
$\probin{y}{X_{n_0}=x}$ en vertu de~\eqref{rt8:1}. Il vient donc 
\begin{equation}
\label{tr8:3}
\probin{y}{X_{n_0}=x} \leqs \probin{y}{X_{n_0}=x}
\probin{x}{\tau_y<\infty}\;.
\end{equation}
Comme $\probin{y}{X_{n_0}=x}\neq 0$ et $\probin{x}{\tau_y<\infty}\leqs 1$,
on a n\'ecessairement $\probin{x}{\tau_y<\infty}=1$.
\end{proof}

Pour montrer qu'un \'etat est r\'ecurrent, le cit\`ere suivant est souvent utile en pratique.

\begin{theorem}[Crit\`ere de r\'ecurrence]
\label{thm:critere_rec} 
Un \'etat $x\in\cX$ est r\'ecurrent si et seulement si 
\begin{equation}
 \sum_{n=0}^\infty \probin{x}{X_n = x} = \infty\;.
\end{equation} 
\end{theorem}

La d\'emonstration de ce r\'esultat est bas\'ee sur la relation suivante. 

\begin{proposition}[\'Equation de renouvellement]
\label{prop_rt1}
Pour tout $x, y\in\cX$ et tout temps $n\in\N$ on a la relation 
\begin{equation}
\label{rt3}
\probin{x}{X_n=y} = \sum_{m=1}^n \probin{x}{\tau_y=m}
\probin{y}{X_{n-m}=y}\;.
\end{equation}
\end{proposition}
\begin{proof}
En d\'ecomposant sur les temps de premier passage en $y$, il vient 
\begin{align}
\nonumber
\probin{x}{X_n=y} 
&= \sum_{m=1}^n \probin{x}{X_1\neq y, \dots, X_{m-1}\neq y,X_m=y,X_n=y} \\
&= \sum_{m=1}^n 
\underbrace{\pcondin{x}{X_n=y}{X_1\neq y, \dots, X_{m-1}\neq y,X_m=y}}_{=\pcondin{x}{X_n=y}{X_m=y}=\probin{y}{X_{n-m}=y}}
\underbrace{\probin{x}{X_1\neq y, \dots, X_{m-1}\neq y,X_m=y}}_{=\probin{x}{\tau_y=m}}\;,
\label{rt3:1}
\end{align}
o\`u nous avons utilis\'e la propri\'et\'e des incr\'ements ind\'ependants.
\end{proof}

\begin{proof}[\textit{D\'emonstration du Th\'eor\`eme~\ref{thm:critere_rec}}]
 \hfill
\begin{itemize}[leftmargin=7mm]
\item[$\Rightarrow$:]
L'\'equation de renouvellement~\eqref{rt3} permet d'\'ecrire 
\begin{align}
\nonumber
S\defby \sum_{n=0}^\infty \probin{x}{X_n=x} 
&= 1 + \sum_{n=1}^\infty \probin{x}{X_n=x} \\
\nonumber
&= 1 + \sum_{n=1}^\infty \sum_{m=1}^n \probin{x}{\tau_x=m}
\probin{x}{X_{n-m}=x} \\
\nonumber
&= 1 + \sum_{m=1}^\infty \probin{x}{\tau_x=m} \sum_{n=m}^\infty
\probin{x}{X_{n-m}=x} \\
&= 1 + \underbrace{\sum_{m=1}^\infty \probin{x}{\tau_x=m}}_{=1}
\sum_{n=0}^\infty \probin{x}{X_n=x} = 1+S\;.
\label{rt4:1}
\end{align}
Comme $S\in[0,\infty]$, l'\'egalit\'e $S=1+S$ implique n\'ecessairement
$S=+\infty$. 

\item[$\Leftarrow$:]
On ne peut pas directement inverser les implications ci-dessus. Cependant,
on peut montrer la contrapos\'ee en d\'efinissant pour tout $0<s<1$ les
s\'eries enti\`eres
\begin{align}
\psi(s) &= \sum_{n=0}^\infty \probin{x}{X_n=x} s^n\;, \\
\phi(s) &= \sum_{n=1}^\infty \probin{x}{\tau_x=n} s^n
= \expecin{x}{s^{\tau_x}}\;.
\label{rt4:2}
\end{align}
Ces s\'eries ont un rayon de convergence sup\'erieur ou \'egal \`a $1$ car
leurs coefficients sont inf\'erieurs ou \'egaux \`a $1$. 
Un calcul analogue au calcul~\eqref{rt4:1} ci-dessus donne alors 
\begin{align}
\psi(s)  
&= 1 + \sum_{m=1}^\infty \probin{x}{\tau_x=m} \sum_{n=m}^\infty
\probin{x}{X_{n-m}=x}s^n \\
&= 1 + \sum_{m=1}^\infty \probin{x}{\tau_x=m}s^m
\sum_{n=0}^\infty \probin{x}{X_n=x}s^{n} 
= 1 + \psi(s)\phi(s)\;,
\label{rt4:3}
\end{align}
d'o\`u
\begin{equation}
\label{rt4:4}
\psi(s) = \frac{1}{1-\phi(s)}\;.
\end{equation}
Par cons\'equent, si $\probin{x}{\tau_i<\infty}=\phi(1)<1$, alors on
obtient, en prenant la limite $s\nearrow1$, 
\begin{equation}
\label{rt4:5}
\sum_{n=0}^\infty \probin{x}{X_n=x} 
= \lim_{s\nearrow1}\psi(s) = \frac{1}{1-\phi(1)} < \infty\;,
\end{equation}
ce qui conclut la d\'emonstration. 
\qed
\end{itemize}
\renewcommand{\qed}{}
\end{proof}


\section{R\'ecurrence positive, probabilit\'e invariante}
\label{sec:rap_rec_pos} 


\begin{definition}[R\'ecurrence positive]
Un \'etat r\'ecurrent $x\in\cX$ est dit \defwd{r\'ecurrent positif} si 
\begin{equation}
 \expecin{x}{\tau_x} < \infty\;.
\end{equation} 
Sinon, l'\'etat est appel\'e \defwd{r\'ecurrent nul}. 
Une \CM\ r\'ecurrente est dite \defwd{r\'ecurrente positive} si tous 
ses \'etats sont r\'ecurrents positifs, et \defwd{r\'ecurrente nulle} 
sinon. 
\end{definition}

La r\'ecurrence positive est \`a nouveau une propri\'et\'e de classe. 

\begin{proposition}[R\'ecurrence positive et communication]
Si les \'etats $x$ et $y$ communiquent, alors $y$ est r\'ecurrent positif si et seulement si 
$x$ est r\'ecurrent positif. En particulier, si la \CM\ est irr\'eductible et admet un 
\'etat r\'ecurrent positif, alors la \CM\ est r\'ecurrente positive.  
\end{proposition}

\begin{remark}[Cas d'un $\cX$ fini]
\label{rem:rec_Xfini} 
Si $\cX$ est fini et la \CM\ est irr\'eductible, alors elle est n\'ecessairement r\'ecurrente positive. En effet, l'irr\'eductibilit\'e montre que pour tout $x\in\cX$, on peut trouver un entier fini $m$ tel que 
\begin{equation}
 p = \max_{y\in\cX} \probin{y}{\tau_x > m} < 1\;.
\end{equation} 
La propri\'et\'e de Markov implique alors que pour tout $k\geqs1$, on a 
\begin{equation}
 \probin{x}{\tau_x > km} \leqs p^k\;.
\end{equation} 
La d\'ecroissance exponentielle des queues de la loi de $\tau_x$ implique 
que $\expecin{x}{\tau_x} < \infty$. 
\end{remark}

Voici un r\'esultat de r\'ecurrence/transience tr\`es classique, qui se d\'emontre 
\`a l'aide du Th\'eo\-r\`eme~\ref{thm:rec_pos}. 

\begin{theorem}[R\'ecurrence/transience de marches al\'eatoires sym\'etriques]
La marche al\'eatoire sym\'etrique sur $\Z^d$ est r\'ecurrente nulle 
si $d\in\set{1,2}$ et transiente si $d\geqs3$. 
\end{theorem}

L'int\'er\^et principal de la d\'efinition de r\'ecurrence positive est li\'e 
\`a l'existence de probabilit\'es invariantes.

\begin{definition}[Mesures et probabilit\'es invariantes]
Une mesure sur $\cX$ (c'est-\`a-dire une application $\mu:\cX\to\R_+=[0,\infty)$) 
est dite \defwd{invariante} si 
\begin{equation}
\label{eq:invariant} 
 \sum_{x\in\cX} \mu(x) p_{xy} = \mu(y) 
 \qquad \forall y\in\cX\;.
\end{equation} 
Si $\mu$ est une mesure de probabilit\'e, on dit que c'est une 
\defwd{probabilit\'e invariante}. On la notera alors souvent $\pi$.
\end{definition}

La relation~\eqref{eq:invariant} s'\'ecrit matriciellement 
\begin{equation}
 \mu P = \mu\;,
\end{equation} 
c'est-\`a-dire que le vecteur ligne $\mu$ est vecteur propre \`a gauche de 
$P$, pour la valeur propre $1$. Si $\pi$ est une probabilit\'e invariante, alors 
\begin{equation}
 \probin{\pi}{X_n = x} = \pi(x) 
 \qquad \forall x\in\cX\;, \forall n\geqs0\;.
\end{equation} 

\begin{example}
Soit $\mu$ une mesure uniforme sur $\Z^d$, c'est-\`a-dire qu'il existe une constante 
$c\in\R$ telle que $\mu(x) = c$ pour tout $x\in\Z^d$. Alors $\mu$ est une mesure 
invariante pour la marche al\'eatoire sym\'etrique sur $\Z^d$. Toutefois, $\mu$ 
n'est pas une mesure de probabilit\'e, car on ne peut pas la normaliser 
(la somme des $\mu(x)$ vaut soit $0$, si $c=0$, soit est infinie, si $c\neq0$). 
\end{example}

\begin{example}
On v\'erifie que la loi binomiale de param\`etres $n$ et $\frac12$ est 
une probabilit\'e invariante du mod\`ele d'Ehrenfest \`a $n$ boules (voir 
Exercice~\ref{exo:Ehrenfest}).
\end{example}

\goodbreak

Le lien entre r\'ecurrence positive et probabilit\'e invariante est 
mis en \'evidence par le r\'esultat suivant.

\begin{theorem}[R\'ecurrence positive et probabilit\'e invariante]
\label{thm:rec_pos_pi} 
Soit $(X_n)_{n\geqs0}$ une \CM\ irr\'eductible sur $\cX$. Alors les conditions 
suivantes sont \'equivalentes~:
\begin{enumerate}
\item   La \CM\ admet une probabilit\'e invariante.
\item   La \CM\ admet un \'etat r\'ecurrent positif. 
\item   Tous les \'etats $x\in\cX$ sont r\'ecurrents positifs.
\end{enumerate}
De plus, si ces propri\'et\'es sont v\'erifi\'ees, alors la probabilit\'e invariante 
est unique, et satisfait 
\begin{equation}
\label{eq:piEtau} 
 \pi(x) = \frac{1}{\expecin{x}{\tau_x}}
 \qquad \forall x\in\cX\;.
\end{equation} 
\end{theorem}

Une mani\`ere de d\'emontrer ce r\'esultat est de fixer un \'etat $z\in\cX$, et de consid\'erer
la mesure $\gamma^{(z)}$, d\'efinie par 
\begin{equation}
\label{eq:gamma(y)} 
 \gamma^{(z)}(x) 
 = \biggexpecin{z}{\sum_{n=1}^{\tau_z} \indicator{X_n = x}}\;,
\end{equation} 
qui mesure le nombre moyen de passages en $x$ entre deux passages en $z$. 
On a alors les propri\'et\'es suivantes. 

\begin{proposition}
\label{prop_stat1}
Supposons la \CM\ irr\'eductible et r\'ecurrente. Alors on a pour tout 
$z\in\cX$~:
\begin{enumerate}
\item	$\smash{\gamma^{(z)}(z)} = 1$;
\item	$\smash{\gamma^{(z)}}$ est une mesure invariante;
\item	Pour tout $x\in\cX$, on a $0<\smash{\gamma^{(z)}(x)}<\infty$; 
\item	$\smash{\gamma^{(y)}}$ est l'unique mesure invariante telle que
$\smash{\gamma^{(z)}(z)} = 1$. 
\end{enumerate}
\end{proposition}
\begin{proof}
 \hfill
\begin{enumerate}
\item	\'Evident, puisque $\tau_z$ est fini presque s\^urement, 
$X_{\tau_z}=z$ et $X_n\neq z$ pour $1\leqs n<\tau_z$.
\item	Nous avons 
\begin{align}
\nonumber
\gamma^{(z)}(x) 
&= \Bigexpecin{z}{\sum_{n=1}^\infty \indexfct{X_n=x,n\leqs\tau_z}} 
= \sum_{n=1}^\infty \probin{z}{X_n=x,n\leqs\tau_z} \\
\nonumber
&= \sum_{y\in\cX} \sum_{n=1}^\infty 
\probin{z}{X_{n-1}=y,n\leqs\tau_z}p_{yx} \\
&= \sum_{y\in\cX} p_{yx} \sum_{m=0}^\infty
\probin{z}{X_m=y,m\leqs\tau_z-1}\;.
\label{stat3:1}
\end{align}
Or la seconde somme dans cette expression peut s'\'ecrire 
\begin{equation}
\label{stat3:2}
\Bigexpecin{z}{\sum_{m=0}^{\tau_z-1} \indexfct{X_m=y}}
= \Bigexpecin{z}{\sum_{m=1}^{\tau_z} \indexfct{X_m=y}}
= \gamma^{(z)}(y)\;,
\end{equation}
vu que $\probin{z}{X_0=y}=\delta_{zy}=\probin{z}{X_{\tau_z}=y}$.
Ceci prouve l'invariance de la mesure $\smash{\gamma^{(z)}}$.

\item	L'invariance de la mesure implique que pour tout $n\geqs0$, 
\begin{equation}
\label{stat3:3}
\gamma^{(z)}(x) = \sum_{y\in\cX}\gamma^{(z)}(y) \probin{y}{X_n=x}\;. 
\end{equation}
En particulier, $1=\gamma^{(z)}(z)\geqs \gamma^{(z)}(y) \probin{y}{X_n=z}$
pour tout $y$. Comme par irr\'eductibilit\'e, il existe un $n$ tel que
$\probin{y}{X_n=z}>0$, on en d\'eduit que $\smash{\gamma^{(z)}(y)}<\infty$
pour tout $y$. D'autre part, on a aussi $\smash{\gamma^{(z)}(x)} \geqs
\probin{z}{X_n=x}$, qui est strictement positif pour au moins un $n$.

\item	Soit $\lambda$ une mesure invariante telle que $\lambda(z)=1$.
Alors pour tout $y$ on a 
\begin{equation}
\label{stat3:4}
\lambda(y) = \sum_{x\neq z} \lambda(x) p_{xy} + p_{zy}
\geqs p_{zy}\;.
\end{equation}
Il vient alors, en minorant $\lambda(x)$ par $p_{zx}$ dans l'expression
ci-dessus, 
\begin{align}
\nonumber
\lambda(y) &\geqs \sum_{x\neq z} p_{zx}p_{xy}
+ p_{zy}\\
&= \probin{z}{X_2=y,\tau_z\geqs 2} + \probin{z}{X_1=y,\tau_z\geqs 1}
\label{stat3:5}
\end{align}
Par r\'ecurrence, on trouve donc pour tout $n\geqs1$ ($a\wedge b$ d\'esigne
le minimum de $a$ et $b$)
\begin{equation}
\lambda(y) \geqs \sum_{m=1}^{n+1} \probin{z}{X_m=y,\tau_z\geqs m}
= \biggexpecin{z}{\sum_{m=1}^{(n+1)\wedge\tau_k}\indexfct{X_m=y}}\;.
\label{stat3:6}
\end{equation}
Lorsque $n$ tend vers l'infini, le membre de droite tend vers
$\smash{\gamma^{(z)}(y)}$. On a donc $\lambda(y)\geqs
\smash{\gamma^{(z)}(y)}$ pour tout $y$. Par cons\'equent,
$\mu=\lambda-\smash{\gamma^{(z)}}$ est une mesure invariante, satisfaisant
$\mu(z)=0$. Comme $\mu(z)=\sum_y\mu(y)\probin{y}{X_n=z}$ pour tout $n$,
l'irr\'eductibilit\'e implique $\mu(y)=0$ $\forall y$, donc
n\'ecessairement $\lambda=\smash{\gamma^{(z)}}$.
\qed
\end{enumerate}
\renewcommand{\qed}{}
\end{proof}

\begin{proof}[\textit{D\'emonstration du Th\'eor\`eme~\ref{thm:rec_pos_pi}}]
 \hfill
\begin{itemize}[leftmargin=14mm]
\item[{$2\Rightarrow 1:$}]	
Si $\mu(z)<\infty$ alors $z$ est r\'ecurrent, donc la \CM, \'etant
irr\'eductible, est r\'ecurrente. Par la proposition pr\'ec\'edente,
$\smash{\gamma^{(z)}}$ est l'unique mesure invariante prenant valeur $1$
en $z$. Or nous avons 
\begin{equation}
\label{stat4:1}
\sum_{y\in\cX}\gamma^{(z)}(y)
= \biggexpecin{z}{\sum_{n=1}^{\tau_z}
\underbrace{\sum_{y\in\cX}\indexfct{X_n=y}}_{=1}}
= \expecin{z}{\tau_z} = \mu(z) < \infty\;.
\end{equation}
Par cons\'equent, la mesure $\pi$ d\'efinie par
$\pi(y)=\gamma^{(z)}(y)/\mu(z)$ est une probabilit\'e invariante.

\item[{$1\Rightarrow 3:$}]
Soit $\pi$ une probabilit\'e invariante, et $z\in\cX$. Alors $\hat\gamma$
d\'efini par $\hat\gamma(y)=\pi(y)/\pi(z)$ est une mesure invariante telle
que $\hat\gamma(z)=1$. Par la proposition pr\'ec\'edente, on a
n\'ecessairement $\hat\gamma=\smash{\gamma^{(z)}}$. Il suit par le m\^eme
calcul que ci-dessus
\begin{equation}
\label{stat4:2}
\expecin{z}{\tau_z} = \sum_{y\in\cX} \hat\gamma(y)
= \frac{1}{\pi(z)}\sum_{y\in\cX}\pi(y) = \frac1{\pi(z)} < \infty\;.
\end{equation}

\item[{$3\Rightarrow 2:$}] \'Evident.
\end{itemize}
Dans ce cas, l'unicit\'e de la mesure suit de celle de $\gamma^{(z)}$, et
la relation~\eqref{eq:piEtau} suit de~\eqref{stat4:2}.
\end{proof}

Dans le cas particulier d'une \CM\ r\'eversible, la probabilit\'e invariante peut \^etre 
d\'eduite imm\'ediatement d'un vecteur r\'eversible.

\begin{proposition}[Probabilit\'es invariante d'une \CM\ r\'eversible]
Soit $(X_n)_{n\geqs0}$ une \CM\ r\'eversible, de vecteur r\'eversible $\alpha$. 
Alors, si 
\begin{equation}
 \cN = \sum_{x\in\cX} \alpha(x) < \infty\;,
\end{equation} 
la \CM\ admet une probabilit\'e invariante, donn\'ee par 
\begin{equation}
 \pi(x) = \frac{1}{\cN} \alpha(x)
 \qquad \forall x\in\cX\;.
\end{equation} 
\end{proposition}
\begin{proof}
Pour tout $x\in\cX$, on a 
\begin{equation}
 \sum_{y\in\cX} \pi(y) p_{yx} 
 = \frac{1}{\cN}\sum_{y\in\cX} \alpha(y) p_{yx}
 = \frac{1}{\cN}\sum_{y\in\cX} p_{xy} \alpha(x) 
 = \frac{1}{\cN} \alpha(x)
 = \pi(x)\;.
\end{equation}
De plus, $\pi$ est bien une mesure de probabilit\'e, puisque la somme des $\pi(x)$ vaut $1$. 
\end{proof}

\begin{figure}
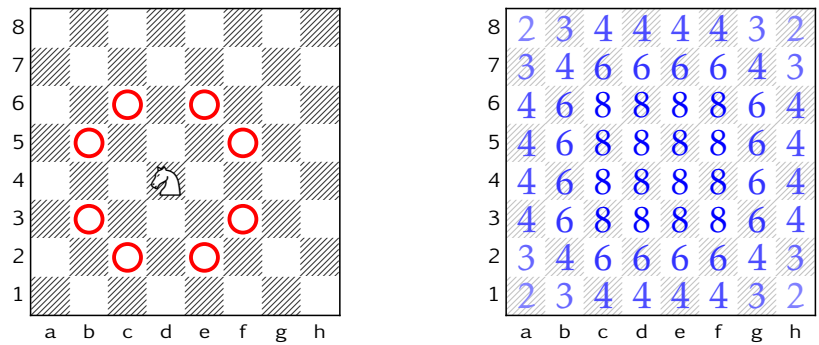

\begin{center}
\vspace{-5mm}
\chessboard[smallboard,
  boardfontsize=14.4pt,
  setwhite={nd4},showmover=false,
  color=red,
  padding=-0.2em,
  pgfstyle=circle,
  markfields={b3,b5,c2,c6,e2,e6,f3,f5}
  ]
\hspace{10mm}
\setchessboard{
blackfieldcolor=black!30,
setfontcolors}
\chessboard[smallboard,
  showmover=false,
  boardfontsize=14.4pt,
  pgfstyle=text,
  color=blue,
  text=$8$\bfseries\sffamily,
  markregion=c3-c3,
  markregion=d3-d3,
  markregion=e3-e3,
  markregion=f3-f3,
  markregion=c4-c4,
  markregion=d4-d4,
  markregion=e4-e4,
  markregion=f4-f4,
  markregion=c5-c5,
  markregion=d5-d5,
  markregion=e5-e5,
  markregion=f5-f5,
  markregion=c6-c6,
  markregion=d6-d6,
  markregion=e6-e6,
  markregion=f6-f6,
  color=blue!80,
  text=$6$\bfseries\sffamily,
  markregion=c2-c2,
  markregion=d2-d2,
  markregion=e2-e2,
  markregion=f2-f2,
  markregion=c7-c7,
  markregion=d7-d7,
  markregion=e7-e7,
  markregion=f7-f7,
  markregion=b3-b3,
  markregion=b4-b4,
  markregion=b5-b5,
  markregion=b6-b6,
  markregion=g3-g3,
  markregion=g4-g4,
  markregion=g5-g5,
  markregion=g6-g6,
  color=blue!70,
  text=$4$\bfseries\sffamily,
  markregion=c1-c1,
  markregion=d1-d1,
  markregion=e1-e1,
  markregion=f1-f1,
  markregion=c8-c8,
  markregion=d8-d8,
  markregion=e8-e8,
  markregion=f8-f8,
  markregion=a3-a3,
  markregion=a4-a4,
  markregion=a5-a5,
  markregion=a6-a6,
  markregion=h3-h3,
  markregion=h4-h4,
  markregion=h5-h5,
  markregion=h6-h6,
  markregion=b2-b2,
  markregion=g2-g2,
  markregion=b7-b7,
  markregion=g7-g7,
  color=blue!60,
  text=$3$\bfseries\sffamily,
  markregion=b1-b1,
  markregion=a2-a2,
  markregion=g1-g1,
  markregion=h2-h2,
  markregion=b8-b8,
  markregion=a7-a7,
  markregion=g8-g8,
  markregion=h7-h7,
  color=blue!50,
  text=$2$\bfseries\sffamily,
  markregion=a1-a1,
  markregion=h1-h1,
  markregion=a8-a8,
  markregion=h8-h8
  ]  
\end{center}
\vspace{-5mm}
 \caption[]{Mouvements permis du cavalier sur l'\'echiquier. Nombre de
 mouvements possibles \`a partir de chaque case.}
 \label{fig_echecs}
\end{figure}

\begin{example}[Le cavalier fou]
Un cavalier se d\'eplace sur un \'echiquier standard (de $64$ cases), en choisissant 
\`a chaque pas l'un des mouvements permis par les r\`egles du jeu des \'echecs, uniform\'ement 
au hasard (\figref{fig_echecs}). La position du cavalier est d\'ecrite par une \CM\ 
sur l'ensemble $\cX$ des $64$ cases de l'\'echiquier. 
Si $\alpha(x)$ d\'esigne le nombre de mouvements permis en partant de la case $x$, alors 
les probabilit\'es de transition sont donn\'ees par 
\begin{equation}
 p_{xy} = 
 \begin{cases}
  \frac{1}{\alpha(x)} & \text{si le mouvement de $x$ vers $y$ est permis\;,}\\
  0 & \text{sinon\;.}
 \end{cases}
\end{equation} 
On v\'erifie que $\alpha$ est un vecteur r\'eversible, et que 
$\cN = \sum_{x\in\cX} \alpha(x) = 336$ 
(voir \figref{fig_echecs}). La \CM\ est donc r\'eversible, et admet la probabilit\'e 
invariante $\pi$ donn\'ee par 
\begin{equation}
 \pi(x) = \frac{\alpha(x)}{336}\;.
\end{equation} 
Le Th\'eor\`eme~\ref{thm:rec_pos_pi} permet alors de calculer le temps de 
r\'ecurrence moyen vers n'importe quel \'etat. Celui-ci vaut 
\begin{equation}
 \expecin{x}{\tau_x} = \frac{1}{\pi(x)} = \frac{336}{\alpha(x)}\;.
\end{equation} 
\end{example}


\section{Ap\'eriodicit\'e, convergence vers la probabilit\'e invariante}
\label{sec:rap_conv} 

\begin{definition}[P\'eriode]
La \defwd{p\'eriode} d'un \'etat $x\in\cX$ est le nombre 
\begin{equation}
 d_x = \pgcd\bigsetsuch{n\geqs1}{\probin{x}{X_n = i} > 0}\;.
\end{equation} 
Si $d_x = 1$, alors on dit que $x$ est \defwd{ap\'eriodique}. 
Si tout $x\in\cX$ est ap\'eriodique, on dit que la \CM\ est ap\'eriodique. 
\end{definition}

La p\'eriode est \`a nouveau un propri\'et\'e de classe. 

\begin{proposition}[P\'eriode et communication]
Si $x \sim y$, alors $d_x = d_y$. Par cons\'equent, si la \CM\ est irr\'eductible 
et admet un \'etat ap\'eriodique, alors la \CM\ est ap\'eriodique. 
\end{proposition}

\begin{example}[Marche al\'eatoire sym\'etrique sur $\Z^d$]
Pour la marche al\'eatoire sym\'etrique sur $\Z^d$, la p\'eriode de l'\'etat $0$ 
vaut $d_0 = 2$. En effet, partant de $0$, la marche ne peut retourner en $0$ 
qu'au temps pairs. Par cons\'equent, la marche n'est pas ap\'eriodique (tous les 
\'etats sont de p\'eriode $2$). 
\end{example}

L'importance de la notion d'ap\'eriodicit\'e vient du r\'esultat crucial suivant.

\begin{theorem}[Convergence vers la probabilit\'e invariante]
\label{thm:convergence_aperiodique} 
Soit $(X_n)_{n\geqs0}$ une \CM\ irr\'eductible, ap\'eriodique et r\'ecurrente 
positive, et soit $\pi$ son unique probabilit\'e invariante. Alors pour toute 
loi initiale $\nu$ et tout $x\in\cX$, on a 
\begin{equation}
 \lim_{n\to\infty} \probin{\nu}{X_n = x} = \pi(x)\;. 
\end{equation} 
\end{theorem}

Nous allons esquisser l'id\'ee principale d'une d\'emonstration de ce th\'eor\`eme, 
due \`a Wolfgang Doeblin. Consid\'erons deux \CMs\ ind\'ependantes, $(X_n)_{n\geqs0}$ 
et $(Y_n)_{n\geqs0}$, ayant les deux la m\^eme matrice de transition $P$, mais 
la premi\`ere partant de $\nu$, alors que la seconde part de $\pi$. Le couple 
$(X_n,Y_n)$ est une \CM\ sur $\cX\times\cX$, de probabilit\'es de transition 
\begin{equation}
 p^\star_{(x,y),(u,v)} = p_{xu}p_{yv}\;, 
\end{equation} 
et de loi initiale $\rho = \nu\otimes\pi$, d\'efinie par 
\begin{equation}
 \rho(x,y) = \nu(x)\pi(y)\;.
\end{equation} 
On montre alors (\`a l'aide du th\'eor\`eme de B\'ezout) que cette \CM\ est 
encore irr\'eductible et ap\'eriodique. Comme elle admet la probabilit\'e 
invariante $\pi\otimes\pi$, elle est aussi r\'ecurrente positive. Soit alors 
\begin{equation}
\label{eq:tau_Delta} 
 \tau_\Delta = \inf\bigsetsuch{n\geqs0}{X_n = Y_n}
\end{equation} 
le temps de passage sur la \defwd{diagonale} 
$\Delta = \setsuch{(x,x)}{x\in\cX}$. On d\'eduit de la r\'ecurrence positive 
que $\tau_\Delta$ est presque s\^urement fini. Introduisons alors le processus 
$(Z_n)_{n\geqs0}$, d\'efini par 
\begin{equation}
 Z_n = 
 \begin{cases}
  X_n & \text{si $n<\tau_\Delta$\;,}\\
  Y_n & \text{si $n\geqs\tau_\Delta$\;.}
 \end{cases}
\end{equation} 
Il suit de l'expression~\eqref{eq:proba_traj} de la probabilit\'e d'une trajectoire 
que $(Z_n)_{n\geqs0}$ est une \CM\ de loi initiale $\nu$ et de matrice de transition $P$. 
Par cons\'equent, $Z_n$ est \'egal en loi \`a $X_n$ pour tout $n\geqs0$. Ceci implique 
que pour tout $n\in\N$ et tout $x\in\cX$, on a 
\begin{equation}
\label{eq:proof_conv_Doeblin} 
 \probin{\rho}{X_n = x,\tau_\Delta \leqs n} 
 = \probin{\rho}{Z_n = x,\tau_\Delta \leqs n} 
 = \probin{\rho}{Y_n = x,\tau_\Delta \leqs n}\;. 
\end{equation} 
La premi\`ere \'egalit\'e suit de l'\'egalit\'e en loi de $X_n$ et $Y_n$, alors que la 
seconde vient du fait que $Z_n = Y_n$ pour $\tau_\Delta \leqs n$. 

On observe maintenant que pour tout $n\in\N$ et tout $x\in\cX$, on a 
\begin{align}
 \probin{\nu}{X_n = x} &= \probin{\rho}{X_n = x, \tau_\Delta \leqs n}
 + \probin{\rho}{X_n = x, \tau_\Delta > n}\;, \\
 \pi(x) = 
 \probin{\pi}{Y_n = x} &= \probin{\rho}{Y_n = x, \tau_\Delta \leqs n}
 + \probin{\rho}{Y_n = x, \tau_\Delta > n}\;.
\end{align}
En prenant la diff\'erence et en utilisant~\eqref{eq:proof_conv_Doeblin}, on obtient 
\begin{equation}
 \bigabs{\probin{\nu}{X_n = x} - \pi(x)} 
 \leqs \bigabs{\probin{\rho}{X_n = x, \tau_\Delta > n} 
 - \probin{\rho}{Y_n = x, \tau_\Delta > n}}
 \leqs 2 \probin{\rho}{\tau_\Delta > n}\;.
\end{equation} 
La \CM\ $(X_n,Y_n)_{n\geqs0}$ \'etant r\'ecurrente positive, cette quantit\'e tend 
vers $0$ lorsque $n$ tend vers l'infini, ce qui prouve le th\'eor\`eme. En 
fait, on a m\^eme obtenu un peu mieux~: pour tout $n\geqs0$, on a 
\begin{equation}
\label{eq:majo_couplage} 
 \sum_{x\in\cX} \bigabs{\probin{\nu}{X_n = x} - \pi(x)}
 \leqs 2 \probin{\rho}{\tau_\Delta > n}\;.
\end{equation} 
Si on arrive \`a majorer la probabilit\'e $\probin{\rho}{\tau_\Delta > n}$, 
on obtient donc une majoration d'une distance entre la loi de $X_n$ et $\pi$ 
(il s'agit d'une distance du type $\ell^1$). C'est un exemple de ce qu'on appelle 
un \defwd{argument de couplage}. 


\section{Exercices}
\label{sec:rap_exo} 

\begin{exercise}
\label{exo:Ehrenfest} 
On consid\`ere le mod\`ele des urnes d'Ehrenfest \`a $N$ boules, c'est-\`a-dire la \CM\ sur l'ensemble $\cX = \set{0,1,\dots N}$ de probabilit\'es de transition 
\[
p_{xy} = 
\begin{cases}
\frac{x}{N} & \text{si $y=x-1$\;,} \\
1-\frac{x}{N} & \text{si $y=x+1$\;,} \\
0 & \text{sinon\;.} 
\end{cases} 
\]

\begin{enumerate}
\item   Montrer que cette cha\^ine de Markov est irr\'eductible. Est-elle ap\'eriodique\,?
 
\item   Montrer que la distribution de probabilit\'e invariante de cette cha\^ine de Markov 
suit une loi bin\^omiale, dont on pr\'ecisera les param\`etres. 
\end{enumerate}
\end{exercise}

\begin{exercise}
Soit $\cG=(V,E)$ un graphe non orient\'e connexe fini. Soit $(X_n)_{n\geqs0}$ la \CM\
 sur $V$ construite en choisissant pour $X_{n+1}$, de mani\`ere
\'equiprobable, l'un des sommets adjacents \`a $X_n$. 
\begin{enumerate}
\item	Montrer que le nombre de voisins de chaque site forme un vecteur
r\'eversible.
\item	En d\'eduire une expression pour la probabilit\'e invariante de la
\CM.
\end{enumerate} 
\end{exercise}

\begin{exercise}
Soit $p\in[0,1]$. 
On consid\`ere la \CM\ suivante sur $\cX=\N$: 
 
\begin{center}
\begin{tikzpicture}[->,>=stealth',shorten >=2pt,shorten <=2pt,auto,node
distance=3.0cm, thick,main node/.style={circle,scale=0.7,minimum size=1.1cm,
fill=blue!20,draw,font=\sffamily\Large}]

  \node[main node] (0) {$0$};
  \node[main node] (1) [right of=0] {$1$};
  \node[main node] (2) [right of=1] {$2$};
  \node[main node] (3) [right of=2] {$3$};
  \node[node distance=2cm] (4) [right of=3] {$\dots$};

  \path[every node/.style={font=\sffamily\small}]
    (0) edge [loop left,left,distance=1.5cm,out=-150,in=150] node {$1-p$} (0)
    (0) edge [bend left,above] node {$p$} (1)
    (1) edge [bend left,above] node {$p$} (2)
    (2) edge [bend left,above] node {$p$} (3)
    (3) edge [bend left,above] node {$p$} (4)
    (1) edge [bend left,below] node {$1-p$} (0)
    (2) edge [bend left,below] node {$1-p$} (1)
    (3) edge [bend left,below] node {$1-p$} (2)
    (4) edge [bend left,below] node {$1-p$} (3)
    ;
\end{tikzpicture}
\end{center}

\begin{enumerate}
\item	Pour quelles valeurs de $p$ la \CM\ est-elle irr\'eductible?

On suppose dans la suite que $p$ est tel que la \CM\ soit irr\'eductible.

\item	La \CM\ est-elle ap\'eriodique?

\item	On suppose que la \CM\ est r\'eversible, et soit $\alpha$ un
vecteur r\'eversible. Ecrire une relation de r\'ecurrence pour les composantes
de $\alpha$, et en d\'eduire $\alpha_n$ en fonction de $\alpha_0$.

\item	Pour quelles valeurs de $p$ la \CM\ admet-elle une probabilit\'e 
invariante $\pi$? D\'eter\-miner $\pi$ pour ces valeurs de $p$.

\item	Pour quelles valeurs de $p$ la \CM\ est-elle r\'ecurrente?
R\'ecurrente positive?

\item	D\'eterminer le temps de r\'ecurrence moyen $\expecin{0}{\tau_0}$. 

\item	Calculer la position moyenne $\expecin{\pi}{X_n}$ pour les valeurs de
$p$ telles que $\pi$ existe. 
\end{enumerate}
\end{exercise}

\begin{exercise}
On consid\`ere une marche al\'eatoire unidimensionnelle sym\'etrique sur 
l'en\-semble $\cX = \set{0,1,\dots,N}$ avec conditions aux bords absorbantes,
c'est-\`a-dire que l'on suppose que $p_{00} = p_{NN} = 1$. Soit 
\[
\tau = \tau_0 \wedge \tau_N
= \inf\bigsetsuch{n\geqs0}{X_n\in\set{0,N}}
\]
le temps d'absorption, et soit 
\[
p(x) = \probin{i}{X_\tau=N}\;.
\]

\begin{enumerate}
\item	D\'eterminer $p(0)$ et $p(N)$. 
\item	Montrer que pour tout $x\in\set{1,\dots,N-1}$, on a 
\[
p(x) = \frac12 \bigbrak{p(x-1)+p(x+1)}\;.
\]

Une fonction $f:\Z\supset A\to\R$ telle que 
$f(x) = \frac12 \brak{f(x-1)+f(x+1)}$ pour tout $x\in A$ est appel\'ee
\emph{harmonique}\/ (discr\`ete). 

\item	Montrer (par l'absurde) le \emph{principe du maximum}: Une fonction
harmonique sur $A$ ne peut atteindre son minimum et son maximum qu'au bord
de $A$ (on pourra supposer $A$ de la forme $A=\set{a,a+1,\dots,b-1,b}$,
dans ce cas son bord est $\partial A=\set{a,b}$). 

\item	Montrer que si $f$ et $g$ sont deux fonctions harmoniques sur $A$,
alors toute combinaison lin\'eaire de $f$ et $g$ est encore harmonique.

\item	Montrer que si $f$ et $g$ sont deux fonctions harmoniques sur $A$,
qui co\"\i ncident sur le bord de $A$, alors elles sont \'egales partout
dans $A$ (consid\'erer $f-g$).

\item	Montrer que toute fonction lin\'eaire $f(x)=cx+h$ est harmonique. 

\item	En utilisant les points 1., 2., 5.~et 6., d\'eterminer la fonction
$p$.
\end{enumerate}
\end{exercise}

\begin{exercise}
On consid\`ere une marche al\'eatoire sym\'etrique sur
$\cX=\set{0,1,\dots,N}$, avec conditions au bord absorbantes,
c'est-\`a-dire que d\`es que la marche atteint l'un des \'etats $0$ ou
$N$, elle y reste ind\'efiniment. Soit 
\[
\tau = \inf\setsuch{n\geqs 0}{X_n\in\set{0,N}}
\]
le temps d'absorption. Par convention, $\tau=0$ si $X_0\in\set{0,N}$. 
Pour $\lambda\in\R$ et $i\in\cX$ on pose 
\[
f(x,\lambda) = \bigexpecin{x}{\e^{-\lambda\tau}\indexfct{X_\tau=N}}
= 
\begin{cases}
\bigexpecin{x}{\e^{-\lambda\tau}} & \text{si $X_\tau=N$\;,} \\
0 & \text{sinon\;.}
\end{cases}
\]

\begin{enumerate}
\item	Que valent $f(0,\lambda)$ et $f(N,\lambda)$?

\item	Montrer que pour tout $x\in\set{1,\dots,N-1}$, 
\[
\probin{x}{\tau=n} = \frac12 
\bigbrak{\probin{x-1}{\tau=n-1} + \probin{x+1}{\tau=n-1}}\;.
\]

\item	Montrer que pour tout $x\in\set{1,\dots,N-1}$, 
\[
f(x,\lambda) = \frac12\e^{-\lambda}
\bigbrak{f(x-1,\lambda) + f(x+1,\lambda)}\;.
\]

\item	Trouver une relation entre $c$ et $\lambda$ telle que l'\'equation
ci-dessus pour $f$ admette des solutions de la forme $f(x,\lambda)=\e^{cx}$. 
Montrer \`a l'aide d'un d\'eveloppement limit\'e que 
\[
c^2 = 2\lambda + \Order{\lambda^2}\;.
\]

\item	D\'eterminer des constantes $a$ et $b$ telles que 
\[
\bigexpecin{x}{\e^{-\lambda\tau}\indexfct{X_\tau=N}}
= a \e^{cx} + b \e^{-cx}\;.
\]

\item	Effectuer un d\'eveloppement limit\'e  au
premier ordre en $\lambda$ de l'\'egalit\'e ci-dessus.
En d\'eduire 
\[
\probin{x}{X_\tau=N}\;.
\]

\item   Calculer
\[
\bigexpecin{x}{\tau \indexfct{X_\tau=N}}\;.
\]

\item	Sans faire les calculs, indiquer comment proc\'eder pour d\'eterminer
la variance de la variable al\'eatoire $\tau \indexfct{X_\tau=N}$ et l'esp\'erance et la variance de $\tau$. 
\end{enumerate}

On rappelle les d\'eveloppements limit\'es suivants:
\begin{align}
\cosh(x) &= \frac{\e^x+\e^{-x}}{2}
= 1 + \frac{1}{2!}x^2 + \Order{x^4}\;, \\
\sinh(x) &= \frac{\e^x-\e^{-x}}{2}
= x + \frac{1}{3!}x^3 + \Order{x^5}\;.
\end{align} 
\end{exercise}


\chapter{Th\'eorie spectrale et vitesse de convergence}
\label{chap:cm_spectrale} 

Dans ce chapitre et le suivant, nous allons consid\'erer des \CMs\ $(X_n)_{n\geqs0}$ 
irr\'eductibles, r\'ecurrentes positives et ap\'eriodiques sur un ensemble 
d\'enombrable $\cX$. Soit $f:\cX\to\R$ une fonction born\'ee, et soit $\pi$ 
la probabilit\'e invariante de la \CM. Le but est d'estimer la quantit\'e 
\begin{equation}
 \expecin{\pi}{f} = \sum_{x\in\cX} \pi(x) f(x)\;.
\end{equation} 
Nous savons par le Th\'eor\`eme~\ref{thm:convergence_aperiodique} que l'on a 
\begin{equation}
 \expecin{\pi}{f} 
 = \lim_{n\to\infty} \sum_{x\in\cX} \probin{\nu}{X_n = x}f(x)
 = \lim_{n\to\infty} \expecin{\nu}{f(X_n)}\;,
\end{equation} 
pour toute loi initiale $\nu$. Notre but est maintenant de majorer l'erreur 
\begin{equation}
\label{eq:erreur_expecf} 
 \bigabs{\expecin{\nu}{f(X_n)} - \expecin{\pi}{f}}\;.
\end{equation} 
Une premi\`ere mani\`ere de le faire est la suivante.

\begin{lemma}[Couplage et vitesse de convergence]
Si la \CM\ est ap\'eriodique, alors 
\begin{equation}
 \bigabs{\expecin{\nu}{f(X_n)} - \expecin{\pi}{f}}
 \leqs 2 \probin{\nu\otimes\pi}{\tau_\Delta > n} \sup_{x\in\cX} \abs{f(x)}\;,
\end{equation} 
o\`u $\tau_\Delta$ est d\'efini dans~\eqref{eq:tau_Delta}. 
\end{lemma}
\begin{proof}
On a 
\begin{equation}
 \expecin{\nu}{f(X_n)} - \expecin{\pi}{f} 
 = \sum_{x\in\cX} \bigbrak{\probin{\nu}{X_n = x} - \pi(x)} f(x)\;.
\end{equation} 
Le r\'esultat suit donc de~\eqref{eq:majo_couplage}. 
\end{proof}

Si l'on arrive \`a contr\^oler $\probin{\rho}{\tau_\Delta > n}$, on obtient donc 
la majoration souhait\'ee. Toutefois, cela n'est pas toujours possible, et on doit 
alors avoir recours \`a d'autres approches. 
Dans ce chapitre, nous allons discuter comment l'erreur~\eqref{eq:erreur_expecf} 
d\'epend de quantit\'es li\'ees aux valeurs propres et vecteurs propres de la matrice de transition $P$. Une autre approche, plus robuste, bas\'ee sur les fonctions de Lyapounov, sera 
discut\'ee dans le chapitre suivant. 


\section{Quelques exemples simples}
\label{sec:spec_exemples} 

\begin{example}
Consid\'erons la matrice stochastique 
\begin{equation}
 P = 
 \begin{pmatrix}
  0 & 1 \\ 1 & 0
 \end{pmatrix}\;.
\end{equation} 
La \CM\ sur $\cX=\set{1,2}$ associ\'ee est irr\'eductible, r\'ecurrente positive, 
mais pas ap\'eriodique~: sa p\'eriode est \'egale \`a $2$. En fait, on a 
\begin{equation}
 P^n = 
 \begin{cases}
  P & \text{si $n$ est impair\;,}\\
  \one & \text{si $n$ est pair\;,}
 \end{cases}
\end{equation} 
o\`u $\one$ d\'enote la matrice identit\'e. Par cons\'equent,
\begin{equation}
 \expecin{\nu}{f(X_n)} 
 = \nu P^n f = 
 \begin{cases}
  \nu(1)f(2) + \nu(2)f(1) & \text{si $n$ est impair\;,}\\
  \nu(1)f(1) + \nu(2)f(2) & \text{si $n$ est pair\;,}
 \end{cases}
\end{equation} 
D'un autre c\^ot\'e, la \CM\ \'etant r\'ecurrente positive, elle admet une 
unique probabilit\'e invariante $\pi$, satisfaisant $\pi P = \pi$. On trouve facilement 
que $\pi = (\frac12, \frac12)$, ce qui implique 
\begin{equation}
 \expecin{\pi}{f(X_n)} 
 = \pi f = \frac12 (f(1) + f(2))\;. 
\end{equation} 
On s'aper\c coit que si $\nu \neq \pi$, alors $\expecin{\nu}{f}$ ne converge pas vers 
$\expecin{\pi}{f}$, sauf dans le cas parti\-culier $f(1) = f(2)$. 

Les valeurs propres de $P$ sont $1$ et $-1$. Des vecteurs propres \`a gauche associ\'es 
sont $\pi$ et $(1, -1)$. La valeur propre $-1$ est associ\'ee au fait que la \CM\ est 
$2$-p\'eriodique. 
\end{example}

\begin{example}
 On peut facilement g\'en\'eraliser cet exemple \`a des p\'eriodes sup\'erieures. Par exemple, 
la matrice stochastique 
\begin{equation}
 P = 
 \begin{pmatrix}
  0 & 1 & 0 \\
  0 & 0 & 1 \\
  1 & 0 & 0
 \end{pmatrix}
\end{equation}
satisfait $P^3 = \one$. Ses valeurs propres sont les trois racines cubiques de $1$, \`a 
savoir $1$ et $\e^{\pm\icx 2\pi/3}$.
La \CM\ associ\'ee est irr\'eductible, r\'ecurrente positive, et de p\'eriode $3$.
Elle admet l'unique probabilit\'e invariante $\pi = (\frac13,\frac13,\frac13)$. 
\`A nouveau, si $\nu \neq \pi$, alors $\expecin{\nu}{f}$ ne converge pas vers 
$\expecin{\pi}{f}$, sauf dans le cas particulier o\`u $f$ est constante.  
\end{example}

\begin{example}
Par contraste, consid\'erons la matrice stochastique 
\begin{equation}
 P = 
 \begin{pmatrix}
  \frac13 & \frac23 \\[3pt]
  \frac23 & \frac13
 \end{pmatrix}\;.
\end{equation} 
La \CM\ associ\'ee est irr\'eductible, r\'ecurrente positive, et ap\'eriodique 
(car, par exemple, on a $\probin{1}{X_1 = 1} = \frac13 > 0$ et 
$\probin{1}{X_2 = 1} = \frac59 > 0$). Les valeurs propres de $P$ sont 
$\lambda_0 = 1$ et $\lambda_1 = -\frac13$. Une mani\`ere de calculer $P^n$ 
est d'utiliser la \defwd{d\'ecomposition de Dunford} (que nous rappellerons \`a 
la section~\ref{sec:spec_Dunford})
\begin{equation}
 P = \lambda_0 \Pi_0 + \lambda_1 \Pi_1\;, 
 \qquad \Pi_0 = 
 \begin{pmatrix}
  \frac12 & \frac12 \\[3pt]
  \frac12 & \frac12
 \end{pmatrix}\;, 
 \qquad \Pi_0 = 
 \begin{pmatrix}
  \frac12 & -\frac12 \\[3pt]
  -\frac12 & \frac12
 \end{pmatrix}\;.
\end{equation} 
Les matrices $\Pi_0$ et $\Pi_1$ sont des \defwd{projecteurs}~: elles satisfont 
$\Pi_0^2 = \Pi_0$, et $\Pi_1^2 = \Pi_1$. Elles sont obtenues chacune en multipliant 
un vecteur propre \`a droite et un vecteur propre \`a gauche de $P$, proprement 
normalis\'es. De plus, on v\'erifie que $\Pi_0\Pi_1 = \Pi_1\Pi_0 = 0$. Ceci implique, 
par la formule du bin\^ome de Newton, que 
\begin{equation}
 P^n = \lambda_0^n \Pi_0 + \lambda_1^n \Pi_1 
 = \Pi_0 + \biggpar{-\frac13}^n \Pi_1\;.
\end{equation} 
Par cons\'equent, nous avons 
\begin{equation}
 \nu P^n f 
 = \frac12 \bigpar{f(1) + f(2)} 
 + \frac12 \biggpar{-\frac13}^n \bigpar{\nu(1) - \nu(2)} \bigpar{f(1) - f(2)}\;. 
\end{equation} 
Comme par ailleurs, $\pi = (\frac12, \frac12)$, on a 
\begin{equation}
 \expecin{\pi}{f}
 = \pi f 
 = \frac12 \bigpar{f(1) + f(2)}\;. 
\end{equation} 
Par cons\'equent, $\expecin{\nu}{f(X_n)}$ converge exponentiellement vite 
vers $\expecin{\pi}{f}$, avec une diff\'erence d'ordre $3^{-n}$. 
\end{example}

Ces exemples sugg\`erent que 
\begin{itemize}
\item   si la \CM\ est p\'eriodique, alors $P$ admet plusieurs 
valeurs propres diff\'erentes de module $1$, $P^n$ ne converge pas lorsque 
$n\to\infty$, et $\expecin{\nu}{f(X_n)}$ ne converge pas vers $\expecin{\pi}{f}$ 
si $\nu\neq\pi$, sauf pour des $f$ tr\`es particuliers;
\item  si la \CM\ est ap\'eriodique, alors $P$ admet $1$ comme valeur propre 
simple, toutes les autres valeurs propres de $P$ sont strictement inf\'erieures 
\`a $1$ en module, et $\expecin{\nu}{f(X_n)}$ converge vers $\expecin{\pi}{f}$ 
si $\nu\neq\pi$.
\end{itemize}
Nous allons voir dans les sections suivantes que ceci est effectivement le cas.


\section{Normes de vecteurs et de matrices}
\label{sec:spec_norm} 

Soit $P$ la matrice de transition d'une \CM\ irr\'eductible et r\'ecurrente 
positive. 
Nous savons que $P$ admet la valeur propre $\lambda_0 = 1$. Un vecteur propre \`a 
gauche associ\'e est $\pi$, alors qu'un vecteur propre \`a droite est le vecteur 
\begin{equation}
 \vone = 
 \begin{pmatrix}
  1 \\ 1 \\ \vdots \\ 1
 \end{pmatrix}\;.
\end{equation} 
En effet, la propri\'et\'e~\eqref{eq:mstoch} d'une matrice stochastique 
\'equivaut \`a $P\vone = \vone$. 

Dans la suite, il sera naturel de travailler avec les normes suivantes.

\begin{definition}[Normes de vecteurs]
La \defwd{norme $\ell^1$} d'un vecteur ligne $\mu$ est d\'efinie par 
\begin{equation}
 \norm{\mu}_1 = \sum_{x\in\cX} \abs{\mu(x)}\;.
\end{equation} 
La \defwd{norme $\ell^\infty$} (ou \defwd{norme sup}) d'un vecteur colonne est d\'efinie par 
\begin{equation}
 \norm{v}_\infty = \sup_{x\in\cX} \abs{v(x)}\;.
\end{equation} 
\end{definition}

Dans la suite, nous utiliserons souvent la majoration \'el\'ementaire 
\begin{equation}
\label{eq:l1_linfty} 
 \bigabs{\mu v} 
 = \biggabs{\sum_{x\in\cX} \mu(x)v(x)} 
 \leqs \sum_{x\in\cX} \abs{\mu(x)v(x)} 
 \leqs \norm{\mu}_1 \norm{v}_\infty\;.
\end{equation} 

\begin{lemma}[Normes et matrice stochastique]
Pour une matrice stochastique $P$, et tout vecteur ligne $\mu$ et vecteur colonne 
$v$ de dimension ad\'equate, on a 
\begin{equation}
 \norm{Pv}_\infty \leqs \norm{v}_\infty 
 \qquad\text{et}\qquad 
 \norm{\mu P}_1 \leqs \norm{\mu}_1\;.
\end{equation} 
De plus, il existe des vecteurs $\mu$ et $v$ non nuls tels que 
$\norm{Pv}_\infty = \norm{v}_\infty$ et $\norm{\mu P}_1 = \norm{\mu}_1$.
\end{lemma}
\begin{proof}
On a 
\begin{equation}
 \norm{Pv}_\infty
 = \sup_{x\in\cX} \biggabs{\sum_{y\in\cX}p_{xy}v(y)}
 \leqs \sup_{x\in\cX} \biggbrak{\norm{v}_\infty \sum_{y\in\cX}p_{xy}}
 = \norm{v}_\infty\;, 
\end{equation} 
et
\begin{equation}
 \norm{\mu P}_1 
 = \sum_{x\in\cX} \biggabs{\sum_{y\in\cX} \mu(y) p_{yx}}
 \leqs \sum_{y\in\cX} \abs{\mu(y)} \sum_{x\in\cX} p_{yx} 
 = \norm{\mu}_1\;.
\end{equation} 
Pour avoir \'egalit\'e, il suffit de prendre $v=\vone$ et $\mu=\pi$. 
\end{proof}

\begin{remark}[Norme subordonn\'ee]
\label{rem:norme_subordonnee} 
On peut associer \`a $P$ une \defwd{norme subordonn\'ee} $\norm{P}$, 
correspondant \`a la norme $\norm{\cdot}_1$ pour la multiplication \`a gauche 
et \`a la norme $\norm{\cdot}_\infty$ pour la multiplication \`a droite, 
satisfaisant 
\begin{equation}
\norm{P}
:= \sup_{v\neq0} \frac{\norm{Pv}_\infty}{\norm{v}_\infty}
= \sup_{\mu\neq0} \frac{\norm{\mu P}_1}{\norm{\mu}_1}
= 1\;.
\end{equation}
\end{remark}

\begin{corollary}[Module des valeurs propres]
Toute valeur propre $\lambda$ d'une matrice stochastique $P$ satisfait 
$\abs{\lambda} \leqs 1$. 
\end{corollary}
\begin{proof}
Soit $\lambda$ une valeur propre de $P$, et $v$ un vecteur propre \`a droite 
associ\'e. Alors 
\begin{equation}
 \abs{\lambda}\norm{v}_\infty 
 = \norm{\lambda v}_\infty 
 = \norm{Pv}_\infty 
 \leqs \norm{v}_\infty\;,
\end{equation} 
d'o\`u le r\'esultat, car on peut diviser des deux c\^ot\'es par $\norm{v}_\infty > 0$. 
\end{proof}


\section{Th\'eor\`eme de Perron--Frobenius et trou spectral}
\label{sec:spec_perron-Frobenius}

Le r\'esultat suivant est un cas particulier du th\'eor\`eme de Perron--Frobenius 
(ce th\'eor\`eme est plus g\'en\'eral, car il admet des versions s'appliquant \`a 
des matrices non stochastiques, \`a condition que tous leurs \'el\'ements soient 
r\'eels non n\'egatifs). 

\begin{theorem}[Perron--Frobenius]
Soit $P$ une matrice stochastique irr\'eductible. Alors 
\begin{itemize}
\item   $P$ admet $\lambda_0 = 1$ comme valeur propre \defwd{simple} (de multiplicit\'e alg\'ebrique $1$);
\item   si $P$ est ap\'eriodique, alors toutes ses valeurs propres autres que 
$\lambda_0$ sont de module strictement inf\'erieur \`a $1$;
\item   si $P$ est p\'eriodique, de p\'eriode $p$, alors elle admet exactement 
$p$ valeurs propres de module $1$, qui sont des racines $p$i\`emes de $1$.
\end{itemize}
\end{theorem}

Nous admettrons ce r\'esultat. Voici toutefois quelques indications sur sa d\'emonstration.

\begin{itemize}
\item  Si la valeur propre $\lambda_0 = 1$ n\'etait pas de multiplicit\'e $1$, on 
pourrait trouver au moins deux vecteurs lignes $\pi$ et $\mu$, lin\'eairement
ind\'ependants, tels que $\pi P = \pi$ et $\mu P = \mu$ (dans le cas diagonalisable, 
sinon l'argument est un peu plus compliqu\'e). 
Le vecteur $\mu$ n'est pas n\'ecessairement une 
mesure de probabilit\'e. Mais on peut trouver $\theta\in[0,1]$ tel que la 
combinaison convexe 
\begin{equation}
 \nu = \theta \mu + (1-\theta)\pi
\end{equation} 
soit une mesure de probabilit\'e. Dans le cas diagonalisable, on trouve 
\begin{equation}
 \nu P^n = \nu
 \qquad \forall n\geqs 0\;.
\end{equation} 
Mais ceci contredit l'unicit\'e de la probabilit\'e invariante. 

\item   Si $P$ est ap\'eriodique, supposons par l'absurde que $P$ admet une valeur 
propre $\lambda$ de module $1$, diff\'erente de $1$. Si $\lambda$ est r\'eelle, 
pour un vecteur propre \`a gauche $\mu$, 
on peut proc\'eder comme au point pr\'ec\'edent, pour construire une mesure 
de probabilit\'e $\nu$ satisfaisant 
\begin{equation}
 \nu P^n = \theta \lambda^n \mu + (1-\theta)\pi\;. 
\end{equation} 
Mais alors $\nu P^n$ ne converge pas vers $\pi$ lorsque $n$ tend vers l'infini, 
ce qui contredit le Th\'eor\`eme~\ref{thm:convergence_aperiodique}. 
Si $\lambda$ est complexe, alors $\bar\lambda$ est \'egalement valeur propre, 
de vecteur propre $\bar\mu$, et on peut appliquer un argument analogue 
avec le vecteur r\'eel $\mu + \bar\mu$.

\item   Si $P$ est p\'eriodique de p\'eriode $P$, l'id\'ee de base est que 
$P^p$ admet $p$ sous-espaces invariants suppl\'ementaires. La restriction de 
$P$ \`a chacun de ces sous-espaces doit admettre la valeur propre $1$, ce qui 
correspond \`a une valeur propre racine $p$i\`eme de l'unit\'e de $P$.
\end{itemize}

Concentrons-nous maintenant sur le cas o\`u $P$ est ap\'eriodique.

\begin{lemma}[Limite de $P^n$]
Si $P$ est ap\'eriodique, alors 
\begin{equation}
\label{eq:convergence_Pn} 
 \lim_{n\to\infty} P^n = \Pi_0
 = \vone \pi\;.
\end{equation} 
La matrice $\Pi_0$ est un \defwd{projecteur}, c'est-\`a-dire qu'elle satisfait 
$\Pi_0^2 = \Pi_0$. 
\end{lemma}
\begin{proof}
Le th\'eor\`eme~\ref{thm:convergence_aperiodique} implique que $\nu P^n$ converge 
vers $\pi$ pour toute loi initiale $\nu$. La relation~\eqref{eq:convergence_Pn} 
s'obtient en appliquant ceci \`a $\delta_x$ pour tout $x\in\cX$. 
La relation $\Pi_0^2 = \Pi_0$ suit du fait que $\pi\vone = 1$, en vertu de~\eqref{eq:mproba}.
\end{proof}

\begin{remark}
La matrice $\Pi_0$ est une matrice dont toutes les lignes sont \'egales. 
En particulier, si $\cX$ est fini, de cardinal $N$, alors 
\begin{equation}
 \Pi_0 = 
\begin{pmatrix}
 \pi(1) & \dots & \pi(N) \\
 \vdots & & \vdots \\
 \pi(1) & \dots & \pi(N) 
\end{pmatrix}\;.
\end{equation} 
\end{remark}

\begin{definition}[Rayon spectral et trou spectral]
Soit $P$ une matrice stochastique irr\'eductible et ap\'eriodique, 
et soit $P_\perp = P - \Pi_0$. Alors le \defwd{rayon spectral} de 
$P_\perp$ est 
\begin{align}
 \rho 
 &= \sup\Bigsetsuch{\abs{\lambda_j}}{\text{$\lambda_j$ 
 est valeur propre de $P_\perp$}} \\
 &= \sup\Bigsetsuch{\abs{\lambda_j}}{\text{$\lambda_j$ 
 est valeur propre de $P$}, \lambda \neq 1}\;.
\end{align} 
Le \defwd{trou spectral} de $P$ est par d\'efinition $1 - \rho$.
\end{definition}

Le th\'eor\`eme de Perron--Frobenius implique que $0 \leqs \rho < 1$, donc 
que $1-\rho > 0$. L'int\'er\^et de cette d\'efinition est li\'e \`a l'observation 
suivante.

\begin{proposition}[Vitesse de convergence et trou spectral]
On a 
\begin{equation}
 \expecin{\nu}{f(X_n)} - \expecin{\pi}{f} 
 = (\nu - \pi)P_\perp^n f\;.
\end{equation} 
\end{proposition}
\begin{proof}
On a une d\'ecomposition de l'espace des mesures 
en deux sous-espace suppl\'ementaires, invariants par $P$, l'un associ\'e \`a 
$\Pi_0$, et l'autre associ\'e \`a $P_\perp$. Le premier est simplement le 
sous-espace vectoriel de dimension $1$ engendr\'e par $\pi$, alors que le 
second est 
\begin{equation}
 \vone_\perp = 
 \Bigsetsuch{\mu:\cX\to\R}{\mu \vone = 0} = 
 \biggsetsuch{\mu:\cX\to\R}{\sum_{x\in\cX}\mu(x) = 0}\;.
\end{equation} 
En effet, si $\mu\in\vone_\perp$, alors 
\begin{equation}
 \mu P \vone = \mu \vone = 0\;,
\end{equation} 
ce qui implique que $\mu P\in\vone_\perp$, ou encore $\vone_\perp P \subset \vone_\perp$. 
De plus, on a 
\begin{align}
\mu P_\perp &= \mu P - \mu\Pi_0 = \mu P \\
 \pi P_\perp &= \pi P - \pi\Pi_0 
 = \pi - \pi\vone \pi 
 = 0 
\label{eq:invarianceP} 
\end{align}
puisque $\mu\Pi_0 = \mu\vone\pi = 0$ et $\pi\vone = 1$. 

D\'ecomposons alors $\nu$ en $\nu = \pi + \mu$. On a $\mu\in\vone_\perp$,
puisque $\mu\vone = \nu\vone - \pi\vone = 1 - 1 = 0$. Il suit de~\eqref{eq:invarianceP}
que pour tout $n\geqs0$, 
\begin{equation}
 \nu P^n = (\pi + \mu)P^n = \pi + \mu P_\perp^n\;.
\end{equation} 
Par cons\'equent,
\begin{equation}
 \expecin{\nu}{f(X_n)} = \nu P^n f = \pi f + \mu P_\perp^n f\;,
\end{equation} 
d'o\`u le r\'esultat. 
\end{proof}

Par la majoration~\eqref{eq:l1_linfty}, on a 
\begin{equation}
\label{eq:decroissance_EfXn} 
 \bigabs{\expecin{\nu}{f(X_n)} - \expecin{\pi}{f}}
 \leqs \norm{\nu-\pi}_1 \norm{P_\perp^n f}\infty\;.
\end{equation} 
On s'attend \`a avoir 
\begin{equation}
\label{eq:borne_Pperp} 
 \norm{P_\perp^n f}_\infty \leqs C\rho^n\norm{f}_\infty
\end{equation} 
pour une constante $C$ \`a d\'eterminer. Si c'est bien le cas, alors on aura 
montr\'e que $\expecin{\nu}{f(X_n)}$ converge exponentiellement vite 
vers $\expecin{\pi}{f}$, avec une erreur qui d\'ecro\^it comme $\rho^n$. 


\section{Diagonalisation et d\'ecomposition de Dunford}
\label{sec:spec_Dunford} 

Notre objectif est maintenant de v\'erifier~\eqref{eq:borne_Pperp}. 
Nous supposons pour l'instant que $\cX$ est fini, de cardinal $N$. 
Consid\'erons d'abord le cas o\`u $P_\perp$ est diagonalisable. Alors il existe une matrice 
non singuli\`ere $S$ telle que 
\begin{equation}
 S^{-1}P_\perp S = \Lambda_\perp = 
 \begin{pmatrix}
  0 & 0 & \dots & \dots & 0 \\
  0 & \lambda_1 & & & \vdots \\
  \vdots & & \ddots & & \vdots \\
  \vdots & & & \lambda_{N-2} & 0 \\
  0 & \dots & \dots & 0 & \lambda_{N-1} 
 \end{pmatrix}\;.
\end{equation} 
En effet, la premi\`ere valeur propre de $P_\perp$ est nulle, puisque $\pi P_\perp = 0$, 
cf.~\eqref{eq:invarianceP}. 
On a alors $P_\perp = S\Lambda_\perp S^{-1}$, et 
\begin{equation}
 P_\perp^n = S\Lambda_\perp^n S^{-1}
 \qquad \forall n\geqs 0\;.
\end{equation} 
On remarque que $\norm{\Lambda_\perp^n g}_\infty \leqs \rho^n \norm{g}_\infty$ 
par d\'efinition du rayon spectral, et que par cons\'equent
\begin{equation}
 \norm{P_\perp^n}_\infty 
 \leqs \norm{S} \, \norm{\Lambda_\perp^n S^{-1}f}_\infty
 \leqs \rho^n \norm{S}\,\norm{S^{-1}}\, \norm{f}_\infty\;,
\end{equation} 
o\`u les normes de $S$ et $S^{-1}$ sont des normes subordonn\'ees, comme 
d\'efinies dans la remarque~\ref{rem:norme_subordonnee}. On conclut donc que~\eqref{eq:borne_Pperp} est v\'erifi\'e, avec $C = \norm{S}\,\norm{S^{-1}}$. 

Si $P_\perp$ n'est pas diagonalisable, on a 
\begin{equation}
 S^{-1}P_\perp S = T_\perp\;,
\end{equation} 
o\`u $T_\perp$ est une matrice triangulaire, diagonale par blocs, o\`u les blocs 
sont des \defwd{blocs de Jordan} de la forme $B(\lambda_j,b_j)$, avec 
\begin{equation}
 B(\lambda,b) = 
 \begin{pmatrix}
  \lambda & 1 & 0 & \dots & 0 \\
  0 & \lambda & 1 & & \vdots \\
  \vdots & & \ddots & \ddots & \\
  \vdots & & & \lambda & 1 \\
  0 & \dots & \dots & 0 & \lambda 
 \end{pmatrix}
 \in \C^{b\times b}\;.
\end{equation} 
La dimension $b_j$ de $B(\lambda_j,b_j)$ d\'epend de la diff\'erence entre la 
\defwd{multiplicit\'e alg\'ebrique} de $\lambda_j$ (sa multiplicit\'e en tant que racine 
du polyn\^ome caract\'eristique), et sa \defwd{multiplicit\'e g\'eom\'etrique} (la dimension 
du noyau de $P - \lambda_j\one$). Dans ce cas, on a 
\begin{equation}
 P_\perp^n = ST_\perp^n S^{-1}
 \qquad \forall n\geqs 0\;.
\end{equation} 
On pourrait alors essayer de majorer $\norm{T_\perp^n g}_\infty$ par une constante 
fois $\rho^n \norm{g}_\infty$. Il est toutefois plus commode de passer par la 
\defwd{d\'ecomposition de Dunford}, que nous rappelons ici.

\begin{proposition}[D\'ecomposition de Dunford]
Soit $P$ une matrice, admettant les valeurs propres diff\'erentes $\lambda_0, \dots, 
\lambda_{k}$. On note $m_i$ la multiplicit\'e alg\'ebrique de $\lambda_i$, et 
$g_i$ sa multiplicit\'e g\'eom\'etrique (on rappelle que $1\leqs g_i\leqs m_i$). 
Alors on a la d\'ecomposition
\begin{equation}
 P = \sum_{i=0}^k \bigpar{\lambda_i \Pi_i + N_i}\;,
\end{equation} 
o\`u 
\begin{itemize}
\item   les $\Pi_i$ sont des projecteurs, satisfaisant 
$\Pi_i\Pi_j = \delta_{ij}\Pi_i$;
\item   les $N_i$ sont nilpotentes~: elles satisfont 
$N_i^{m_i-g_i} = 0$;
\item   on a $N_iN_j = 0$ si $i\neq j$ 
et $P_i N_j = N_j P_i = \delta_{ij}N_i$. 
\end{itemize}
\end{proposition}

Il suit de la derni\`ere propri\'et\'e que 
\begin{equation}
 P_\perp^n = \sum_{i=1}^k \bigpar{\lambda_i \Pi_i + N_i}^n\;,
\end{equation} 
et la formule du bin\^ome de Newton implique 
\begin{equation}
 \bigpar{\lambda_i \Pi_i + N_i}^n 
 = \Pi_i \sum_{p=0}^{m_i - g_i - 1} \lambda_i^{n-p} \binom{n}{p} N_i^p\;. 
\end{equation} 
En effet, le fait que $N_i^{m_i-g_i} = 0$ implique que tous les termes 
avec $p \geqs m_i - g_i$ sont nuls. 

Le point important ici est que puisque $m_i - g_i$ est born\'e, $\norm{P_\perp^n f}_\infty$ 
d\'ecro\^it toujours comme $\rho^n$, m\^eme si ce terme est multipli\'e par une constante qui 
d\'epend de mani\`ere plus compliqu\'ee de $P_\perp$ (mais pas de $n$). Ainsi,~\eqref{eq:borne_Pperp} reste vrai, avec un $C$ d\'ependant des termes de la d\'ecomposition de Dunford.

Nous avons suppos\'e jusqu'ici que $\cX$ \'etait fini. Si $\cX$ est infini, la 
matrice stochastique d\'efinit un op\'erateur lin\'eaire dit \defwd{compact}, ce 
qui signifie essentiellement qu'il applique des ensembles compacts sur des ensembles 
born\'es (dont la fermeture est compacte). Pour ces op\'erateurs, la notion de valeur 
propre est encore bien d\'efinie. En particulier, on sait que toute valeur propre 
non nulle de $P$ est de multiplicit\'e finie. Par cons\'equent, on a encore une 
d\'ecomposition de Dunford. Toutefois, il est moins clair que la constante $C$ 
dans~\eqref{eq:borne_Pperp} est toujours finie.  


\section{Cas r\'eversible}
\label{sec:spec_reversible} 

Les \CMs\ r\'eversibles se pr\^etent mieux \`a une \'etude
spectrale que les \CMs\ non r\'eversibles. Pour le voir, supposons la
\CM\ irr\'eductible et r\'ecurrente positive, de distribution
stationnaire $\pi$, et introduisons le produit scalaire
\begin{equation}
\label{rev6}
\pscal fg_\pi = \sum_{x\in\cX} \pi(x) \cc{f(x)} g(x)\;,
\end{equation}
o\`u $f, g\in\C^{\cX}$ sont des vecteurs colonne.
On d\'enote par $\ell^2(\C,\pi)$ l'ensemble des vecteurs $f$ tels que 
$\pscal{f}{f}_\pi < \infty$. C'est un espace de Hilbert. 

\begin{lemma}[Caract\`ere autoadjoint de $P$]
L'op\'erateur lin\'eaire $P$ est autoadjoint dans 
l'espace de Hilbert $\cH = \ell^2(\C,\pi)$, c'est-\`a-dire 
\begin{equation}
 \pscal f{Pg}_\pi = \pscal {Pf}g_\pi 
 \qquad 
 \forall f, g \in\cH\;.
\end{equation} 
\end{lemma}
\begin{proof}
On a 
\begin{equation}
\pscal f{Pg}_\pi = \sum_{x\in\cX} \pi(x) \cc{f(x)} \sum_{y\in\cX} p_{xy}g(y)
= \sum_{y\in\cX} \pi(y) \sum_{x\in\cX} p_{yx} \cc{f(x)} g(y)
= \pscal {Pf}g_\pi\;,
\end{equation}
o\`u on a utilis\'e la r\'eversibilit\'e dans la deuxi\`eme \'egalit\'e. 
\end{proof}

Rappelons un r\'esultat classique de la th\'eorie des espaces de Hilbert. 

\begin{proposition}[Th\'eor\`eme spectral]
Soit $P$ un op\'erateur autoadjoint compact dans un espace de Hilbert $\cH$. 
Alors toutes les valeurs propres de $P$ sont r\'eelles, et les espaces
propres associ\'es sont orthogonaux. De plus, $\cH$ admet une base orthonorm\'ee 
de vecteurs propres, dans laquelle $P$ est diagonale. 
\end{proposition}
\begin{proof}
Soient $v_1$ et $v_2$ deux vecteurs propres \`a droite de $P$, de valeurs propres respectives
$\lambda_1$ et $\lambda_2$. Alors 
\begin{equation}
\label{rev8}
(\cc\lambda_1 - \lambda_2) \pscal{v_1}{v_2}_\pi 
= \pscal{\lambda_1v_1}{v_2}_\pi - \pscal{v_1}{\lambda_2v_2}_\pi 
= \pscal{Pv_1}{v_2}_\pi - \pscal{v_1}{Pv_2}_\pi = 0\;.
\end{equation}
D'une part, prenant $v_1=v_2$, on obtient que $\lambda_1$ est r\'eelle.
D'autre part, si $\lambda_1\neq\lambda_2$, on obtient l'orthogonalit\'e de
$v_1$ et $v_2$. 

Le fait que $P$ est diagonalisable se montre par r\'ecurrence. On sait que 
$P$ admet au moins une valeur propre complexe, avec vecteur propre associ\'e $v$. 
On montre alors que le compl\'ement orthogonal $v_\perp = \setsuch{w\in\cH}{\pscal{w}{v}_\pi = 0}$ 
est invariant par $P$. La restriction $P_\perp$ de $P$ \`a $v_\perp$ admet \`a nouveau 
une valeur propre, ce qui permet d'\'etablir l'h\'er\'edit\'e (si $P$ est de dimension 
finie, la r\'ecurrence s'arr\^ete lorsque le compl\'ement orthogonal est $\set{0}$).
\end{proof}

On a \'egalement un lien explicite entre vecteurs propres \`a gauche et \`a droite. 

\begin{lemma}[Vecteurs propres \`a droite et \`a gauche]
Si $v$ est un vecteur propre \`a droite de l'op\'erateur autoadjoint $P$, alors 
$\mu$ d\'efini par 
\begin{equation}
 \mu(x) = \pi(x) v(x) 
 \qquad \forall x\in\cX
\end{equation} 
est un vecteur propre \`a gauche, pour la m\^eme valeur propre.
\end{lemma}
\begin{proof}
Soit $v$ un vecteur colonne tel que $Pv = \lambda v$. 
Pour tout $x\in\cX$, on a
\begin{equation}
\bigpar{\mu P}_x
= \sum_{y\in\cX} \mu(y)p_{yx} 
= \sum_{y\in\cX} v(y) \pi(y) p_{yx}
= \pi(x) \sum_{y\in\cX} p_{xy} v(y)
= \pi(x) \bigpar{Pv}_x
= \lambda \pi(x) v(x)
= \lambda \mu(x)\;.
\end{equation}
Par cons\'equent, $\mu P = \lambda\mu$. 
\end{proof}

Une premi\`ere cons\'equence du caract\`ere autoadjoint de $P$ 
est une repr\'esentation variationnelle du trou spectral.

\begin{proposition}[Principe min-max]
Le trou spectral de $P$ satisfait 
\begin{equation}
\label{rev9}
\rho = \sup_{v \colon \pscal{v}{\vone}_\pi=0}
\frac{\abs{\pscal{v}{Pv}_\pi}}{\pscal{v}{v}_\pi}\;.
\end{equation}
\end{proposition}
\begin{proof}
Soit $(v_k)_{k\geqs0}$ une base orthonorm\'ee de vecteurs propres \`a droite
de $P$. Alors tout $v\in\cH$ s'\'ecrit 
\begin{equation}
 v = \sum_{k\geqs0} c_k v_k\;, 
 \qquad\text{ o\`u }
 c_k = \pscal{v_k}{v}_\pi\;.
\end{equation} 
On obtient alors 
\begin{align}
 \pscal{v}{v}_\pi 
 &= \sum_{k,\ell\geqs0} \cc{c}_k c_\ell \pscal{v_k}{v_\ell}_\pi 
 = \sum_{k\geqs0} \abs{c_k}^2\;, \\
 \pscal{v}{Pv}_\pi 
 &= \sum_{k,\ell\geqs0} \cc{c}_k c_\ell \pscal{v_k}{Pv_\ell}_\pi 
 = \sum_{k\geqs0} \lambda_k\abs{c_k}^2\;.
\end{align}
La premi\`ere relation n'est autre que la relation de Parseval. 
Par cons\'equent,
\begin{equation}
 \frac{\abs{\pscal{v}{Pv}_\pi}}{\pscal{v}{v}_\pi}
 \leqs \frac{\sum_{k\geqs0} \abs{\lambda_k}\abs{c_k}^2}{\sum_{k\geqs0} \abs{c_k}^2}\;.
\end{equation} 
Si $\pscal{v}{\vone}_\pi = 0$, alors $c_0 = 0$, de sorte que cette 
quantit\'e est born\'ee par $\rho$. L'\'egalit\'e a lieu dans le cas $v = v_1$, 
si on a num\'erot\'e les valeurs propres de mani\`ere que $\abs{\lambda_1} = \rho$. 
\end{proof}

Il est \'egalement possible d'obtenir une majoration analogue 
\`a~\eqref{eq:decroissance_EfXn}. M\^eme si elle ne peut pas sembler optimale, 
elle a le m\'erite d'\^etre explicite.

\begin{proposition}[Vitesse de convergence dans le cas r\'eversible]
Si la \CM\ est r\'eversible, on a la majoration
\begin{equation}
 \bigabs{\expecin{\nu}{f(X_n)} - \expecin{\pi}{f}}
 \leqs \rho^n 
 \norm{f}_\infty
 \norm{\nu-\pi}_1^{1/2}
 \sup_{x\in\cX} \biggabs{\frac{\nu(x)}{\pi(x)}-1}^{1/2}\;.
\end{equation} 
\end{proposition}
\begin{proof}
Il s'agit de majorer $\abs{(\nu-\pi)P_\perp^n f}$. La d\'ecomposition de Dunford 
s'\'ecrit 
\begin{equation}
 P_\perp^n = \sum_{k\geqs1} \lambda_k \Pi_k\;,
\end{equation} 
o\`u le projecteur $\Pi_k$ peut s'\'ecrire $\Pi_k = v_k \mu_k$. En effet, $\Pi_k$ 
projette bien sur $v_k$ par action \`a droite, et sur $\mu_k$ par action \`a gauche. 
De plus, $\Pi_k^2 = v_k (\mu_k v_k) \mu_k = \Pi_k$, puisque 
\begin{equation}
 \mu_k v_k = \sum_{x\in\cX} \mu_k(x) v_k(x)
 = \sum_{x\in\cX} \pi(x)v_k(x) v_k(x)
 = \pscal{v_k}{v_k}_\pi = 1\;.
\end{equation} 
Nous avons donc 
\begin{equation}
\label{eq:proof_nupif} 
 (\nu-\pi)P_\perp^n f = \sum_{k\geqs1} \lambda_k (\nu-\pi)v_k \mu_k f
 = \sum_{k\geqs1} \lambda_k a_k b_k\;,
\end{equation} 
o\`u nous avons pos\'e
\begin{equation}
 a_k = \mu_k f 
 = \sum_{x\in\cX} \mu_k(x)f(x) 
 = \sum_{x\in\cX} \pi(x)v_k(x)f(x)
 = \pscal{v_k}{f}_\pi\;, 
\end{equation} 
et 
\begin{equation}
 b_k = (\nu-\pi)v_k 
 = \sum_{x\in\cX} (\nu(x)-\pi(x))v_k(x)
 = \pscal{g}{v_k}_\pi\;,
\end{equation} 
o\`u $g$ est le vecteur colonne de composantes $g(x) = (\nu(x)-\pi(x))/\pi(x)$. 
Il suit alors de~\eqref{eq:proof_nupif} et de l'in\'egalit\'e de Cauchy--Schwarz que 
\begin{equation}
 \bigabs{(\nu-\pi)P_\perp^n f} 
 \leqs \rho \sum_{k\geqs1} \abs{a_k b_k} 
 \leqs \rho 
 \biggpar{\sum_{k\geqs1} a_k^2}^{1/2} 
 \biggpar{\sum_{k\geqs1} b_k^2}^{1/2}\;. 
\end{equation} 
Or, par la relation de Parseval, 
\begin{equation}
 \sum_{k\geqs1} a_k^2 
 \leqs \pscal{f}{f}_\pi = \sum_{x\in\cX} \pi(x) f(x)^2 
 \leqs\norm{f}_\infty^2\;.
\end{equation} 
D'autre part, 
\begin{equation}
 \sum_{k\geqs1} b_k^2 
 \leqs \pscal{g}{g}_\pi 
 = \sum_{x\in\cX} \pi(x)g(x)^2 
 \leqs \sup_{x\in\cX} \abs{g(x)} \, \norm{\pi g}_1\;.
\end{equation} 
Comme $\norm{\pi g}_1 = \norm{\nu - \pi}_1$, le r\'esultat est prouv\'e. 
\end{proof}

Le facteur $\norm{\nu - \pi}_1$ ne pose pas de probl\`eme, car on peut 
toujours le majorer par $\norm{\nu}_1 + \norm{\pi}_1 = 2$. Pour que le supremum 
sur $x$ soit petit, il faut que $\nu(x)$ ne soit pas trop diff\'erent de $\pi(x)$, 
du moins si $\pi(x)$ est petit. Une possibilit\'e est de choisir pour $\nu$ la 
probabilit\'e uniforme sur un ensemble probable sous $\pi$, et sur lequel $\pi$ ne 
varie pas trop.

\begin{proposition}[Cas d'un $\nu$ uniforme]
Soit $\cX_0 \subset \cX$ un ensemble fini, tel que 
\begin{equation}
 \pi(X_0^c) := \sum_{x\notin X_0} \pi(x) = \delta
 \qquad 
 \text{et}
 \qquad 
 \max_{x\in\cX_0} \pi(x) \leqs (1+c) \min_{x\in\cX_0} \pi(x)\;.
\end{equation} 
Soit $\nu$ la loi uniforme sur $\cX_0$. Alors 
\begin{equation}
 \norm{\nu-\pi}_1 \leqs 2\delta + c
 \qquad 
 \text{et}
 \qquad 
 \sup_{x\in\cX} \biggabs{\frac{\nu(x)}{\pi(x)}-1}
 \leqs \max\biggset{1, \frac{c(1 + \delta)}{(1+c)(1-\delta)}}\;.
\end{equation} 
\end{proposition}
\begin{proof}
Soit 
\begin{equation}
 m = \min_{x\in\cX_0} \pi(x)\;, 
 \qquad 
 M = \max_{x\in\cX_0} \pi(x)\;.
\end{equation} 
Alors on a $M \leqs (1+c) m$ et 
\begin{equation}
 m \abs{\cX_0} \leqs \pi(\cX_0) = 1-\delta \leqs M \abs{\cX_0}\;. 
\end{equation} 
En combinant ces in\'egalit\'es, on obtient 
\begin{equation}
 M \leqs \frac{(1+c)(1-\delta)}{\abs{\cX_0}} 
 \qquad\text{et}\qquad 
 m \geqs \frac{1-\delta}{(1+c)\abs{\cX_0}}\;.
\end{equation} 
On a 
\begin{equation}
 \norm{\nu-\pi}_1 
 = \sum_{x\in\cX_0} \biggabs{\frac{1}{\abs{\cX_0}} - \pi(x)}
 + \sum_{x\in\cX_0^c} \pi(x)\;.
\end{equation} 
La seconde somme vaut $\delta$, alors qu'en utilisant le fait que 
$m\leqs\pi(x)\leqs M$ dans la premi\`ere somme, on obtient, en simplifiant l'expression 
obtenue, que celle-ci est toujours inf\'erieure \`a $\delta+c$. Ceci prouve la majoration 
de $\norm{\nu-\pi}_1$. Pour la seconde majoration, on utilise le fait que 
\begin{equation}
 \sup_{x\in\cX} \biggabs{\frac{\nu(x)}{\pi(x)}-1} 
 = \max\biggset{\sup_{x\in\cX_0} \biggabs{\frac{\nu(x)}{\pi(x)}-1}, 1}\;,
\end{equation} 
et on borne la premi\`ere somme \`a nouveau \`a l'aide de l'encadrement 
$m\leqs\pi(x)\leqs M$. 
\end{proof}

Le message essentiel \`a retenir de ce chapitre est que la th\'eorie spectrale 
permet de montrer que $\expecin{\pi_0}{f(X_n)}$ converge exponentiellement vite 
vers $\expecin{\pi}{f}$, avec un exposant d\'etermin\'e par le trou spectral, 
et une constante proportionnelle \`a $\norm{f}_\infty$. 
Toutefois, si $\cX$ est grand ou infini, il n'est pas facile de d\'eterminer 
explicitement le trou spectral, ainsi que la constante. 
C'est pour cette raison que nous allons introduire une autre approche, bas\'ee sur 
des fonctions de Lyapounov, qui est plus flexible et a l'avantage de fournir 
des valeurs explicites de l'exposant et de la constante. 


\section{Exercices}
\label{sec:spectral_exo} 

\begin{exercise}
On consid\`ere la marche al\'eatoire sym\'etrique sur le cercle discret \`a $N$ sites~: 
\[
 p_{xy} = 
 \begin{cases}
  \frac12 & \text{si $y = x+1$\;,} \\
  \frac12 & \text{si $y = x-1$\;,} \\
  0 & \text{sinon\;,}
 \end{cases}
\]
avec l'identification modulo $N$\,: $N+1 = 1$, $0 = N$.

\begin{enumerate}
\item   Quelle est la matrice de transition de cette \CM\ ?

\item   Par un argument de sym\'etrie, trouver la probabilit\'e invariante de la cha\^ine. 

\item   Soit $\omega = \e^{2\pi\icx/N}$. Montrer que pour tout $k\in\set{0,\dots,N-1}$, le vecteur 
$v_k$ de composantes 
\[
 v_{k,x} = \omega^{k(x-1)}\;, \qquad x\in\set{1,\dots,N}
\]
est un vecteur propre de $P$. En d\'eduire les valeurs propres de $P$. 

\item   D\'eterminer le rayon spectral $\rho$ de $P$ (sa valeur propre diff\'erente de $1$ de plus grand module). Distinguer les cas $N$ pair et $N$ impair.

\item   Par un d\'eveloppement limit\'e, d\'eterminer le trou spectral $1-\rho$ \`a l'ordre dominant en $N$.  
\end{enumerate} 
\end{exercise}

\begin{exercise}
 Soit $p\in]0,1[$ et $q = 1 - p$. On consid\`ere la marche al\'eatoire asym\'etrique sur le cercle 
discret \`a $N$ sites~: 
\[
 p_{xy} = 
 \begin{cases}
  p & \text{si $y = x+1$\;,} \\
  q & \text{si $y = x-1$\;,} \\
  0 & \text{sinon\;.}
 \end{cases}
\]
Par la m\^eme m\'ethode qu'\`a l'exercice pr\'ec\'edent, d\'eterminer, en fonction de $p$, le rayon spectral $\rho$ de $P$, ainsi que le trou spectral $1-\rho$ \`a l'ordre dominant en $N$.  
\end{exercise}


\chapter{Fonctions de Lyapounov et vitesse de convergence}
\label{chap:cm_Lyapounov} 

Dans ce chapitre, nous consid\'erons \`a nouveau des \CMs\ $(X_n)_{n\geqs0}$ 
irr\'eductibles, r\'ecurrentes positives et ap\'eriodiques sur un ensemble 
d\'enombrable $\cX$. Soit $f:\cX\to\R$ une fonction born\'ee, et soit $\pi$ 
la probabilit\'e invariante de la \CM. Le but est \`a nouveau de majorer l'erreur 
\begin{equation}
 \bigabs{\expecin{\nu}{f(X_n)} - \expecin{\pi}{f}}\;.
\end{equation} 
Au lieu d'utiliser des informations sur les valeurs propres de la matrice de 
transition $P$, nous allons ici baser l'analyse sur des propri\'et\'es de 
fonctions dites de Lyapounov. Si les estimations fournies par ces fonctions ne 
sont pas toujours aussi pr\'ecises que celles provenant de l'analyse spectrale, 
la m\'ethode est plus robuste, et donne souvent des bornes explicites. 


\section{Notations -- formalisme des g\'en\'erateurs}
\label{sec:generateurs} 

Commen\c cons par pr\'eciser quelques d\'efinitions li\'ees aux mesures
et aux fonctions tests. 

\begin{definition}[Mesures sign\'ees]
\label{def:mesure} 
Une \defwd{mesure sign\'ee finie} sur $\cX$ est une application 
$\mu:\cX\to\R$ telle que 
\begin{equation}
 \norm{\mu}_1 := \sum_{x\in\cX} \abs{\mu(x)} < \infty\;.
\end{equation} 
On notera $\cE_1$ l'espace de Banach des mesures sign\'ees finies. 

\noindent
Si $\mu:\cX\to[0,1]$, et $\norm{\mu}_1 = 1$, alors $\mu$ est une 
\defwd{mesure de probabilit\'e}. 
\end{definition}

Notons que la somme de deux mesures de probabilit\'e n'est pas une mesure de probabilit\'e. Le sous-ensemble des mesures de probabilit\'e n'est donc pas un sous-espace de $\cE_1$. Cependant, la combinaison convexe de deux mesures de probabilit\'e est une mesure de probabilit\'e. 

\begin{definition}[Fonctions test]
\label{def:fct_test} 
Une \defwd{fonction test} (ou \defwd{observable}) sur $\cX$ est une application 
$f:\cX\to\R$ telle que 
\begin{equation}
 \norm{f}_\infty := \sup_{x\in\cX} \abs{f(x)} < \infty\;.
\end{equation} 
On notera $\cE_\infty$ l'espace de Banach des fonctions test. 
\end{definition}

Les notations suivantes, en parties d\'ej\`a introduites, vont s'av\'erer utiles.

\begin{itemize}
\item   Pour une mesure sign\'ee finie $\mu$ et une fonction test $f$, nous \'ecrirons 
\begin{equation}
 \mu(f) = \sum_{x\in\cX} \mu(x) f(x)\;.
\end{equation} 
Cette quantit\'e est bien d\'efinie, car 
\begin{equation}
 \abs{\mu(f)} 
 \leqs \sum_{x\in\cX} \abs{\mu(x)} \abs{f(x)}
 \leqs \sup_{x\in\cX} \abs{f(x)} \sum_{x\in\cX} \abs{\mu(x)} 
 = \norm{f}_\infty \norm{\mu}_1
 < \infty\;.
\end{equation} 

\item   Si $\mu$ est une mesure de probabilit\'e, nous \'ecrirons aussi 
$\mu(f) = \expecin{\mu}{f}$. 

\item   Si $\delta_x$ d\'enote la mesure de Dirac en $x$ (c'est-\`a-dire que 
$\delta_x(x) = 1$ et $\delta_x(y) = 0$ si $y\neq x$), on abr\`ege $\expecin{\delta_x}{f}$ par $\expecin{x}{f}$.

\item   Pour $A\subset\cX$, on \'ecrit 
\begin{equation}
 \mu(A) = \mu(\indicator{A}) = \sum_{x\in A} \mu(x)\;.
\end{equation} 

\item   Si $\mu$ est une mesure de probabilit\'e, alors $\mu(A)$ est aussi la probabilit\'e de $A$. 

\item   Pour une mesure de probabilit\'e $\mu$ et une fonction test $f$, on \'ecrira
\begin{equation}
 \expecin{\mu}{f(X_n)} 
 = \mu P^n f 
 = \sum_{x\in\cX} \sum_{y\in\cX} 
 \mu(x) (P^n)_{xy} f(y)\;,
\end{equation} 
o\`u $(P^n)_{xy}$ est l'\'el\'ement de matrice $(x,y)$ de $P^n$. 
\end{itemize}

\begin{definition}[Distance en variation totale]
La \defwd{distance en variation totale} entre deux mesures $\mu,\nu\in\cE_1$ est 
\begin{equation}
 \normTV{\mu-\nu} = 2 \sup\bigsetsuch{\abs{\mu(A) - \nu(A)}}{A \subset X}\;.
\end{equation} 
\end{definition}

Intuitivement, deux mesures sont d'autant plus proches en variation totale qu'elles donnent des probabilit\'es proches aux \'ev\'enements. 

Pour des mesures de probabilit\'e, le r\'esultat suivant montre que la distance en variation totale est en fait \'equivalente \`a la norme $\ell^1$. 

\begin{lemma}[\'Equivalence des distances]
\label{lem:TV} 
Si $\mu$ et $\nu$ sont deux mesures de probabilit\'e, alors 
\begin{equation}
 \normTV{\mu - \nu}
 = \sum_{x\in\cX} \abs{\mu(x) - \nu(x)}
 = \norm{\mu - \nu}_1\;.
\end{equation} 
\end{lemma}

\begin{proof}
Soit $B = \setsuch{x\in\cX}{\mu(x) > \nu(x)}$. Alors on a 
\begin{equation}
\label{eq:equiv_proof1} 
 0 \leqs \mu(B) - \nu(B)
 = 1 -  \mu(B^c) + (1 - \nu(B^c)) 
 = \nu(B^c) -  \mu(B^c)\;,
\end{equation} 
ce qui implique 
\begin{align}
 \sum_{x\in\cX} \abs{\mu(x) - \nu(x)}
 &= \sum_{x\in B} (\mu(x) - \nu(x)) + \sum_{x\in B^c} (\nu(x) - \mu(x)) \\
 &= \mu(B) - \nu(B) + \nu(B^c) - \mu(B^c) \\
 &= 2 \bigbrak{\mu(B) - \nu(B)}
\label{eq:equiv_proof2} 
\end{align}
par~\eqref{eq:equiv_proof1}.
De plus, pour tout $A \subset \cX$, 
\begin{equation}
 \mu(A) - \nu(A)
 \leqs \sum_{x\in A\cap B} (\mu(x) - \nu(x)) 
 \leqs \sum_{x\in B} (\mu(x) - \nu(x))
 = \mu(B) - \nu(B)\;,
\end{equation} 
o\`u nous avons utilis\'e \`a deux reprises le fait que $\mu(x) \leqs \nu(x)$ sur $A\cap B^c$. De m\^eme, 
\begin{equation}
 \nu(A) - \mu(A)
 \leqs \sum_{x\in A\cap B^c} (\nu(x) - \mu(x)) 
 \leqs \nu(B^c) - \mu(B^c) 
 = \mu(B) - \nu(B)\;.
\end{equation} 
Il suit de~\eqref{eq:equiv_proof2} que 
\begin{equation}
 \abs{\mu(A) - \nu(A)} \leqs \mu(B) - \nu(B) = \frac12\norm{\mu-\nu}_1\;.
\end{equation} 
De plus, si $A=B$, on a \'egalit\'e. 
\end{proof}

\begin{definition}[G\'en\'erateur]
Soit $P$ la matrice de transition d'une \CM\ sur un ensemble d\'enombrable $\cX$. Le \defwd{g\'en\'erateur} de la \CM\ est l'application $\cL:\cE_\infty\to\cE_\infty$ donn\'ee par 
\begin{equation}
\label{eq:def_gen} 
 (\cL f)(x) = \sum_{y\in \cX} p_{xy} \bigbrak{f(y) - f(x)}\;.
\end{equation} 
\end{definition}

Remarquons que comme $ \sum_{y\in \cX} p_{xy} = 1$, on a l'expression \'equivalente 
\begin{equation}
 (\cL f)(x) = \biggbrak{\sum_{y\in \cX} p_{xy}f(y)} - f(x)
 = \expecin{x}{f(X_1)} - f(x)\;.
\end{equation} 
On peut donc \'ecrire $\cL = P - \one$, o\`u $\one$ d\'enote la matrice identit\'e.


\section{Fonctions de Lyapounov}
\label{sec:Lyap} 

Dans la suite, nous supposons que $P$ est la matrice de transition d'une \CM\ 
\defwd{irr\'eductible} sur $\cX$. De plus, nous supposons que $\cX$ est \'equip\'e d'une norme $\norm{\cdot}$. Par exemple, si $\cX \subset \Z$, on peut prendre $\norm{x} = \abs{x}$. Si $\cX \subset \Z^d$, on peut prendre la norme Euclidienne (ou toute autre norme \'equivalente). 

\begin{definition}[Fonction de Lyapounov]
Une \defwd{fonction de Lyapounov} est une fonction 
$V: \cX\to \R_+ = [0,\infty[$ satisfaisant
\begin{equation}
\label{eq:gen} 
 V(x) \to +\infty 
 \qquad \text{pour $\norm{x}\to\infty$\;.}
\end{equation} 
\end{definition}

\begin{proposition}[Formule de Dynkin]
\label{prop:Dynkin} 
Pour toute fonction de Lyapounov $V$, on a 
\begin{equation}
\label{eq:Dynkin} 
 \bigexpecin{x}{V(X_n)} 
 = V(x) + \biggexpecin{x}{\sum_{m=0}^{n-1} (\cL V)(X_m)}\;.
\end{equation} 
De plus, si $\tau$ est un temps d'arr\^et tel que $\expecin{x}{\tau} < \infty$, alors 
\begin{equation}
 \bigexpecin{x}{V(X_\tau)} 
 = V(x) + \biggexpecin{x}{\sum_{m=0}^{\tau-1} (\cL V)(X_m)}\;.
\end{equation} 
\end{proposition}
\begin{proof}
Montrons~\eqref{eq:Dynkin}. 
On proc\`ede par r\'ecurrence sur $n$. L'initialisation se fait pour $n=1$, o\`u la d\'efinition~\eqref{eq:def_gen} du g\'en\'erateur implique 
\begin{equation}
 \bigexpecin{x}{V(X_1)} = V(x) + (\cL V)(x)\;.
\end{equation} 
Pour v\'erifier l'h\'er\'edit\'e, 
une premi\`ere fa\c con de proc\'eder est d'\'ecrire 
\begin{align}
\bigexpecin{x}{V(X_{n+1})} 
&= \sum_{y\in\cX} V(y) \probin{x}{X_{n+1} = y} \\
&= \sum_{y\in\cX} V(y) \sum_{z\in\cX} 
\underbrace{\pcondin{x}{X_{n+1}=y}{X_n=z}}_{=p_{zy}}
\bigprobin{x}{X_n = z} \\
&= \sum_{z\in\cX} \bigprobin{x}{X_n = z}
\underbrace{\sum_{y\in\cX} V(y) p_{zy}}_{=(\cL V)(z) + V(z)} \\
&= \biggexpecin{x}{\sum_{z\in\cX}\indicator{X_n=z}(\cL V)(z)}
+ \sum_{z\in\cX} \bigprobin{x}{X_n = z}V(z) \\
&= \bigexpecin{x}{(\cL V)(X_n)} + \bigexpecin{x}{V(X_n)}\;.
\end{align}
Une autre mani\`ere de proc\'eder est d'utiliser le formalisme des 
esp\'erances conditionnelles, en \'ecrivant 
\begin{equation}
 \bigexpecin{x}{V(X_{n+1})} = \bigexpecin{x}{V(X_n)} +
 \bigexpecin{x}{V(X_{n+1}) - V(X_n)}\;.
\end{equation} 
Or, si $\cF_n$ d\'enote la tribu engendr\'ee par $(X_0, X_1, \dots, X_n)$, on a 
\begin{align}
 \bigexpecin{x}{V(X_{n+1}) - V(X_n)}
 &= \bigexpecin{x}{\bigecondin{x}{V(X_{n+1}) - V(X_n)}{\cF_n}} \\
 &= \bigexpecin{x}{\bigexpecin{X_n}{V(X_{n+1}) - V(X_n)}}
 = \bigexpecin{x}{(\cL V)(X_n)}\;.
\end{align} 
Avec l'hypoth\`ese de r\'ecurrence, ceci conclut la d\'emonstration. 
\end{proof}

\begin{theorem}[Croissance sous-exponentielle]
\label{thm:sous_exp} 
Supposons qu'il existe une fonction de Lyapounov $V$ et $c > 0$, $d\geqs0$ tels que 
\begin{equation}
 (\cL V)(x) \leqs c V(x) + d
 \qquad \forall x\in\cX\;.
\end{equation} 
Alors on a 
\begin{equation}
 \bigexpecin{x}{V(X_n)} \leqs (1+c)^n V(x) + \frac{(1+c)^n-1}{c}d
\end{equation} 
pour tout $n\in\N$ et tout $x\in\cX$. 
\end{theorem}
\begin{proof}
Commen\c cons par consid\'erer le cas $d = 0$. 
Notons $f_n(x) = \expecin{x}{V(X_n)}$. Alors la formule de Dynkin implique 
\begin{align}
 f_n(x) &= V(x) + \biggexpecin{x}{\sum_{m=0}^{n-1} (\cL V)(X_m)} \\
 &\leqs V(x) + c \sum_{m=0}^{n-1} f_m(x)\;.
\end{align}
En utilisant $f_0(x) = V(x)$ comme initialisation, on obtient facilement 
par r\'ecurrence sur $n$ que $f_n(x) \leqs (1+c)^n V(x)$. 
Dans le cas $d > 0$, la relation de r\'ecurrence devient 
\begin{equation}
 f_n(x) \leqs V(x) + c \sum_{m=0}^{n-1} f_m(x) + nd\;.
\end{equation} 
En posant 
\begin{equation}
 f_n(x) \leqs (1+c)^n V(x) + k_n d\;,
\end{equation} 
on obtient pour $k_n$ la relation de r\'ecurrence 
\begin{equation}
 k_n = c\sum_{m=0}^{n-1} k_m + n\;, 
 \qquad 
 k_0 = 0\;.
\end{equation}
On v\'erifie par r\'ecurrence que ceci \'equivaut \`a 
\begin{equation}
 k_n = \frac1c \Bigpar{(1+c)^n - 1}\;,
\end{equation} 
d'o\`u le r\'esultat. 
\end{proof}

\begin{theorem}[Non-explosion]
\label{thm:non_explosion} 
Supposons qu'il existe $d \geqs 0$ et un ensemble born\'e $K\subset\cX$ tel que 
pour tout $x\in\cX$, on ait 
\begin{equation}
 (\cL V)(x) 
 \leqs d \indicator{K}(x) 
 = 
 \begin{cases}
  d    & \text{si $x\in K$\;,} \\
  0    & \text{sinon\;.}
 \end{cases}
\end{equation} 
Alors 
\begin{equation}
 \biggprobin{x}{\lim_{n\to\infty} \norm{X_n} = \infty} = 0
 \qquad \forall x\in\cX\;.
\end{equation} 
\end{theorem}
\begin{proof}[D\'emonstration (id\'ee)] 
Soit $\Omega_1$ l'\'ev\'enement 
\begin{equation}
 \Omega_1 = \biggsetsuch{\omega}{\lim_{n\to\infty} \norm{X_n(\omega)} = \infty}\;.
\end{equation} 
Consid\'erons 
le cas $d=0$. Alors la formule de Dynkin implique 
\begin{equation}
 \bigexpecin{x}{V(X_n)} \leqs V(x)\;.
\end{equation} 
Or, on a aussi 
\begin{equation}
 \bigexpecin{x}{V(X_n)} 
 = \underbrace{\bigexpecin{x}{V(X_n) \indicator{\Omega_1^c}}}_{\geqs0}  
 + \bigexpecin{x}{V(X_n)\indicator{\Omega_1}} 
\end{equation} 
Par cons\'equent, $\bigexpecin{x}{V(X_n)\indicator{\Omega_1}} 
\leqs  \bigexpecin{x}{V(X_n)}\leqs V(x)$. 
Comme $V(X_n)$ tend vers l'infini sur $\Omega_1$, ceci n'est possible que si $\bigexpecin{x}{\indicator{\Omega_1}} = \probin{x}{\Omega_1} = 0$. 
%
\end{proof}

\begin{theorem}[R\'ecurrence positive]
\label{thm:rec_pos} 
Soit $f: \cX\to[1,\infty[$ et $V$ une fonction de Lyapounov telle que 
\begin{equation}
 (\cL V)(x) \leqs -cf(x) + d\indicator{K}(x) 
 \qquad \forall x\in \cX\;,
\end{equation} 
pour un ensemble born\'e $K\subset \cX$ et des constantes $c>0$ et $d\geqs0$. 
Supposons de plus qu'il existe $\delta > 0$ tel que $K$ satisfait 
\begin{equation}
\label{eq:proba_lb}
p_{xy} \geqs \delta    
\qquad \forall x, y\in K\;.
\end{equation} 
Alors la \CM\ est r\'ecurrente positive, et admet donc une mesure de probabilit\'e invariante $\pi$. De plus, 
\begin{equation}
 \pi(f) < \infty\;.
\end{equation} 
\end{theorem}
\begin{proof}
Commen\c cons par consid\'erer le cas o\`u $K = \set{x_0}$ contient un seul point. Nous voulons montrer que $x_0$ est r\'ecurrent positif, car dans ce cas, la \CM\ \'etant irr\'eductible, elle sera r\'ecurrente positive. Soit donc 
\begin{equation}
 \tau_K = \tau_{x_0} 
 = \inf\setsuch{n\geqs1}{X_n\in K}\;.
\end{equation} 
Pour tout $T\in\N$, posons $\tau_K\wedge T = \min\set{\tau_K,T}$. 
Par la formule de Dynkin, qui s'applique puisque $\tau_K\wedge T$ 
est un temps d'arr\^et fini presque s\^urement, donc d'esp\'erance finie, nous avons 
\begin{align}
 0 \leqs \bigexpecin{x_0}{V(X_{\tau_K\wedge T)}}
 &= V(x_0) + \biggexpecin{x_0}{\sum_{m=0}^{\tau_K\wedge T-1} (\cL V)(X_m)} \\
 \label{eq:proof_recpos1} 
 &\leqs V(x_0) - c \biggexpecin{x_0}{\sum_{m=0}^{\tau_K\wedge T-1} f(X_m)} + d \\
 &\leqs V(x_0) - c\bigexpecin{x_0}{\tau_K\wedge T} + d\;.
\end{align}
Il suit que 
\begin{equation}
 \bigexpecin{x_0}{\tau_{x_0}\wedge T} = \bigexpecin{x_0}{\tau_K\wedge T}
 \leqs \frac{V(x_0)+d}{c}\;, 
\end{equation} 
ce qui est fini. Comme la quantit\'e de droite est ind\'ependante de $T$, on peut prendre 
la limite lorsque $T$ tend vers l'infini, pour laquelles $\tau_K\wedge T$ tend vers 
$\tau_K$. Ceci implique que la \CM\ est r\'ecurrente positive. 

Dans le cas o\`u $K$ contient au moins deux points, nous savons d\'ej\`a par le raisonnement ci-dessus que $\bigexpecin{x_0}{\tau_K}$ est fini pour tout $x_0\in K$. 
Il nous faut montrer que $\bigexpecin{x_0}{\tau_{x_0}}$ est \'egalement fini pour tout $x_0\in K$.

Soit  $\tau_{K,n}$ le temps du $n$i\`eme passage de la cha\^ine en $K$, et soit 
\begin{equation}
 Y_n = X_{\tau_{K,n}}\;.
\end{equation} 
(La \CM\ $(Y_n)_{n\geqs0}$ est appel\'ee la \defwd{trace} de $(X_n)_{n\geqs0}$ sur $K$.) On peut d\'eduire de l'hypoth\`ese~\eqref{eq:proba_lb} que $Y_n$ atteint $x_0$ en un temps d'esp\'erance finie. Ceci suit du fait qu'une cha\^ine r\'ecurrente sur un ensemble fini est r\'ecurrente positive. Comme le temps de retour vers $K$ est born\'e, il suit qu'on a bien $\bigexpecin{x_0}{\tau_{x_0}} < \infty$. 

La \CM\ \'etant r\'ecurrente positive, elle admet une unique probabilit\'e invariante $\pi$. Il reste \`a montrer que $\pi(f)$ est fini. En fait (voir~\eqref{eq:gamma(y)}), la mesure $\mu$ donn\'ee par 
\begin{equation}
 \mu(y) = \biggexpecin{x_0}{\sum_{n=1}^{\tau_{x_0}} \indicator{X_n = y}}
\end{equation} 
est invariante pour tout $x_0\in\cX$. La probabilit\'e invariante $\pi$ est obtenue en normalisant $\mu$. Comme 
\begin{equation}
 \sum_{y\in\cX} \mu(y)
 = \biggexpecin{x_0}{\sum_{n=1}^{\tau_{x_0}} 
 \underbrace{\sum_{y\in\cX}\indicator{X_n = y}}_{=1}}
 = \bigexpecin{x_0}{\tau_{x_0}}\;,
\end{equation} 
on conclut que 
\begin{equation}
 \pi(y) = \frac{1}{\bigexpecin{x_0}{\tau_{x_0}}}
 \biggexpecin{x_0}{\sum_{n=1}^{\tau_{x_0}} \indicator{X_n = y}}\;.
\end{equation} 
Il suit que 
\begin{equation}
 \pi(f) 
 = \frac{1}{\bigexpecin{x_0}{\tau_{x_0}}}
 \biggexpecin{x_0}{\sum_{n=1}^{\tau_{x_0}} 
 \underbrace{\sum_{y\in\cX} f(y)\indicator{X_n = y}}_{=f(X_n)}}
 = \frac{1}{\bigexpecin{x_0}{\tau_{x_0}}}
 \biggexpecin{x_0}{\sum_{n=1}^{\tau_{x_0}} f(X_n)}\;.
\end{equation} 
Si $K=\set{x_0}$~\eqref{eq:proof_recpos1} implique que pour tout $T\in\N$, 
\begin{equation}
  \biggexpecin{x_0}{\sum_{n=1}^{\tau_{x_0}\wedge T} f(X_n)}
  =  \biggexpecin{x_0}{\sum_{n=0}^{\tau_{x_0}\wedge T-1} f(X_n)}
  \leqs \frac{V(x_0)+d}{c}\;,
\end{equation}  
ce qui permet de majorer $\pi(f)$, en prenant le limite $T\to\infty$. 
Si $K$ contient deux points ou plus, on a une majoration analogue, 
mais avec $d$ multipli\'e par l'esp\'erance du nombre de points de $K$
visit\'es avant d'atteindre $x_0$, qui est finie. 
\end{proof}

\begin{remark}
On peut affaiblir l'hypoth\`ese~\eqref{eq:proba_lb} sur $K$ de plusieurs mani\`eres. 
\begin{itemize}
\item   Il suffit de supposer qu'il existe un $k\geqs1$ tel que 
$(P^k)_{xy} \geqs \delta$ pour tout $x,y\in K$ et un $\delta>0$.

\item   Une condition suffisante encore plus faible est qu'il existe des r\'eels 
$a_1, a_2, \dots \geqs 0$, de somme \'egale \`a $1$, tels que 
\begin{equation}
 \sum_{k=1}^\infty a_k (P^k)_{xy} \geqs \delta
\end{equation} 
pour tout $x,y\in K$ et un $\delta>0$. Des ensembles $K$ satisfaisant ce crit\`ere sont appel\'es \defwd{petits} (\myquote{petite set}\ en anglais). 
\end{itemize}

\end{remark}


\section{Normes \`a poids}
\label{sec:normes_poids} 

Pour obtenir des r\'esultats de convergence de la loi de $X_n$ vers $\pi$, il est 
plus utile de travailler avec des normes \`a poids.

\begin{definition}[Norme \`a poids sur les fonctions test]
\label{def:norme_poids_fct_test} 
Soit $W: \cX\to [1,\infty[$. La \defwd{norme \`a poids $W$} d'une fonction 
test est d\'efinie comme 
\begin{equation}
 \norm{f}_W = \sup_{x\in\cX} \frac{\abs{f(x)}}{W(x)}\;.
\end{equation} 
On notera $\cE_\infty^W$ l'espace de Banach des fonctions test $f$ telles que 
$\norm{f}_W < \infty$.
\end{definition}

Notons les propri\'et\'es suivantes\,:
\begin{itemize}
\item   Pour tout $x\in\cX$, on a $\abs{f(x)} \leqs \norm{f}_W W(x)$.

\item   On a $\norm{f}_W \leqs \norm{f}_\infty$. Par cons\'equent $\cE_\infty \subset \cE_\infty^W$.

\item   Plus g\'en\'eralement, si $W_1$ et $W_2$ sont deux poids tels que 
$W_1(x) \leqs W_2(x)$ pour tout $x\in\cX$, alors $\norm{f}_{W_2} \leqs \norm{f}_{W_1}$. 
Par cons\'equent $\cE_\infty^{W_1} \subset \cE_\infty^{W_2}$.
\end{itemize}

On peut \'egalement d\'efinir une m\'etrique duale \`a $\norm{\cdot}_W$ entre mesures sign\'ees finies de la mani\`ere suivante.

\begin{definition}[Distance \`a poids entre mesures]
Pour une fonction poids $W: \cX\to [1,\infty[$ et deux mesures sign\'ees finies $\mu, \nu$, 
on pose 
\begin{align}
\label{eq:def_dist_W} 
 \rho_W(\mu, \nu) 
 &= \sup_{f \colon \norm{f}_W \leqs 1}
 \sum_{x\in\cX} f(x) \abs{\mu(x) - \nu(x)} \\
 &= \sup_{f \colon \norm{f}_W \neq 0}
 \frac{1}{\norm{f}_W}
 \sum_{x\in\cX} f(x) \abs{\mu(x) - \nu(x)}\;.
\end{align} 
\end{definition}

On a alors les propri\'et\'es suivantes\,:

\begin{itemize}
\item   On peut remplacer $f(x)$ par $\abs{f(x)}$ dans la d\'efinition~\eqref{eq:def_dist_W}. 
En effet, cette transformation ne change pas $\norm{f}_W$.

\item   On peut \'egalement remplacer $\abs{\mu(x) - \nu(x)}$ par $\mu(x) - \nu(x)$. 
En effet, il suffit de changer le signe de $f(x)$ selon le signe de $\mu(x) - \nu(x)$
pour trouver le m\^eme r\'esultat. 

\item   Si $W(x) = 1$ pour tout $x\in\cX$, alors on a 
\begin{equation}
 \rho_W(\mu, \nu) = \normTV{\mu - \nu}\;.
\end{equation} 
En effet, le supremum dans~\eqref{eq:def_dist_W} est alors atteint pour la fonction test $f$ valant $1$ partout. 

\item   Pour un poids $W$ g\'en\'eral, on a  
\begin{equation}
\label{eq:rhoW} 
 \rho_W(\mu, \nu) = \sum_{x\in\cX} W(x) \abs{\mu(x) - \nu(x)}\;.
\end{equation} 
En effet, le supremum dans~\eqref{eq:def_dist_W} est atteint pour $f(x) = W(x)$ pour tout $x\in\cX$.

\item   On a la majoration 
\begin{equation}
\label{eq:majo_munuf} 
 \bigabs{(\mu - \nu)(f)} \leqs \norm{f}_W \rho_W(\mu,\nu)\;.
\end{equation} 
En effet, 
\begin{equation}
 \bigabs{(\mu - \nu)(f)} \leqs \sum_{x\in\cX} \abs{\mu(x) - \nu(x)} \, \abs{f(x)} 
 \leqs \norm{f}_W \sum_{x\in\cX} \abs{\mu(x) - \nu(x)} W(x)\;,
\end{equation} 
d'o\`u le r\'esultat, par~\eqref{eq:rhoW}.
\end{itemize}


\section{Un crit\`ere de convergence}
\label{sec:convergence} 

Les th\'eor\`emes de la section~\ref{sec:Lyap} sont d\^us \`a Meyn et Tweedie~\cite{Meyn_Tweedie_92}. 
Leurs travaux fournissent \'egalement un r\'esultat de convergence de $\expecin{x}{f(X_n)}$ 
vers $\pi(f)$, mais les hypoth\`eses sont assez difficiles \`a v\'erifier (notamment, 
tous les ensembles born\'es doivent \^etre petits), et les bornes obtenues ne sont 
pas explicites. 

Le r\'esultat suivant est d\^u \`a Hairer et Mattingly~\cite{Hairer_Mattingly_11}. Les hypoth\`eses sont plus faciles \`a v\'erifier en pratique, et la majoration obtenue pour $\expecin{x}{f(X_n)} - \pi(f)$ a l'avantage de faire intervenir des constantes explicites.

\begin{theorem}[Crit\`ere de convergence pour esp\'erances]
\label{thm:convergence} 
Supposons que les deux conditions suivantes soient satisfaites.
\begin{enumerate}
\item  \textbf{Condition de d\'erive g\'eom\'etrique\,:}
Il existe $d\geqs0$, $c>0$ et une fonction de Lyapounov $V$ tels que 
\begin{equation}
\label{eq:derive_geom} 
 (\cL V)(x) \leqs -c V(x) + d
 \qquad \forall x\in\cX\;.
\end{equation} 

\item   \textbf{Condition de minoration\,:}
Pour un $R > 2d/c$, soit $K = \setsuch{x\in\cX}{V(x) < R}$.
Alors il existe $\alpha\in]0,1[$ et une mesure de probabilit\'e $\nu$ telle que 
\begin{equation}
\label{eq:minoration} 
 \inf_{x\in K} p_{xy} =
 \inf_{x\in K} \bigprobin{x}{X_1 = y} \geqs \alpha \nu(y)
 \qquad \forall y\in\cX\;.
\end{equation} 
\end{enumerate}

\noindent
Alors il existe des constantes $M>0$ et $\bar\gamma < 1$ telles que 
\begin{equation}
\label{eq:borne_cv_expec} 
 \norm{\expecin{\cdot}{f(X_n)} - \pi(f)}_{1+V}
 \leqs M\bar\gamma^n \norm{f - \pi(f)}_{1+V}
\end{equation} 
pour toute fonction test $f\in\cE_\infty^{1+V}$. 
\end{theorem}

Pr\'ecisons qu'on a 
\begin{equation}
 \norm{\expecin{\cdot}{f(X_n)} - \pi(f)}_{1+V}
 = \sup_{x\in\cX} \frac{\abs{\expecin{x}{f(X_n)} - \pi(f)}}{1+V(x)}\;.
\end{equation} 
La majoration~\eqref{eq:borne_cv_expec} peut donc s'\'ecrire 
\begin{equation}
 \bigabs{\expecin{x}{f(X_n)} - \pi(f)}
 \leqs (1+V(x)) M\bar\gamma^n \norm{f - \pi(f)}_{1+V}
 \qquad \forall x\in\cX\;.
\end{equation} 
Comme $\bar\gamma < 1$, on a donc convergence exponentielle de $\expecin{x}{f(X_n)}$ 
vers $\pi(f)$. Comme en pratique, on peut souvent choisir $x$ tel que $V(x)$ ne soit 
pas trop grand, la d\'ependance en $V(x)$ ne pose pas de probl\`eme. 

La pr\'esence de $\norm{f - \pi(f)}_{1+V}$ n'est pas vraiment restrictive non plus.
Si par exemple $f\in\cE_\infty$, ou si $f$ est \`a support compact (c'est-\`a-dire 
nulle en-dehors d'un ensemble compact), cette quantit\'e est finie. 

La condition de minoration~\eqref{eq:minoration} est un peu plus faible que la 
condition~\eqref{eq:proba_lb} du Th\'eo\-r\`eme~\ref{thm:rec_pos}. Le point crucial est 
que l'on ait une borne inf\'erieure sur les probabilit\'es de transition qui soit 
ind\'ependante du point de d\'epart dans $K$. 

Dans la suite, nous utiliserons les notations
\begin{align}
 (\mu\cP)(y) 
 &= \probin{\mu}{X_1 = y} 
 = \sum_{x\in\cX} \mu(x) p_{xy}\;, \\
 (\cP f)(x) 
 &= \expecin{x}{f(X_1)} 
 = \sum_{y\in\cX} p_{xy} f(y)\;.
\end{align} 
La condition de d\'erive g\'eom\'etrique~\eqref{eq:derive_geom} est \'equivalente \`a
\begin{equation}
 (\cP V)(x) \leqs \gamma V(x) + d
 \qquad \forall x\in\cX\;.
\end{equation} 
avec $\gamma = 1-c$.
L'ingr\'edient essentiel de la d\'emonstation du Th\'eor\`eme~\ref{thm:convergence} 
est la borne suivante.

\begin{proposition}[L'application $\cP$ est contractante pour la distance $\rho_{1+\beta V}$]
\label{prop:contraction} 
Il existe $\bar\gamma\in]0,1[$ et $\beta>0$ tels que
\begin{equation}
 \rho_{1+\beta V}(\mu\cP,\nu\cP) \leqs \bar\gamma \rho_{1+\beta V}(\mu,\nu)\;.
\end{equation} 
\end{proposition}

En fait, la proposition donne des expressions explicites pour les constantes $\beta$ et 
$\bar\gamma$\,: pour tout choix de $\alpha_0$ et $\gamma_0$ tels que 
\begin{equation}
\label{eq:cond_gamma0} 
 0 < \alpha_0 < \alpha 
 \qquad\text{et}\qquad 
 \gamma + \frac{2d}{R} < \gamma_0 < 1\;,
\end{equation} 
on peut prendre 
\begin{align}
 \beta &= \frac{\alpha_0}{d}\;,\\
 \bar\gamma &= \max\biggset{1 - (\alpha-\alpha_0), \frac{2+R\beta\gamma_0}{2+R\beta}}
 = 1 - \min\biggset{\alpha - \alpha_0, \frac{R\beta(1-\gamma_0)}{2 + R\beta}}\;.
\label{eq:beta_gammabar} 
\end{align} 
De plus, nous verrons ci-dessous que l'on a  
\begin{equation}
\label{eq:M} 
 M = \frac{1}{1-\bar\gamma} 
 \sup_{x\in\cX} \frac{2 + \beta\brak{(1+\gamma)V(x)+d}}{1 + \beta V(x)}
 \leqs\frac{\max\set{1+\gamma,2+\beta d}}{1-\bar\gamma}\;. 
\end{equation} 
Ces expressions ne sont par particuli\`erement \'el\'egantes, mais elles ont le m\'erite 
d'\^etre explicites, ce qui peut servir dans les applications. Notons que si $d$ tend 
vers $0$, alors $\beta$ tend vers l'infini, et on peut prendre $\bar\gamma$ arbitrairement 
proche de $\gamma$. 

Nous allons d'abord montrer que la Proposition~\ref{prop:contraction} implique bien le 
Th\'eor\`eme~\ref{thm:convergence}. 

\begin{proof}[\textit{D\'emonstration du Th\'eor\`eme~\ref{thm:convergence}}]
Nous donnons tout d'abord une d\'emonstation de l'existence de $\pi$, m\^eme si 
celle-ci suit en fait du Th\'eor\`eme~\ref{thm:rec_pos}. C'est une application 
du th\'eor\`eme du point fixe de Banach.

Fixons $x_0\in\cX$, et soit $\mu_0^{x_0} = \delta_{x_0}$ la mesure de Dirac en $x_0$. 
Soit $\mu^{x_0}_n = \mu_0^{x_0}\cP^n$. La proposition implique 
\begin{equation}
 \rho_{1+\beta V}(\mu^{x_0}_{n+1},\mu^{x_0}_n)
 \leqs \bar\gamma \rho_{1+\beta V}(\mu^{x_0}_n,\mu^{x_0}_{n-1})
 \leqs \dots 
 \leqs \bar\gamma^n \rho_{1+\beta V}(\mu^{x_0}_1,\mu^{x_0}_0)\;.
\end{equation} 
On a donc une suite de Cauchy, et comme on sait que la distance en variation totale 
est compl\`ete, donc a fortiori la distance $\rho_{1+\beta V}$, on en conclut que la 
suite des $\mu^{x_0}_n$ converge vers une mesure $\pi$ en variation totale. De plus, 
$\pi(1+V)$ est finie, car $\rho_{1+\beta V}(\pi, \mu_0^{x_0})$ l'est.

Pour montrer que $\pi$ est invariante, il suffit d'observer que 
\begin{equation}
 \pi\cP = \lim_{n\to\infty} \mu^{x_0}_0\cP^{n+1}
 = \lim_{n\to\infty} \mu^{x_0}_0\cP^n = \pi\;.
\end{equation} 
Afin de d\'emontrer~\eqref{eq:borne_cv_expec}, il est utile de centrer $f$. Posons donc 
\begin{equation}
 \hat f(x) = f(x) - \pi(f)  
 \qquad \forall x\in\cX\;.
\end{equation} 
Alors $\pi(\hat f) = \pi(f) - \pi(\pi(f)) = 0$, puisque $\pi(\pi(f)) = \pi(f)$. Ainsi
\begin{equation}
 \cP^n f - \pi(f)
 = \cP^n \hat f + \cP^n \pi(f) - \pi(\hat f) - \pi(\pi(f)) = \cP^n \hat f\;.
\end{equation} 
Ceci permet d'\'ecrire 
\begin{equation}
 \norm{\expecin{\cdot}{f(X_n)} - \pi(f)}_{1+V}
 = \norm{\cP^n f - \pi(f)}_{1+V}
  = \norm{\cP^n \hat f}_{1+V}\;.
\end{equation} 
Il s'agit donc de montrer qu'il existe une constante $M$ telle que pour toute fonction test $\hat f$ 
satisfaisant $\pi(\hat f) = 0$, on ait
\begin{equation}
 \norm{\cP^n \hat f}_{1+V}
 \leqs M\bar\gamma^n \norm{\hat f}_{1+V}\;.
\end{equation} 
Or, comme $(\cP^n \hat f)(x) = \delta_x(\cP^n \hat f) = \mu^x_n(\hat f)$
et $\pi(\hat f) = 0$, on a 
\begin{equation}
 \norm{\cP^n \hat f}_{1+\beta V}
 = \sup_{x\in\cX} \frac{\abs{(\mu^x_n - \pi)(\hat f)}}{1 + \beta V(x)}
 \leqs \norm{\hat f}_{1+\beta V} \sup_{x\in\cX} \frac{\rho_{1+\beta V}(\mu^x_n,\pi)}{1 + \beta V(x)}
\end{equation} 
en vertu de~\eqref{eq:majo_munuf}. Observons que 
\begin{align}
\rho_{1+\beta V}(\mu^x_n,\pi)
&= \rho_{1+\beta V}(\mu^x_n,\mu^x_{n+1}) + \rho_{1+\beta V}(\mu^x_{n+1},\mu^x_{n+2}) + \dots \\
&\leqs \bigbrak{\bar\gamma^n + \bar\gamma^{n+1} + \dots} \rho_{1+\beta V}(\mu^x_1,\mu^x_0) \\
&= \frac{\bar\gamma^n}{1 - \bar\gamma} \rho_{1+\beta V}(\mu^x_1,\mu^x_0)\;.
\end{align}
De plus, 
\begin{align}
\rho_{1+\beta V}(\mu^x_1,\mu^x_0)
&= \sum_{y\in\cX} (1 + \beta V(y)) \abs{\mu^x_1(y) - \mu^x_0(y)} \\
&\leqs \sum_{y\in\cX} (1 + \beta V(y)) \mu^x_1(y) + 1 + \beta V(x) \\
&= (1 + \beta (\cP V))(x) + 1 + \beta V(x)\;.
\end{align}
Comme $(\cP V)(x) = (\cL V)(x) + V(x) \leqs \gamma V(x) + d$ par l'hypoth\`ese de d\'erive 
g\'eom\'etrique~\eqref{eq:derive_geom}, on obtient finalement
\begin{equation}
 \norm{\cP^n \hat f}_{1+\beta V}
 \leqs \norm{\hat f}_{1+\beta V} \frac{\bar\gamma^n}{1 - \bar\gamma}
 \sup_{x\in\cX} \frac{2 + \beta[V(x) + \gamma V(x) + d]}{1 + \beta V(x)}
 =: M \bar\gamma^n \norm{\hat f}_{1+\beta V}\;.
\end{equation} 
En particulier, la borne est vraie pour $\beta = 1$, ce qui conclut la d\'emonstation.
\end{proof}

Il nous reste \`a d\'emontrer la Proposition~\ref{prop:contraction}.
L'id\'ee est de travailler avec une d\'efinition alternative de la distance $\rho_\beta$. 
On introduit sur $\cX$ la distance 
\begin{equation}
 d_\beta(x,y)
 = 
 \begin{cases}
  0 & \text{si $x=y$\;,} \\
  2 + \beta V(x) + \beta V(y) & \text{si $x\neq y$\;,}
 \end{cases}
\end{equation} 
(on v\'erifie facilement que $d_\beta$ satisfait bien la d\'efinition d'une distance), 
et la semi-norme de Lipschitz
\begin{equation}
\label{def:seminormf} 
 \normDgamma{f}_\beta = \sup_{x\neq y} \frac{\abs{f(x)-f(y)}}{d_\beta(x,y)}\;,
\end{equation} 
C'est une semi-norme, et non une norme, car $\normDgamma{f}_\beta = 0$ n'implique pas 
$f = 0$ (mais seulement que $f$ est constante). 
On a alors le r\'esultat suivant.

\begin{lemma}[\'Equivalence des distances]
\label{lem:equiv_distances} 
On a 
\begin{equation}
\label{eq:lem_equiv_distances} 
 \rho_{1+\beta V}(\mu,\nu) 
 = \rho^*_{1+\beta V}(\mu,\nu)
 := \sup_{f: \normDgamma{f}_\beta \leqs 1} 
 \sum_{x\in\cX} f(x) (\mu(x) - \nu(x))\;.
\end{equation}  
\end{lemma}

Notons que la seule diff\'erence entre la d\'efinition~\eqref{eq:def_dist_W} de 
$\rho_{1+\beta V}(\mu,\nu)$ et cette \'egalit\'e est l'ensemble des $f$ sur lequel on prend le supremum. 

Montrons tout d'abord que ce lemme implique la proposition.

\begin{proof}[\textit{D\'emonstration de la Proposition~\ref{prop:contraction}}]
L'id\'ee est de montrer que $\cP$ est contractante dans la seminorme $\normDgamma{\cdot}_\beta$, 
\`a savoir
\begin{equation}
\label{eq:contraction_seminorm} 
 \normDgamma{\cP f}_\beta \leqs \bar\gamma \normDgamma{f}_\beta\;.
\end{equation} 
En effet, ceci implique 
\begin{align}
 \rho_{1+\beta V}(\mu\cP, \nu\cP)
 &= \sup_{f\colon \normDgamma{f}_\beta \neq 0}
 \frac{1}{\normDgamma{f}_\beta} (\mu\cP - \nu\cP)(f) \\
 &= \sup_{f\colon \normDgamma{f}_\beta \neq 0}
 \frac{1}{\normDgamma{f}_\beta} (\mu - \nu)(\cP f) \\
 &\leqs \sup_{f_1\colon \normDgamma{f_1}_\beta \neq 0}
 \frac{\bar\gamma}{\normDgamma{f_1}_\beta} (\mu - \nu)(f_1) \\
 &= \bar\gamma\rho_{1+\beta V}(\mu, \nu)\;.
\end{align}
Afin de d\'emontrer~\eqref{eq:contraction_seminorm}, nous fixons un $f$ tel que 
$\norm{f}_{1+\beta V} \leqs 1$ et $\normDgamma{f}_\beta\leqs1$. Il s'agit de montrer que 
\begin{equation}
 \bigabs{(\cP f)(x) - (\cP f)(y)} 
 \leqs \bar\gamma d_\beta(x,y)
\end{equation} 
pour tout $x, y\in\cX$. La relation est vraie pour $x=y$, donc nous supposons 
$x\neq y$. Nous consid\'erons deux cas s\'epar\'ement.
\begin{itemize}
\item   \textbf{Cas 1:} $V(x) + V(y) \geqs R$. Dans ce cas, on a 
\begin{equation}
\label{eq:borne_cas1} 
  \bigabs{(\cP f)(x)}
  = \biggabs{\sum_{y\in\cX} p_{xy} f(y)}
  \leqs \underbrace{\norm{f}_{1+\beta V}}_{\leqs 1} \sum_{y\in\cX} (1+\beta V(y)) p_{xy}
  \leqs 1 + \beta (\cP V)(x)\;, 
\end{equation} 
d'o\`u, par la condition~\eqref{eq:derive_geom} de d\'erive g\'eom\'etrique, 
\begin{align}
\bigabs{(\cP f)(x) - (\cP f)(y)}
&\leqs 2 + \beta (\cP V)(x) + \beta (\cP V)(y) \\
&\leqs 2 + \beta\gamma V(x) + \beta\gamma V(y) + 2\beta d\\
&\leqs 2 + \beta\gamma_0 [V(x) + V(y)]
\end{align}
pour tout $\gamma_0$ satisfaisant~\eqref{eq:cond_gamma0}. Un calcul \'el\'ementaire 
montre alors qu'en posant 
\begin{equation}
 \gamma_1 = \frac{2 + \beta R\gamma_0}{2+\beta R}\;,
\end{equation} 
on obtient bien la majoration requise
\begin{equation}
 \bigabs{(\cP f)(x) - (\cP f)(y)}
 \leqs \gamma_1 \bigpar{2 + V(x) + V(y)}
 = \gamma_1 d_\beta(x,y)\;.
\end{equation} 

\item   \textbf{Cas 2:} $V(x) + V(y) < R$. Dans ce cas, on a $x, y\in K$.
La matrice $\widetilde\cP$ d'\'el\'ements 
\begin{equation}
 \tilde p_{xy} = \frac{1}{1-\alpha} p_{xy} - \frac{\alpha}{1-\alpha} \nu(y)
\end{equation} 
est une matrice stochastique. En effet, la condition de minoration~\eqref{eq:minoration} 
montre que ces \'el\'ements sont tous positifs ou nuls, et on v\'erifie imm\'ediatement 
que la somme sur $y$ des $\tilde p_{xy}$ vaut $1$. De plus, on a 
\begin{equation}
 (\cP f)(x) = (1-\alpha) (\widetilde\cP f)(x) + \alpha \nu(f)\;,
\end{equation} 
d'o\`u, par un calcul analogue \`a~\eqref{eq:borne_cas1},  
\begin{align}
\bigabs{(\cP f)(x) - (\cP f)(y)}
&= (1-\alpha) \bigabs{(\widetilde\cP f)(x) - (\widetilde\cP f)(y)} \\
&\leqs (1-\alpha) \bigbrak{2 + \beta (\widetilde\cP V)(x) + \beta (\widetilde\cP V)(y)}\\
&\leqs (1-\alpha) + \beta (\cP V)(x) + \beta (\cP V)(y)\\
&\leqs 2(1-\alpha) + \beta\gamma [V(x) + V(y)] + 2\beta d\\
&\leqs \gamma_2 d_\beta(x,y)\;,
\end{align}
si $\beta = \frac{\alpha_0}{d}$ est donn\'e par~\eqref{eq:beta_gammabar} et 
$\gamma_2 = \max\set{\gamma, 1 - (\alpha-\alpha_0)}$. La troisi\`eme ligne suit du fait que 
\begin{equation}
 (\widetilde \cP V)(x) \leqs \frac{1}{1-\alpha} (\cP V)(x)\;.
\end{equation} 
\end{itemize}
Le r\'esultat suit, avec $\bar\gamma = \max\set{\gamma_1,\gamma_2}$. 
\end{proof}

Il nous reste \`a d\'emontrer le Lemme~\ref{lem:equiv_distances}.

\begin{proof}[\textit{D\'emonstration du Lemme~\ref{lem:equiv_distances}}]
D\'efinissons les boules 
\begin{equation}
 \cB = \setsuch{f\in\cE_\infty}{\norm{f}_{1+\beta V}\leqs 1}\;, 
 \qquad 
 \cB^* = \setsuch{f\in\cE_\infty}{\normDgamma{f}_{\beta}\leqs 1}\;.
\end{equation} 
Alors on peut \'ecrire 
\begin{equation}
 \rho_{1+\beta V}(\mu,\nu) = \sup_{f\in\cB} f(\mu - \nu)\;, 
 \qquad 
 \rho^*_{1+\beta V}(\mu,\nu) = \sup_{f\in\cB^*} f(\mu - \nu)\;.
\end{equation}
L'observation cruciale est la suivante. Pour $c\in\R$, soit $f+c$ la fonction 
translat\'ee de $c$, d\'efinie par $(f+c)(x) = f(x) + c$ pour tout $x\in\cX$. 
Alors on a 
\begin{equation}
 (f+c)(\mu-\nu)
 = \sum_{x\in\cX}(f(x)+c)(\mu(x)-\nu(x)) 
 = f(\mu-\nu) + c\sum_{x\in\cX}\mu(x) - c\sum_{x\in\cX}\nu(x)
 = f(\mu-\nu)\;, 
\end{equation} 
puisque $\mu$ et $\nu$ sont des mesures de probabilit\'e. 
Par cons\'equent, on a aussi 
\begin{equation}
 \rho_{1+\beta V}(\mu,\nu) = \sup_{f\in\widehat\cB} f(\mu - \nu)\;, 
 \qquad 
 \rho^*_{1+\beta V}(\mu,\nu) = \sup_{f\in\widehat\cB^*} f(\mu - \nu)\;, 
\end{equation}
o\`u on a d\'efini les \myquote{cylindres}\ 
\begin{equation}
 \widehat\cB = \setsuch{f+c}{f\in\cB, c\in\R}\;, 
 \qquad 
 \widehat\cB^* = \setsuch{f+c}{f\in\cB^*, c\in\R}\;.
\end{equation} 
Ces cylindres sont obtenus en \'etendant les boules $\cB$ et $\cB^*$ dans la direction 
des fonctions constantes. 

Nous allons montrer ci-dessous que 
\begin{equation}
\label{eq:triple_norm_inf} 
 \normDgamma{f}_\beta = \inf_{c\in\R} \norm{f+c}_{1+\beta V}\;.
\end{equation} 
Cela impliquera que $\widehat\cB = \widehat\cB^*$, et par cons\'equent que 
$\rho_{1+\beta V}(\mu,\nu) = \rho^*_{1+\beta V}(\mu,\nu)$, qui est l'\'egalit\'e voulue.

Comme remarqu\'e juste apr\`es la D\'efinition~\ref{def:norme_poids_fct_test}, 
on a, pour tout $x\in\cX$, 
\begin{equation}
 \abs{f(x)} \leqs \norm{f}_{1+\beta V}(1 + \beta V(x))\;.
\end{equation} 
Par cons\'equent, on a pour tout $x\neq y\in\cX$ 
\begin{equation}
 \frac{\abs{f(x)-f(y)}}{d_\beta(x,y)}
 \leqs \frac{\abs{f(x)} + \abs{f(y)}}{2 + \beta V(x) + \beta V(y)}
 \leqs \norm{f}_{1+\beta V}\;.
\end{equation} 
Ceci implique 
\begin{equation}
 \normDgamma{f}_\beta \leqs \norm{f}_{1+\beta V}\;.
\end{equation} 
Comme de plus la d\'efinition de $\normDgamma{f}_\beta$ ne d\'epend que 
de diff\'erences de $f$ en des points diff\'erents, on a 
$\normDgamma{f+c}_\beta = \normDgamma{f}_\beta$ pour tout $c\in\R$. On a donc 
obtenu 
\begin{equation}
 \normDgamma{f}_\beta \leqs \inf_{c\in\R} \norm{f+c}_{1+\beta V}\;.
\end{equation} 
Pour montrer l'in\'egalit\'e inverse, posons 
\begin{equation}
 c^* = \inf_{x\in\cX} \Bigpar{1 + \beta V(x) - f(x)}\;.
\end{equation} 
Commen\c cons par montrer que $\abs{c^*} < \infty$. 
Comme $c^* \leqs \bigpar{1 + \beta V(x_0) - f(x_0)}$ pour tout $x_0\in\cX$, 
il suffit pour cela de montrer que $c^*$ est born\'ee inf\'erieurement.
Cela suit du fait que pour tout $x,y\in\cX$, on a 
\begin{equation}
 f(x) \leqs \abs{f(y)} + \abs{f(x)-f(y)} 
 \leqs \abs{f(y)} + 2 + \beta V(x) + \beta V(y)\;,
\end{equation} 
qui implique 
\begin{equation}
 1 + \beta V(x) - f(x) \geqs -1 -\beta V(y) - \abs{f(y)}\;.
\end{equation} 
Comme $V(y)$ est finie pour au moins un $y\in\cX$, on a bien la minoration voulue. 

On observe maintenant que, d'une part, 
\begin{equation}
\label{eq:ff_bornesup} 
 f(x) + c^* \leqs f(x) + 1 + \beta V(x) - f(x)
 = 1 + \beta V(x)\;,
\end{equation} 
alors que d'autre part, 
\begin{align}
\label{eq:ff_borneinf} 
 f(x) + c^* 
 &= \inf_{y\in\cX} \bigpar{f(x) + 1 + \beta V(y) - f(y)} \\
 &\geqs \inf_{y\in\cX} \bigpar{1 + \beta V(y) 
 - \underbrace{\normDgamma{f}_\beta}_{= 1} d_\beta(x,y)} \\
 &\geqs \inf_{y\in\cX} \bigpar{1 + \beta V(y) - 2 - \beta V(x) - \beta V(y)} \\
 &\geqs -1 - \beta V(x)\;.
\end{align}
Il suit de~\eqref{eq:ff_bornesup} et~\eqref{eq:ff_borneinf} que 
\begin{equation}
 \abs{f(x) + c^*} \leqs 1+\beta V(x)\;,
\end{equation} 
d'o\`u l'in\'egalit\'e inverse 
\begin{equation}
 \inf_{c\in\R} \norm{f+c}_{1+\beta V} 
 \leqs \norm{f+c^*}_{1+\beta V}
 \leqs 1 = \normDgamma{f}_\beta\;.
\end{equation} 
Nous avons donc d\'emontr\'e~\eqref{eq:triple_norm_inf}, et par suite 
aussi~\eqref{eq:lem_equiv_distances}.
\end{proof}


\section{Exercices}
\label{sec:Lyapounov_exo} 

\begin{exercise}
On consid\`ere la marche al\'eatoire sym\'etrique sur $\Z$. 

\begin{enumerate}
\item   Calculer $(\cL V)(x)$ pour la fonction de Lyapounov $V(x) = x^2$. 
Expliciter la formule de Dynkin pour un temps $n$ d\'eterministe. 
Appliquer, si possible, les th\'eor\`emes de croissance sous-exponentielle, 
non-explosion, r\'ecurrence positive et de convergence.

\item   Calculer $(\cL V)(x)$ pour la fonction de Lyapounov $V(x) = \abs{x}$. 
Comme au point pr\'ec\'edent, expliciter la formule de Dynkin, et appliquer, 
si possible les diff\'erents th\'eor\`emes.

\item   Que se passe-t-il pour la marche al\'eatoire asym\'etrique~?
\end{enumerate} 
\end{exercise}

\begin{exercise}
Soit $p\in]0,1[$ et $q = 1 - p$. On consid\`ere la marche al\'eatoire sur $\N$ de probabilit\'es 
de transition 
\[
 p_{xy} = 
 \begin{cases}
  p & \text{si $y = x+1$\;,} \\
  q & \text{si $y = 0$\;,} \\
  0 & \text{sinon\;.}
 \end{cases}
\]
Calculer $(\cL V)(x)$ pour la fonction de Lyapounov $V(x) = \abs{x}$. 
Que peut-on en d\'eduire~? 
\end{exercise}

\begin{exercise}
Pour $x\in\N$, on d\'enote par 
\[
 \intpart{x} = \max\setsuch{y\in\N}{y\leqs x}
\]
la partie enti\`ere de $x$. 

Soit $p\in[0,1]$. On consid\`ere la cha\^ine de Markov sur $\N=\set{0,1,2,\dots}$ de probabilit\'es de transition 
\[
 p_{xy} = 
 \begin{cases}
  p & \text{si $y = x + 1$\;,}\\
  1 - p & \text{si $y = \intpart{x/2}$\;,}\\
  0 & \text{sinon\;.}
 \end{cases}
\]

Pour $\alpha\in\R$, on pose $V(x) = \e^{\alpha x}$.

\begin{enumerate}
\item   Pour quelles valeurs de $\alpha$ la fonction $V$ est-elle une 
fonction de Lyapounov~? On dira que ces valeurs de $\alpha$ sont \emph{admissibles}.

\item   Calculer $(\cL V)(x)$, o\`u $\cL$ est le g\'en\'erateur de la cha\^ine de Markov. On distinguera les cas $x$ pair et $x$ impair. 

\item   Pour quels $p$ la cha\^ine est-elle \`a croissance sous-exponentielle~?

\item   D\'eterminer une fonction $f_p(\alpha)$ telle que 
\[
 (\cL V)(x) \leqs - f_p(\alpha) V(x)
\]
pour tout $x\in\N^* = \set{1,2,\dots}$. 

\item   \'Etudier la fonction $\alpha\mapsto f_p(\alpha)$ pour les valeurs 
de $\alpha$ admissibles~: comportement aux 
bords du domaine, croissance/d\'ecroissance, convexit\'e. 

\item   Pour quelles valeurs de $p$ existe-t-il un $\alpha$ admissible 
tel que $f_p(\alpha) > 0$~?

\item   Pour quelles valeurs de $p$ peut-on affirmer l'existence d'une unique probabilit\'e 
invariante $\pi$ telle que la loi de $X_n$ converge exponentiellement vite vers $\pi$~? 
\end{enumerate}
\end{exercise}


\chapter{Algorithmes MCMC}
\label{chap:cm_MCMC} 


\section{M\'ethodes Monte Carlo}
\label{sec:MC} 

On appelle \defwd{m\'ethodes Monte Carlo}\/ un ensemble d'algorithmes
stochastiques, introduits dans les ann\'ees 1940 par le mathématicien 
polonais Stanis\l{}aw Ulam, permettant d'estimer num\'eriquement des grandeurs pouvant
\^etre consid\'er\'ees comme des esp\'erances. En voici pour
commencer un exemple tr\`es simple.

\begin{example}[Calcul d'un volume]
On aimerait calculer num\'eriquement le volume $\abs{V}$ d'un
sous-ensemble compact $V$ de $\R^N$. On suppose que $V$ est donn\'e par un
certain nombre $M$ d'in\'egalit\'es: 
\begin{equation}
\label{mcmc1}
V = \bigsetsuch{x\in\R^N}{f_1(x)\geqs0, \dots, f_M(x)\geqs0}\;. 
\end{equation}
Par exemple, si $M=1$ et $f_1(x)=1-\norm{x}^2$, alors $V$ est une boule.
Dans ce cas, bien s\^ur, le volume de $V$ est connu explicitement. On
s'int\'eresse \`a des cas o\`u $V$ a une forme plus compliqu\'ee, par
exemple une intersection d'un grand nombre de boules et de demi-espaces.
Dans la suite nous supposerons, sans limiter la g\'en\'eralit\'e, que $V$
est contenu dans le cube unit\'e $[0,1]^N$.
\end{example}

Une premi\`ere m\'ethode de calcul num\'erique du volume consiste \`a
discr\'etiser l'espace. Divisons le cube $[0,1]^N$ en cubes de cot\'e
$\eps$ (avec $\eps$ de la forme $1/K$, $K\in\N$). Le nombre total de
ces cubes est \'egal \`a $1/\eps^N=K^N$. On compte alors le nombre $n$ de
cubes dont le centre est contenu dans $V$, et le volume de $V$ est 
approximativement \'egal \`a $n\eps^N$. Plus pr\'ecis\'ement, on peut
encadrer $\abs{V}$ par $n_-\eps^N$ et $n_+\eps^N$, o\`u $n_-$ est le
nombre de cubes enti\`erement contenus dans $V$, et $n_+$ est le nombre de
cubes dont l'intersection avec $V$ est non vide (Figure~\ref{fig:discretisation}). Toutefois, effectuer ces tests n'est en g\'en\'eral pas ais\'e num\'eriquement.

Quelle est la pr\'ecision de l'algorithme~? Si le bord $\partial V$ est
raisonnablement lisse, l'erreur faite pour $\eps$ petit est de l'ordre de
la mesure $\abs{\partial V}$ du bord fois $\eps$. Pour calculer $\abs{V}$ avec une
pr\'ecision donn\'ee $\delta$, il faut donc choisir $\eps$ d'ordre
$\delta/\abs{\partial V}$. Cela revient \`a un nombre de cubes d'ordre 
\begin{equation}
\label{mcmc2}
\biggpar{\frac{\abs{\partial V}}{\delta}}^N\;,
\end{equation}
ou encore, comme on effectue $M$ tests pour chaque cube, \`a un nombre de
calculs d'ordre $(M\abs{\partial V}/\delta)^N$. Ce nombre ne pose pas de
probl\`eme pour les petites dimensions ($N=1,2$ ou $3$ par exemple), mais
cro\^\i t vite avec la dimension $N$. C'est ce qu'on appelle le \defwd{fl\'eau 
de la dimension}.

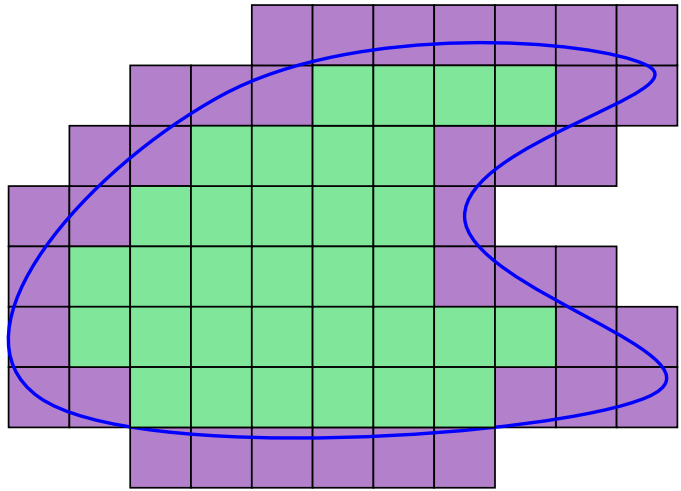
\begin{figure}
\vspace{-3mm}
\begin{center}
\scalebox{0.8}{
\begin{tikzpicture}[->,>=stealth',shorten >=2pt,shorten <=2pt,auto,node
distance=3.0cm,thick,
cont/.style={thick,rectangle,scale=1,minimum size=1cm,
fill=green!80!blue!50,draw,font=\sffamily\Large},
int/.style={thick,,rectangle,scale=1,minimum size=1cm,
fill=red!40!blue!50,draw,font=\sffamily\Large}]

\node[int] at (-5,0) {};
\node[int] at (-4,0) {};
\node[int] at (2,0) {};

\node[int] at (-5,-1) {};
\node[int] at (2,-1) {};
\node[int] at (3,-1) {};
\node[int] at (4,-1) {};

\node[int] at (-5,-2) {};
\node[int] at (4,-2) {};
\node[int] at (5,-2) {};

\node[int] at (-5,-3) {};
\node[int] at (-4,-3) {};
\node[int] at (4,-3) {};
\node[int] at (3,-3) {};
\node[int] at (5,-3) {};

\node[int] at (-3,-4) {};
\node[int] at (-2,-4) {};
\node[int] at (-1,-4) {};
\node[int] at (0,-4) {};
\node[int] at (1,-4) {};
\node[int] at (2,-4) {};

\node[int] at (-4,1) {};
\node[int] at (-3,1) {};
\node[int] at (2,1) {};
\node[int] at (3,1) {};
\node[int] at (4,1) {};

\node[int] at (-3,2) {};
\node[int] at (-2,2) {};
\node[int] at (-1,2) {};
\node[int] at (4,2) {};
\node[int] at (5,2) {};

\node[int] at (-1,3) {};
\node[int] at (0,3) {};
\node[int] at (1,3) {};
\node[int] at (2,3) {};
\node[int] at (3,3) {};
\node[int] at (4,3) {};
\node[int] at (5,3) {};

\draw[blue,ultra thick] plot[smooth cycle,tension=0.8]
  coordinates{(-5,-3) (5,-3) (2,0) (5,2.5) (-2,2)};

\node[cont] at (-3,0) {};
\node[cont] at (-2,0) {};
\node[cont] at (-1,0) {};
\node[cont] at (0,0) {};
\node[cont] at (1,0) {};

\node[cont] at (-2,1) {};
\node[cont] at (-1,1) {};
\node[cont] at (0,1) {};
\node[cont] at (1,1) {};

\node[cont] at (0,2) {};
\node[cont] at (1,2) {};
\node[cont] at (2,2) {};
\node[cont] at (3,2) {};

\node[cont] at (-4,-1) {};
\node[cont] at (-3,-1) {};
\node[cont] at (-2,-1) {};
\node[cont] at (-1,-1) {};
\node[cont] at (0,-1) {};
\node[cont] at (1,-1) {};

\node[cont] at (-4,-2) {};
\node[cont] at (-3,-2) {};
\node[cont] at (-2,-2) {};
\node[cont] at (-1,-2) {};
\node[cont] at (0,-2) {};
\node[cont] at (1,-2) {};
\node[cont] at (2,-2) {};
\node[cont] at (3,-2) {};

\node[cont] at (-3,-3) {};
\node[cont] at (-2,-3) {};
\node[cont] at (-1,-3) {};
\node[cont] at (0,-3) {};
\node[cont] at (1,-3) {};
\node[cont] at (2,-3) {};

\end{tikzpicture}
}
\end{center}
\vspace{-4mm}
 \caption[]{Calcul d'une aire par discr\'etisation. L'aire est encadr\'ee par 
 $n_-\eps^2$ et $n_+\eps^2$, o\`u $n_-$ est le nombre de carr\'es de c\^ot\'e $\eps$ enti\`erement contenus dans le domaine, et $n_+$ est le nombre de carr\'es intersectant le domaine.}
 \label{fig:discretisation}
\end{figure}

Une alternative int\'eressante pour les $N$ grands est fournie par
l'\emph{algorithme de Monte Carlo}. Dans ce cas, on g\'en\`ere une suite $X_1,
X_2, \dots, X_n, \dots$ de variables al\'eatoires ind\'epen\-dantes,
identiquement distribu\'ees (i.i.d.), de loi uniforme sur $[0,1]^N$. Ceci
est facile \`a r\'ealiser num\'eriquement, car on dispose de
g\'en\'erateurs de nombres (pseudo-)al\'eatoires distribu\'es
uniform\'ement sur $[0,1]$ (ou plut\^ot sur
$\set{0,1,\dots,n_{\text{max}}}$ o\`u $n_{\text{max}}$ est un entier du
genre $2^{31}-1$, mais en divisant ces nombres par $n_{\text{max}}$, on
obtient de bonnes approximations de variables uniformes sur $[0,1]$). Il
suffit alors de consid\'erer des $N$-uplets de tels nombres. 

Consid\'erons alors les variables al\'eatoires i.i.d. 
\begin{equation}
\label{mcmc3}
Y_n = \indicator{X_n\in V}\;, 
\qquad n = 1,2,\dots\;.
\end{equation}
On aura 
\begin{equation}
\label{mcmc4}
\expec{Y_n} = \bigprob{X_n\in V} = \abs{V}\;.
\end{equation}
Les moyennes empiriques
\begin{equation}
\label{mcmc5}
S_n = \frac1n \sum_{m=1}^n Y_m
\end{equation}
ont esp\'erance $\expec{S_n}=\abs{V}$ et variance
$\variance{S_n}=\variance{Y_1}/n$. La loi faible des grands nombres
implique que 
\begin{equation}
\label{mcmc6}
\lim_{n\to\infty} \Bigprob{\bigabs{S_n-\expec{S_n}} > \delta} = 0
\end{equation}
pour tout $\delta>0$. Par cons\'equent, $S_n$ devrait donner une bonne
approximation du volume $\abs{V}$ lorsque $n$ est suffisamment grand.\footnote{La
loi forte des grands nombres affirme par ailleurs que $S_n$ tend vers
$\abs{V}$ presque s\^urement, c'est-\`a-dire que $S_n$ n'est plus vraiment
al\'eatoire \`a la limite.} 

Pour savoir comment choisir $n$ en fonction de la pr\'ecision d\'esir\'ee,
il nous faut estimer la probabilit\'e que $\abs{S_n-\abs{V}} > \delta$,
pour $n$ grand mais fini. Une premi\`ere estimation est fournie par
l'in\'egalit\'e de Bienaym\'e--Chebychev, qui affirme que 
\begin{equation}
\label{mcmc7}
\Bigprob{\bigabs{S_n-\expec{S_n}} > \delta} \leqs 
\frac{\variance(S_n)}{\delta^2}
= \frac{\variance(Y_1)}{\delta^2 n}
< \frac{1}{\delta^2 n}\;,
\end{equation}
o\`u nous avons utilis\'e le fait que
$\variance(Y_1)\leqs\expec{Y_1^2}\leqs 1$. On peut donc affirmer que pour
que la probabilit\'e de faire une erreur sup\'erieure \`a $\delta$ soit
inf\'erieure \`a $\eps$, il suffit de choisir 
\begin{equation}
\label{mcmc8}
n > \frac1{\delta^2\eps}\;.
\end{equation}
Comme pour chaque $m$, il faut g\'en\'erer $N$ variables al\'eatoires, et
effectuer $M$ tests, le nombre de calculs n\'ecessaires est d'ordre
$MN/(\delta^2\eps)$. L'avantage de cette m\'ethode est que ce nombre ne
cro\^\i t que lin\'eairement avec $N$, par opposition \`a la croissance
exponentielle dans le cas de la discr\'etisation. On notera toutefois que
contrairement \`a la discr\'etisation, qui donne un r\'esultat certain aux
erreurs pr\`es, la m\'ethode de Monte Carlo ne fournit que des r\'esultats
vrais avec tr\`es grande probabilit\'e (d'o\`u son nom).

\begin{remark}[Estimation d'erreur am\'elior\'ee]
L'estimation~\eqref{mcmc7} est assez pessimiste, et peut \^etre 
consid\'erablement am\'elior\'ee. Par exemple, le th\'eor\`eme
central limite montre que 
\begin{equation}
\label{mcmc9}
\lim_{n\to\infty}
\Biggprob{\frac{(S_n-\expec{S_n})^2}{\variance(S_n)} > \eta^2}
= \int_{\abs{x}>\eta} \frac{\e^{-x^2/2}}{\sqrt{2\pi}} \6x\;,
\end{equation}
dont le membre de droite d\'ecro\^\i t comme $\e^{-\eta^2/2}$ pour $\eta$
grand. Cela indique que pour $n$ grand, 
\begin{equation}
\label{mcmc10}
\Bigprob{\abs{S_n-\abs{V}} > \delta} \simeq
\e^{-n\delta^2/2\variance(Y_1)}\;.
\end{equation}
Ceci permet d'am\'eliorer le crit\`ere~\eqref{mcmc8} en 
\begin{equation}
\label{mcmc11}
n > \const \frac{\log(1/\eps)}{\delta^2}\;,
\end{equation}
d'o\`u un nombre de calculs d'ordre $NM\log(1/\eps)/\delta^2$. Cette
estimation n'est pas une borne rigoureuse, contrairement \`a~\eqref{mcmc8},
parce qu'on n'a pas tenu compte de la vitesse de convergence
dans~\eqref{mcmc9}, qui par ailleurs ne s'applique que pour $\eta$
ind\'ependant de $\eps$. On devrait plut\^ot utiliser des estimations
provenant de la th\'eorie des grandes d\'eviations, que nous n'aborderons
pas ici. Les r\'esultats sont toutefois qualitativement corrects.
\end{remark}

\begin{example}[Estimation d'un volume]
A titre d'illustration, supposons qu'on veuille d\'eterminer le volume
d'un domaine de dimension $N=1000$, d\'efini par $M=10$ in\'egalit\'es,
avec une pr\'ecision de $\delta=10^{-4}$. La m\'ethode de discr\'etisation
n\'ecessite un nombre de calculs d'ordre $10^{5000}$, ce qui est
irr\'ealisable avec les ordinateurs actuels. La m\'ethode de Monte Carlo,
en revanche, fournit un r\'esultat de la m\^eme pr\'ecision, s\^ur avec
probabilit\'e $1-10^{-6}$, avec un nombre de calculs d'ordre
$\log(10^6)\cdot 10^{12} \simeq 10^{13}$, ce qui ne prend que quelques
minutes sur un PC. La Table~\ref{tab:MC} compare des co\^uts pour 
diff\'erentes valeurs de $N$. 
\end{example}

\begin{table}[ht]
\begin{center}
 \begin{tabular}{|r|r|r|r|}
\hline
\myvrule{13pt}{5pt}{0pt}
$N$ & Discr\'etisation & Monte Carlo BC & Monte Carlo TCL \\
\hline
\myvrule{13pt}{5pt}{0pt}
$1$ & $10^5$ & $10^{15}$ & $1,\!4\cdot 10^{10}$ \\
$2$ & $10^9$ & $2\cdot10^{15}$ & $2,\!8\cdot 10^{10}$ \\
$3$ & $10^{13}$ & $3\cdot10^{15}$ & $4,\!2\cdot 10^{10}$ \\
$10$ & $10^{41}$ & $10^{16}$ & $1,\!4\cdot 10^{11}$ \\
$100$ & $10^{401}$ & $10^{17}$ & $1,\!4\cdot 10^{12}$ \\
$1000$ & $10^{4001}$ & $10^{18}$ & $1,\!4\cdot 10^{13}$ \\
\hline
\end{tabular}
\end{center}
\vspace{-3mm}
\caption[]{Comparaison, pour diff\'erentes dimensions $N$ et un nombre $M=10$ d'in\'egalit\'es, des co\^uts de calcul d'un volume avec pr\'ecision $\delta = 10^{-4}$, par discr\'etisation ($10^{4N+1}$), par la m\'ethode de Monte Carlo avec $\eps=10^{-6}$ avec l'estimation de Bienaym\'e--Chebychev ($10^{15} N$), et par la m\^eme m\'ethode avec 
l'estimation bas\'ee sur le th\'eor\`eme central limite ($\log(10^6)\cdot 10^9 N$).}
\label{tab:MC} 
\end{table}

La m\'ethode de Monte Carlo se g\'en\'eralise facilement \`a d'autres
probl\`emes que des calculs de volume. Supposons par exemple donn\'e un
espace probabilis\'e $(\Omega,\cF,\pi)$, et une variable al\'eatoire
$Y:\Omega\to\R$. On voudrait estimer l'esp\'erance de $Y$. Pour cela,
l'algorithme de Monte Carlo consiste \`a g\'en\'erer des variables
al\'eatoires ind\'ependantes $Y_1, Y_2, \dots, Y_n, \dots$, toutes de loi
$\pi Y^{-1}$, et de calculer leur moyenne. Cette moyenne doit converger
vers l'esp\'erance cherch\'ee (pourvu que $Y$ soit int\'egrable). 

Cet algorithme n'est toutefois efficace que si l'on arrive \`a g\'en\'erer
les variables al\'eatoires $Y_i$ de mani\`ere efficace. Une fois de plus,
ceci est relativement facile en dimension faible, mais devient rapidement
difficile lorsque la dimension cro\^\i t. 

\begin{remark}[Cas unidimensionnel]
\label{rem:mcmc2}
Une variable al\'eatoire uni\-dimensionnelle $Y$ s'ob\-tient facilement \`a
partir d'une variable de loi uniforme. Soit en effet $U$ une variable
uniforme sur $[0,1]$. Sa fonction de r\'epartition est donn\'ee par 
\begin{equation}
\label{mcmc12}
F_U(u) = \prob{U\leqs u} = u 
\qquad
\text{pour $0\leqs u\leqs 1$.} 
\end{equation}
On cherche une fonction $\varphi$ telle que la variable $Y=\varphi(u)$
admette la fonction de r\'eparti\-tion prescrite $F_Y(y)$. Or on a 
\begin{equation}
\label{mcmc13}
F_Y(y) = \prob{Y\leqs y} = \prob{\varphi(U)\leqs y} 
= \prob{U \leqs \varphi^{-1}(y)} = \varphi^{-1}(y)\;.
\end{equation}
Il suffit donc de prendre $Y = F_Y^{-1}(U)$. 
\end{remark}

\begin{example}[Loi exponentielle]
Par exemple, pour g\'en\'erer une variable de loi exponentielle, dont la
fonction de r\'eparti\-tion est donn\'ee par $F_Y(y)=1-\e^{-\lambda y}$, il
suffit de prendre 
\begin{equation}
\label{mcmc14}
Y = -\frac1\lambda \log(1-U)\;.
\end{equation}
\end{example}

\begin{example}[Loi normale -- Algorithme de Box--Muller]
Pour la loi normale, cette m\'ethode n\'ecessiterait le calcul approch\'e
de sa fonction de r\'epartition, ce qui est num\'eriquement peu efficace. 
Il existe toutefois une alternative permettant d'\'eviter ce calcul.
Soient en effet $U$ et $V$ deux variables al\'eatoires ind\'ependantes, de
loi uniforme sur $[0,1]$. On introduit successivement les variables 
\begin{align}
\nonumber
R &= \sqrt{-2\log(1-U)}\;,
& 
Y_1 &= R\cos\Phi\;, \\
\Phi &= 2\pi V\;,
&
Y_2 &= R\sin\Phi\;.
\label{mcmc15}
\end{align}
Alors $Y_1$ et $Y_2$ sont des variables al\'eatoires ind\'ependantes, de
loi normale centr\'ee r\'eduite. Pour le voir, on v\'erifie d'abord que
$R$ a la fonction de r\'epartition $1-\e^{-r^2/2}$, donc la densit\'e
$r\e^{-r^2/2}$. Le couple $(R,\Phi)$ a donc la densit\'e jointe
$r\e^{-r^2/2}/(2\pi)$, et la formule de changement de variable montre que
le couple $(Y_1,Y_2)$ a la densit\'e jointe
$\e^{-(y_1^2+y_2^2)/2}/(2\pi)$, qui est bien celle d'un couple de
variables normales centr\'ees r\'eduites ind\'ependantes.
\end{example}

Bien entendu, les esp\'erances de variables al\'eatoires de loi
unidimensionnelle sont soit connues explicitement, soit calculables
num\'eriquement par la simple estimation d'une int\'egrale. Nous nous
int\'eressons ici \`a des situations o\`u la loi de $Y$ n'est pas aussi
simple \`a repr\'esenter. 


\section{M\'ethodes Monte Carlo par cha\^ines de Markov}
\label{sec:MCMC} 

Consid\'erons le cas d'un espace probabilis\'e discret
$(\cX,\cP(\cX),\pi)$, o\`u $\cX$ est un ensemble d\'enombrable, mais
tr\`es grand. Par exemple, dans le cas du mod\`ele d'Ising (voir la 
section~\ref{sec:ex_Ising}), $\cX=\set{-1,1}^N$ est de cardinal $2^N$, et
on s'int\'eresse \`a des $N$ grands, par exemple d'ordre $1000$. La mesure de
probabilit\'e $\pi$ est dans ce cas une application de $\cX$ vers $[0,1]$
telle que la somme des $\pi(x)$ vaut $1$. 

On voudrait estimer l'esp\'erance d'une variable al\'eatoire $Y:\cX\to\R$,
comme par exemple l'aimantation dans le cas du mod\`ele d'Ising~: 
\begin{equation}
\label{mcmc16}
\expecin{\pi}{Y} = \sum_{x\in\cX} Y(x)\pi(x)\;.
\end{equation}
La m\'ethode de Monte Carlo usuelle consiste alors \`a g\'en\'erer une
suite de variables al\'eatoires $X_0, X_1, \dots$ sur $\cX$,
ind\'ependantes et de loi $\pi$, puis de calculer la moyenne des
$Y(X_m)$. 

Pour g\'en\'erer ces $X_m$, on pourrait envisager de proc\'eder comme suit~: on
d\'efinit un ordre total sur $\cX$, et on d\'etermine la fonction de
r\'epartition 
\begin{equation}
\label{mcmc17}
x \mapsto
F_\pi(x) = \sum_{y\leqs x} \pi(y)\;.
\end{equation}
Si $U$ est une variable de loi uniforme sur $[0,1]$, alors $F_\pi^{-1}(U)$
suit la loi $\pi$. Toutefois, en proc\'edant de cette mani\`ere, on ne
gagne rien, car le calcul des sommes~\eqref{mcmc17} est aussi long que le
calcul de la somme~\eqref{mcmc16}, que l'on voulait pr\'ecis\'ement
\'eviter~!

Les m\'ethodes MCMC (pour \emph{Monte Carlo Markov Chain}\/) \'evitent cet
inconv\'enient. L'id\'ee est de simuler \emph{en m\^eme temps}\/ la
loi $\pi$ et la variable al\'eatoire $Y$, \`a l'aide d'une \CM\ sur $\cX$, 
de probabilit\'e invariante $\pi$. 

Soit donc $(X_n)_{n\in\N}$ une telle cha\^ine, suppos\'ee
irr\'eductible, r\'ecurrente positive, ap\'eriodique, et de loi initiale 
$\nu$ arbitraire. On lui associe une suite $Y_n=Y(X_n)$ de variables al\'eatoires.
Celles-ci peuvent se d\'ecomposer comme suit~: 
\begin{equation}
\label{mcmc18}
Y_n = \sum_{x\in\cX} Y(x) \indicator{X_n=x}\;.
\end{equation}
Consid\'erons les moyennes empiriques 
\begin{equation}
\label{mcmc19}
S_n = \frac1n \sum_{m=0}^{n-1} Y_m\;. 
\end{equation}
Le Th\'eor\`eme~\ref{thm:convergence_aperiodique} permet d'\'ecrire 
\begin{align}
\nonumber
\lim_{n\to\infty} \bigexpecin{\nu}{S_n} 
&= \lim_{n\to\infty} \frac1n \biggexpecin{\nu}{\sum_{m=0}^{n-1}Y_m} \\
\nonumber
&= \sum_{x\in\cX} Y(x) \lim_{n\to\infty} \frac1n 
\biggexpecin{\nu}{\sum_{m=0}^{n-1}\indicator{X_m=x}} \\
\nonumber
&= \sum_{x\in\cX} Y(x) \pi(x) \\
&= \expecin{\pi}{Y}\;.
\label{mcmc20}
\end{align}
L'esp\'erance de $S_n$ converge bien vers l'esp\'erance cherch\'ee. Pour
pouvoir appliquer l'id\'ee de la m\'ethode Monte Carlo, il nous faut plus,
\`a savoir la convergence (au moins) en probabilit\'e de $S_n$ vers
$\expec{Y}$. On ne peut pas invoquer directement la loi des grands nombres,
ni le th\'eor\`eme central limite, car les $Y_n$ ne sont plus
ind\'ependants. Mais il s'av\`ere que des r\'esultats analogues restent
vrais dans le cas de cha\^ines de Markov. 

\begin{theorem}[R\'eduction de variance partant de la probabilit\'e invariante]
Supposons la cha\^ine\ r\'eversible, et de loi initiale \'egale
\`a sa probabilit\'e invariante. Soit $\rho$ le rayon spectral associ\'e \`a la cha\^ine. Alors 
\begin{equation}
\label{mcmc21}
\variance(S_n) \leqs \frac1n
\Biggpar{\frac{1+\rho}{1-\rho}} \variance^\pi(Y)\;.
\end{equation}
\end{theorem}
\begin{proof}
Comme la cha\^ine\ d\'emarre dans la probabilit\'e invariante $\pi$, tous
les $Y_i$ ont m\^eme loi $\pi Y^{-1}$, m\^eme s'ils ne sont pas
ind\'ependants. Il suit que 
\begin{align}
\nonumber
\variance(S_n) &= \frac1{n^2} 
\Biggbrak{\sum_{m=0}^{n-1}\variance(Y_m) + 2\sum_{0\leqs p<q<n}
\cov(Y_p,Y_q)} \\
&= \frac1n \variance^\pi(Y) + \frac2{n^2} \sum_{m=1}^{n-1} (n-m)
\cov(Y_0,Y_m)\;,
\label{mcmc22}
\end{align}
en vertu du fait que $(Y_p,Y_q)$ a la m\^eme loi que $(Y_0,Y_{q-p})$. Or si 
$\vone=\smash{\transpose{(1,1,\dots,1)}}$ on a 
\begin{align}
\nonumber
\cov(Y_0,Y_m) 
&= \Bigexpecin{\pi}{(Y_0-\expecin{\pi}{Y_0}) (Y_m-\expecin{\pi}{Y_m})} \\
\nonumber
&= \sum_{x\in\cX} \sum_{y\in\cX} \bigpar{Y(x)-\expecin{\pi}{Y}}
\bigpar{Y(y)-\expecin{\pi}{Y}} 
\underbrace{\probin{\pi}{X_0=x,X_m=y}}_{=\pi(x)(P^m)_{xy}}\\
\nonumber
&= \sum_{x\in\cX} \pi(x)\bigpar{Y(x)-\expecin{\pi}{Y}}
\bigbrak{P^m(Y-\expecin{\pi}{Y}\vone)}_x\\
\nonumber
&= \pscal{Y-\expecin{\pi}{Y}\vone}{P^m(Y-\expecin{\pi}{Y}\vone)}_\pi\\
&\leqs \rho^m
\pscal{Y-\expecin{\pi}{Y}\vone}{Y-\expecin{\pi}{Y}\vone}_\pi
= \rho^m \variance^\pi(Y)\;.
\label{mcmc23}
\end{align}
Dans l'in\'egalit\'e \`a la derni\`ere ligne, nous avons utilis\'e le fait que 
$Y-\expecin{\pi}{Y}\vone\in \vone_\perp$ puisque la somme des $\pi(x)(Y(x)-\expecin{\pi}{Y})$ 
est nulle, et que par cons\'equent ce vecteur se trouve dans le sous-espace
compl\'ementaire au vecteur propre $\vone$. Le r\'esultat suit alors en
rempla\c cant dans~\eqref{mcmc22}, en majorant $n-m$ par $n$ et en sommant
une s\'erie g\'eom\'etrique. 
\end{proof}

Il suit de cette estimation et de l'in\'egalit\'e de Bienaym\'e--Chebychev
que pour calculer $\expecin{\pi}{Y}$ avec une pr\'ecision $\delta$ et avec
probabilit\'e $1-\eps$, il faut choisir 
\begin{equation}
\label{mcmc24}
n \geqs 
\frac{\variance^\pi(Y)}{\delta^2\eps}
\biggpar{\frac{1+\rho}{1-\rho}}\;.
\end{equation}
En pratique, on ne peut pas faire d\'emarrer la cha\^ine\ exactement avec
la probabilit\'e invariante. Ceci conduit \`a une convergence un peu plus
lente, mais du m\^eme ordre de grandeur puisque la loi des $Y_n$ converge
exponentiellement vite vers $\pi Y^{-1}$. Les r\'esultats sont bien s\^ur
meilleurs si on choisit bien la condition initiale, c'est-\`a-dire de
mani\`ere \`a ce que la loi des $Y_n$ converge rapidement. 


\section{Algorithmes de type Metropolis}
\label{sec:Metropolis} 

Nous avons vu comment estimer l'esp\'erance d'une variable al\'eatoire
$Y$ \`a l'aide d'une cha\^ine de Markov de probabilit\'e invariante
donn\'ee par la loi de $Y$. Pour que cet algorithme soit efficace, il faut
encore que l'on puisse trouver facilement, en fonction de cette loi, une 
matrice de transition donnant la probabilit\'e invariante souhait\'ee. 

Une m\'ethode pour le faire a \'et\'e d\'evelopp\'ee \`a Los Alamos au 
d\'ebut des ann\'ees 1950, par Nicholas Metropolis, Arianna Rosenbluth, 
Marshall Rosenbluth, Augusta Teller et Edward Teller (plus connu pour le 
d\'eveloppement de la bombe \`a hydrog\`ene)\footnote{Il semble par ailleurs 
que la contribution principale de Nicholas Metropolis ait \'et\'e de mettre 
\`a disposition du temps de calcul sur l'ordinateur MANIAC qu'il g\'erait.}. 
L'algorithme devrait donc \^etre appel\'e 
Metropolis--Rosenbluth--Rosenbluth--Teller--Teller, mais 
il est plus connu sous le nom d'algorithme de Metropolis, ou de 
Metropolis--Hastings, pour une forme plus g\'en\'erale d\'evelopp\'ee 
par la suite par Wilfred Keith Hastings. 

Le but de cet algorithme est d'\'echantillonner une \defwd{mesure de Gibbs}, 
de la forme 
\begin{equation}
\label{metro1}
\pi(x) = \frac{\e^{-\beta H(x)}}{Z_\beta}\;,
\qquad
\text{o\`u }
Z_\beta = \sum_{x\in\cX}\e^{-\beta H(x)}\;.
\end{equation}
Le param\`etre $\beta$ d\'esigne la temp\'erature inverse du syst\`eme, et 
la fonction $H: \cX\to\R$ associe \`a toute configuration $x\in\cX$ son \'energie. 
Nous en avons vu un exemple dans la section~\ref{sec:ex_Ising} avec le mod\`ele 
d'Ising. Il s'agit donc de construire une \CM\ sur $\cX$ admettant $\pi$ comme 
probabilit\'e invariante. Une mani\`ere simple d'approcher ce probl\`eme 
est de chercher une \CM\ r\'eversible. On cherche donc une matrice de 
transition $P$ sur $\cX$ dont les \'el\'ements satisfont 
\begin{equation}
\label{metro4}
\pi(x)p_{xy} = \pi(y)p_{yx}
\end{equation}
pour toute paire $(x,y)\in\cX\times\cX$. Cela revient \`a
imposer que 
\begin{equation}
\label{metro5}
\frac{p_{xy}}{p_{yx}}
= \e^{-\beta\Delta H(x,y)}\;,
\end{equation}
o\`u 
\begin{equation}
 \Delta H(x,y) = H(y) - H(x)
\end{equation} 
est la diff\'erence d'\'energie entre les \'etats $y$ et $x$. On notera que
cette condition ne fait pas intervenir la constante de normalisation
$Z_\beta$, ce qui est souhaitable, car le calcul de cette constante est
aussi co\^uteux que celui de $\expecin{\pi}{Y}$. 

L'algorithme de Metropolis consiste dans un premier temps \`a
d\'efinir un ensemble de transitions permises, c'est-\`a-dire une relation
sym\'etrique $\sim$ sur $\cX$ (on supposera toujours que
$x\not\sim x$). Une fois la relation $\sim$ fix\'ee, on choisit des probabilit\'es de
transition telles que 
\begin{equation}
\label{metro6}
p_{xy} = 
\begin{cases}
\myvrule{10pt}{14pt}{0pt} 
p_{yx} \e^{-\beta\Delta H(x,y)} 
&\text{si $x\sim y$\;,}\\
\displaystyle
1 - \sum_{z\sim x} p_{xz}
&\text{si $x=y$\;,}\\
0
&\text{autrement\;.}
\end{cases}
\end{equation}
On remarque que la \CM\ est irr\'eductible \`a condition que la relation $\sim$ 
le soit (deux \'etats quelconques de $\cX$ peuvent \^etre reli\'es par un 
chemin d'\'etats \'equivalents par $\sim$). De plus, la \CM\ est ap\'eriodique 
si $p_{xx} > 0$ pour tout $x\in\cX$. Si $\cX$ est fini, la cha\^ine est automatiquement 
r\'ecurrente positive. 
Pour satisfaire la condition de r\'eversibilit\'e~\eqref{metro4} 
lorsque $x\sim y$, une possibilit\'e est de prendre
\begin{equation}
\label{metro7}
p_{xy} = 
\begin{cases}
q 
&\text{si $H(y)\leqs H(x)$\;,}\\
q 
\e^{-\beta\Delta H(x,y)}
&\text{si $H(y)> H(x)$\;,}
\end{cases}
\end{equation}
o\`u $q\in]0,1]$ est une constante qui contr\^ole la vitesse de l'algorithme. Elle doit \^etre choisie assez petite pour que $p_{xx}$ soit positif. Ce choix revient \`a effectuer la transition avec probabilit\'e $q$ si elle d\'ecro\^\i t l'\'energie, et de ne l'effectuer qu'avec probabilit\'e $q\e^{-\beta\Delta H(x,y)}$ si elle fait cro\^\i tre l'\'energie. Une autre possibilit\'e est de choisir 
\begin{equation}
\label{metro8}
p_{xy} = \frac{q}{1+\e^{\beta\Delta
H(x,y)}}\;.
\end{equation}

\begin{remark}[$q$ non constant]
Au lieu de choisir un $q$ constant, on peut \'egalement choisir des coefficients 
$q_{xy}$ d\'ependant de $x$ et $y$, satisfaisant $q_{xy} = q_{yx}$, et avec 
$\sum_{y\sim x}q_{xy}$ assez petit pour avoir $p_{xx} > 0$. Cela peut permettre, 
dans certains cas, d'acc\'el\'erer la convergence de l'algorithme. 
\end{remark}

Nous allons illustrer cette m\'ethode dans le cas du mod\`ele d'Ising, mais
on voit facilement comment la g\'en\'eraliser \`a d'autres syst\`emes. 
Rappelons que dans le cas du mod\`ele d'Ising (voir la section~\ref{sec:ex_Ising}), 
l'univers est donn\'e par $\cX=\set{-1,1}^\Lambda$, o\`u $\Lambda$ est un sous-ensemble
(suppos\'e ici de cardinal fini $N$) de $\Z^d$. L'\'energie est donn\'ee par 
\begin{equation}
\label{metro2}
H(x) = -\sum_{i,j\in\Lambda\colon\norm{i-j}=1}x_ix_j 
- h \sum_{i\in\Lambda} x_i\;,
\end{equation}
o\`u $h\in\R$ est le champ magn\'etique. L'objectif est de calculer
l'esp\'erance de la variable aimantation, donn\'ee par
\begin{equation}
\label{metro3}
m(x) = \sum_{i\in\Lambda} x_i\;.
\end{equation}
Le choix de la relation sym\'etrique $\sim$ sur $\cX$ d\'epend de la physique 
que l'on souhaite mod\'eliser. Les deux choix les plus courants sont 
\begin{itemize}
\item	la \defwd{dynamique de Glauber}\/, qui consiste \`a choisir
$x\sim y$ si et seulement si les deux configurations $x$ et
$y$ diff\`erent en exactement un point de $\Lambda$; on parle de dynamique
de renversement de spin;

\item	la \defwd{dynamique de Kawasaki}\/, qui consiste \`a choisir
$x\sim y$ si et seulement si $y$ est obtenue en
intervertissant deux composantes de $x$; on parle de dynamique
d'\'echange de spin. Dans ce cas, la cha\^ine n'est pas irr\'eductible sur
$\cX$, car elle conserve le nombre total de spins $+1$ et $-1$ : elle
est en fait irr\'eductible sur chaque sous-ensemble de configurations \`a
nombre fix\'e de spins de chaque signe.
\end{itemize}

Remarquons que le calcul de la diff\'erence d'\'energie $\Delta H$ est
particuli\`erement simple dans le cas de la dynamique de Glauber, car
seuls le spin que l'on renverse et ses voisins entrent en compte. Ainsi,
si $R_k(x)$ d\'enote la configuration obtenue en renversant le spin
num\'ero $k$ de $x$, on aura 
\begin{equation}
\label{metro9}
\Delta H(x,R_k(x)) = 2x_k 
\Biggbrak{\sum_{\;j\colon\norm{j-k}=1}x_j + h}\;,
\end{equation}
qui est une somme de $2d+1$ termes pour un r\'eseau $\Lambda\subset\Z^d$.

Concr\`etement, l'algorithme de Metropolis avec dynamique de Glauber
s'impl\'emente de la mani\`ere suivante (avec $q = 1/N$)~: 

\begin{mdframed}[innerleftmargin=7mm,innertopmargin=10pt,innerbottommargin=10pt]
\begin{enumerate}
\item	{\bf \'Etape d'initialisation~:}
\begin{itemize}
\item	choisir une configuration initiale $X_0$ (de pr\'ef\'erence 
telle que $\pi(X_0)$ ne soit pas trop petit);
\item	calculer $m_0=m(X_0)$ (n\'ecessite $N$ calculs);
\item	calculer $H(X_0)$ (n\'ecessite de l'ordre de $dN$ calculs);
\item	poser $S = m_0$.
\end{itemize}

\item	{\bf Etape d'it\'eration~:} Pour $n=0, 1, \dots, n_{\max}-1$, 
\begin{itemize}
\item	choisir un spin $k$ au hasard uniform\'ement dans $\Lambda$;
\item	calculer $\Delta H(X_n,y)$, o\`u
$y=R_k(X_n)$ est obtenu en renversant le spin choisi;
\item	si $\Delta H(X_n,y) \leqs 0$, poser $X_{n+1} = y$; 
\item	si $\Delta H(X_n,y) > 0$, 
soit $B_n$ une variable de Bernoulli de param\`etre 
$\e^{-\beta\Delta H(X_n,y)}$, qui est ind\'ependante des autres al\'eas; 
poser $X_{n+1}=y$ si $B_n=1$, et $X_{n+1}=X_n$ si $B_n=0$;
\item	si $B_n=1$ (on a renvers\'e le spin $k$), alors 
$m_{n+1}=m_n+2(X_n)_k$, sinon $m_{n+1}=m_n$;
\item	ajouter $m_{n+1}$ \`a $S$.
\end{itemize}
\end{enumerate}
\end{mdframed}

Le quotient $S/(n+1)$ converge alors vers $\expecin{\pi}{m}$, avec une vitesse
d\'etermin\'ee par~\eqref{mcmc21}. La seule quantit\'e difficile \`a
estimer est le trou spectral $1-\rho$. Nous allons donner une exemple de son 
estimation \`a l'aide de la m\'ethode des fonctions de Lyapounov dans la 
section suivante. 


\section{Mod\`ele d'Ising sur le cercle discret}
\label{sec:MCMC_conv} 

Nous donnons dans cette section quelques exemples d'estimations de vitesse 
de convergence pour le mod\`ele d'Ising sur $\Lambda = \Z/N\Z$,  
pour la dynamique de Glauber et une matrice de transition d\'efinie 
par~\eqref{metro7}. Cela signifie que l'on consid\`ere $N$ spins align\'es, 
avec conditions aux bords p\'eriodiques~: on identifie $i=0$ avec $i=N$
et $i=-1$ avec $i=N-1$, de mani\`ere que chaque spin $i\in\Lambda$ ait exactement deux 
voisins $i-1$ et $i+1$. Soit $N_\pm(x)$ le nombre de spins de la configuration 
$x$ valant $\pm1$. Alors on a 
\begin{equation}
\label{eq:Nplus_Nminus} 
 N_+(x) + N_-(x) = N\;, \qquad 
 N_+(x) - N_-(x) = m(x)\;,
\end{equation} 
ce qui implique $m(x) = N - 2N_-(x)$. 
Soit par ailleurs 
\begin{equation}
 I(x) = \bigabs{\bigsetsuch{i\in\Lambda}{x_{i+1}\neq x_i}}
\end{equation} 
le nombre d'\myquote{interfaces}\ de $x$, c'est-\`a-dire 
le nombre de fois que la fonction $i\mapsto x_i$ change de signe en faisant le tour 
du cercle discret. Alors on a 
\begin{equation}
 \sum_{i,j\in\Lambda\colon\norm{i-j}=1}x_ix_j 
 = \sum_{i\in\Lambda} x_ix_{i+1} 
 = N - 2I(x)\;.
\end{equation} 
Par cons\'equent, l'\'energie du mod\`ele d'Ising peut \'egalement s'\'ecrire 
\begin{align}
 H(x) &= 2I(x) - hm(x) - N\\
 &= 2I(x) + 2hN_-(x) - N(1+h)\;.
\end{align} 
Les constantes $-N$ et $-N(1+h)$ n'ont pas d'incidence sur la dynamique de Glauber, 
qui ne fait intervenir que des diff\'erences d'\'energie entre configurations. 
D\'enotons par $\boxplus$ la configuration dont tous les spins valent $+1$, 
et par $\boxminus$ celle dont tous les spins valent $-1$. Alors on a 
\begin{equation}
 H(\boxplus) = -N(1+h)\;, \qquad 
 H(\boxminus) = -N(1-h)\;.
\end{equation} 
Dans la suite, on supposera que $N$ est pair, et que $0 < h\leqs 1$. Dans ce cas, 
la configuration d'\'energie minimale est $\boxplus$, et on v\'erifie qu'on a 
deux configurations d'\'energie maximale \'egale \`a $N$, donn\'ees par 
\begin{equation}
 (1,-1,1,-1,\dots)
 \qquad \text{et} \qquad 
 (-1,1,-1,1,\dots)\;.
\end{equation} 
Afin de pouvoir appliquer l'appro\-che par fonctions 
de Lyapounov, nous commen\c cons par d\'eterminer le g\'en\'erateur.

\begin{proposition}[G\'en\'erateur pour la dynamique de Glauber]
\label{prop:generateur_Glauber} 
Pour tout $x\in\cX$, notons 
\begin{align}
 A_+(x) &= \bigsetsuch{y\in\cX}{y\sim x, H(y) > H(x)}\;, \\
 A_0(x) &= \bigsetsuch{y\in\cX}{y\sim x, H(y) = H(x)}\;, \\
 A_-(x) &= \bigsetsuch{y\in\cX}{y\sim x, H(y) < H(x)}\;. 
\end{align}
Alors pour toute fonction $V:\cX\to\R$, on a 
\begin{equation}
\label{eq:LV(x)} 
 (\cL V)(x) = q \Biggbrak{\sum_{y\in A_-(x)\cup A_0(x)} \bigpar{V(y)-V(x)}
 + \sum_{y\in A_+(x)}\bigpar{V(y)-V(x)} \e^{-\beta \Delta H(x,y)}}\;.
\end{equation} 
\end{proposition}
\begin{proof}
Il suit de~\eqref{metro6} et~\eqref{metro7} que 
\begin{align}
 (\cL V)(x) &=
 \sum_{y\in A_-\cup A_0} p_{xy} V(y) 
 + \sum_{y\in A_+} p_{xy} V(y) 
 + p_{xx} V(x) - V(x) \\
 &= q \sum_{y\in A_-\cup A_0} V(y)
 + q \sum_{y\in A_+} \e^{-\beta\Delta H(x,y)} V(y)
 - q \Biggbrak{\sum_{y\in A_-\cup A_0} 1
 + \sum_{y\in A_+} \e^{-\beta\Delta H(x,y)}}V(x)\;,
\end{align}
d'o\`u le r\'esultat, en regroupant les termes. 
\end{proof}

Commen\c cons par \'etudier le cas particulier $\beta = 0$, qui correspond \`a 
une temp\'erature infinie. Dans ce cas, toutes les transitions permises ont la 
m\^eme probabilit\'e $q = 1/N$, et le syst\`eme effectue une marche al\'eatoire 
sym\'etrique sur l'hypercube de dimension $N$. La probabilit\'e invariante est 
simplement la mesure uniforme. 

Remarquons que $m(X_n)$ satisfait 
\begin{equation}
\label{eq:dynamique_m} 
m(X_{n+1}) = 
\begin{cases}
m(X_n) + 2 & \text{avec probabilit\'e $\dfrac{N_-(X_n)}{N} = \dfrac12 - \dfrac{m(X_n)}{2N}$\;,} \\[4mm]
m(X_n) - 2 & \text{avec probabilit\'e $\dfrac{N_+(X_n)}{N} = \dfrac12 + \dfrac{m(X_n)}{2N}$\;.} 
\end{cases}
\end{equation} 
Comme ces probabilit\'es de transition ne d\'ependent que de $m(X_n)$, la suite des 
$M_n = m(X_n)$ est une \CM\ sur $\cM = \set{-N,-N+2,\dots,N-2,N}$. En fait, 
$(M_n + N)/2$ n'est autre que le mod\`ele d'Ehrenfest \`a $N$ boules, qui admet la 
loi binomiale de param\`etres $(N,\frac12)$ comme probabilit\'e invariante. Ceci donne 
l'id\'ee d'utiliser $V(x) = m(x)^2$ comme fonction de Lyapounov. En effet, les valeurs 
de $m$ proches de $0$ sont plus probables que les valeurs proches de $\pm N$. On obtient 
facilement la condition de d\'erive g\'eom\'etrique suivante.

\begin{proposition}[Condition de d\'erive g\'eom\'etrique pour $V = m^2$]
Si $\beta = 0$, la fonction de Lyapounov $V(x) = m(x)^2$ satisfait 
\begin{equation}
\label{eq:LV_m2} 
 (\cL V)(x) = -c V(x) + d\;, 
 \qquad \text{avec $c = \dfrac{4}{N}$ et  $d = 4$\;.}
\end{equation} 
\end{proposition}
\begin{proof}
Par la Proposition~\ref{prop:generateur_Glauber} avec $\beta = 0$, on a  
\begin{align}
(\cL V)(x)
&= q \sum_{y\sim x} \Bigbrak{m(y)^2 - m(x)^2} \\
&= q \Bigbrak{N_-(x)\Bigpar{(m(x)+2)^2 - m(x)^2} + N_+(x)\Bigpar{(m(x)-2)^2 - m(x)^2}} \\
&= 4q \Bigbrak{N - V(x)}
\end{align}
en vertu de~\eqref{eq:Nplus_Nminus}. Le r\'esultat suit du fait que $q=\frac1N$. 
\end{proof}

Pour appliquer la condition de minoration~\eqref{eq:minoration} du 
Th\'eor\`eme~\ref{thm:convergence}, il nous faut choisir $R > 2d/c = 2N$, donc 
par exemple $R = 4N$. Ainsi, on aura 
\begin{equation}
\label{eq:K_Ising_beta0} 
 K = \Bigsetsuch{x\in\cX}{m(x)^2 < 4N}
 = \Bigsetsuch{x\in\cX}{\abs{m(x)} < 2\sqrt{N}}\;.
\end{equation} 
On constate que la condition de minoration ne peut pas \^etre satisfaite, car $p_{xy}$ 
est non nulle seulement si $x\sim y$, ce qui n'est pas le cas pour tous les $x,y\in K$. 
Une mani\`ere de r\'esoudre ce probl\`eme est de consid\'erer une puissance de la 
matrice de transision.

\begin{lemma}[Condition de d\'erive pour processus acc\'el\'er\'e]
\label{lem:derive_itere} 
Soit $(X_n)_{n\geqs0}$ une \CM\ dont le g\'en\'erateur satisfait 
\begin{equation}
 (\cL V)(x) \leqs -c V(x) + d \qquad \forall x\in\cX 
\end{equation} 
pour une fonction de Lyapounov $V$ et des constantes $c > 0$ et $d\geqs 0$. 
Alors pour tout $T\in\N^*$, le g\'en\'erateur $\cL^T$ du processus acc\'el\'er\'e 
$(X_{nT})_{n\geqs0}$ satisfait 
\begin{equation}
 (\cL^T V)(x) \leqs -\bigbrak{1 - (1-c)^T} V(x) + \bigbrak{1 - (1-c)^T} \frac{d}{c} 
 \qquad \forall x\in\cX\;. 
\end{equation} 
\end{lemma}
\begin{proof}
Pour $x\in\cX$ et $n\in\N$, soit  
\begin{equation}
 g_n(x) = (\cP^n V)(x) = \bigexpecin{x}{V(X_n)}\;.
\end{equation} 
Alors on obtient, comme dans la d\'emonstration de la formule de Dynkin, 
\begin{equation}
 g_{n+1}(x) \leqs -(1-c) g_n(x) + d\;.
\end{equation} 
En utilisant $g_0(x) = V(x)$ comme condition initiale, on obtient facilement 
par r\'ecurrence sur $n$ 
\begin{equation}
 g_n(x) \leqs -(1-c)^n V(x) + \bigbrak{1 - (1-c)^n} \frac{d}{c}
 \qquad \forall n\in\N^*\;.
\end{equation} 
Le r\'esultat suit alors du fait que $(\cL^T V)(x) = g_T(x) - V(x)$. 
\end{proof}

On peut simplifier l'\'etude en ne consid\'erant que la \CM\ $(M_n)_{n\geqs0}$. 
Notons $\hat\cL$ son g\'en\'erateur, et $\hat\pi$ sa probabilit\'e invariante, 
qui se d\'eduit de la loi binomiale. Alors~\eqref{eq:LV_m2} implique que si 
$V(m) = m^2$, on a
\begin{equation}
 (\hat\cL V)(m) = -c V(m) + d\;,
 \qquad \text{avec $c = \dfrac{4}{N}$ et  $d = 4$\;.}
\end{equation}
Cette relation peut aussi \^etre d\'eduite directement de~\eqref{eq:dynamique_m}. 
On obtient alors le r\'esultat de convergence suivant, en appliquant le 
Th\'eor\`eme~\ref{thm:convergence} a une puissance convenablement choisie de la matrice de transition de 
$(M_n)_{n\geqs0}$.

\begin{proposition}[Convergence pour fonctions de l'aimantation lorsque $\beta=0$]
Il existe des constants $C > 0$ et $\bar\gamma < 1$, ind\'ependantes de $N$, telles 
que pour toute fonction test $f:\cM\to\R$, on ait 
\begin{equation}
\label{eq:convergence_aimantation} 
 \bigabs{\expecin{m}{f(M_n)} - \hat\pi(f)} 
 \leqs C(1+m^2) \bar\gamma^{n/N} \norm{f - \hat\pi(f)}_{1+m^2}\;.
\end{equation} 
\end{proposition}
\begin{proof}
Consid\'erons le processus $(M^T_n)_{n\geqs0}$ acc\'el\'er\'e d'un facteur $T\in\N^*$, 
d\'efini comme $M^{T}_n = M_{Tn}$.
En proc\'edant comme dans la d\'emonstration du Lemme~\ref{lem:derive_itere}, 
on obtient 
\begin{equation}
 (\hat\cL^T V)(m) = -c_T V(m) + d_T\;, 
 \qquad 
 \text{avec $c_T = 1-(1-c)^T$ et $d_T = c_T\frac{d}{c}$\;.}
\end{equation} 
On remarque que $2d_T/c_T = 2d/c = 2N$. On peut donc utiliser la m\^eme 
valeur $4N$ pour $R$, et 
\begin{equation}
 K = \Bigsetsuch{m\in\cM}{\abs{m} < 2\sqrt{N}\,}\;.
\end{equation} 
En revanche, on a gagn\'e en ce qui concerne la condition de minoration. En effet, pour $m$ 
d'ordre $\sqrt{N}$, l'\'evolution de $M_n$ est essentiellement une marche al\'eatoire 
sym\'etrique (plus pr\'ecis\'ement, les probabilit\'es de transitions sont minor\'ees 
par celles d'une marche al\'eatoire de param\`etre $\frac12 - \Order{N^{-1/2}}$).
La variance d'une telle marche al\'eatoire au temps $T$ est d'ordre $T$. Il suit 
du th\'eor\`eme central limite (ou du th\'eor\`eme de de Moivre--Laplace) que 
\begin{equation}
 \bigprobin{m}{M_T = p} \geqs \frac{a}{\sqrt{T}}
 \qquad \forall m,p\in K
\end{equation} 
pour une constante $a$ ind\'ependante de $T$. Choisissons alors pour $\nu$
la probabilit\'e uniforme sur $K$. 
Prenons $T = bN$ pour un $b$ positif. Comme le cardinal de $K$ est d'ordre $\sqrt{N}$, 
on peut choisir $b$ de telle mani\`ere que 
\begin{equation}
 \inf_{m\in K} \bigprobin{m}{M_T = p} \geqs \frac12 \nu(p) 
 \qquad \forall p\in K\;.
\end{equation} 
Comme $\nu(m) = 0$ pour tout $m\not\in K$, la condition de minoration~\eqref{eq:minoration}
est satisfaite avec $\alpha = \frac12$. On v\'erifie alors que les choix 
\begin{equation}
 \alpha_0 = \frac14\;, \qquad 
 \gamma_0 = 1 - \frac{c_T}4
\end{equation} 
satisfont~\eqref{eq:cond_gamma0}, avec $\gamma = 1 - c_T$. On trouve ensuite \`a l'aide 
de~\eqref{eq:beta_gammabar} 
\begin{equation}
 \beta = \frac{1}{4Nc_T}\;, \qquad 
 R\beta = \frac{1}{c_T}\;, \qquad 
 \bar\gamma = \max\biggset{\frac14, 1 - \frac{c_T}{4(1+2c_T)}}\;.
\end{equation}
Comme de plus 
\begin{equation}
 \log(1-c_T) = T\log\biggpar{1-\frac{4}{N}} = -\frac{4T}{N} + \biggOrder{\frac{T}{N^2}}
 = -4b + \biggOrder{\frac{1}{N}}\;, 
\end{equation} 
$c_T$ converge vers une limite ind\'ependante de $N$ lorsque $N\to\infty$ (qui vaut $1-\e^{-4b}$). 

Par le Th\'eor\`eme~\ref{thm:convergence}, l'esp\'erance de 
$f(M^{T}_n)$ converge vers sa limite $\hat\pi(f)$ exponentiellement vite, 
avec un taux $\bar\gamma^n$. Le r\'esultat suit en revenant au temps non acc\'el\'er\'e, 
quitte \`a remplacer $\bar\gamma$ par $\bar\gamma^{1/b}$.
\end{proof}

La borne~\eqref{eq:convergence_aimantation} montre que pour approcher $\hat\pi(f)$ \`a 
une distance d'ordre $\delta$, il faut choisir un $n$ d'ordre $N\log(1/\delta)$. 
Cette convergence est relativement rapide. Bien entendu, comme la probabilit\'e invariante 
$\hat\pi$ est connue explicitement, on n'a pas besoin d'estimer $\hat\pi(f)$, on peut la calculer directement avec un co\^ut $N$. On notera aussi que ce r\'esultat ne marche que pour des fonctions de l'aimantation. Il n'affirme rien, par exemple, sur des fonctions 
qui d\'ependraient de $I(x)$. 

Le second cas particulier que nous allons consid\'erer est celui o\`u $0 < h \leqs 1$ et $\beta$ 
est assez grand (dans un sens \`a pr\'eciser plus bas), donc o\`u la temp\'erature est assez 
basse. Pour $h > 0$, l'\'etat d'\'energie minimale est $\boxplus$. Un candidat pour une fonction de Lyapounov est la diff\'erence d'\'energie 
\begin{equation}
\label{eq:Lyapounov_H} 
 V(x) = H(x) - H(\boxplus)
 = 2I(x) + 2h N_-(x)\;,
\end{equation} 
qui est bien positive ou nulle. Nous commen\c cons par donner une majoration 
g\'en\'erale de $\cL H$ (comme $\cL c = 0$ pour toute constante $c$, cela fournit 
\'egalement une majoration de $\cL V$). 

\begin{lemma}[Majoration de $\cL H$]
\label{lem:majo_LH} 
Supposons $\beta \geqs \frac1{2h}$. Alors 
pour tout $x\in\cX$, on a 
\begin{equation}
\label{eq:majoration_LH} 
 (\cL H)(x) \leqs 
 -2qh\abs{A_-(x)} + 2qh \e^{-2\beta h}\abs{A_+(x)}\;.
\end{equation} 
\end{lemma}
\begin{proof}
Il suit de l'expression~\eqref{metro9} de $\Delta H$ que 
\begin{equation}
 \Delta H(x,y) \neq 0 
 \qquad \Rightarrow \qquad 
 \abs{\Delta H(x,y)} \geqs 2h\;.
\end{equation} 
En effet, la somme des $x_j$ est n\'ecessairement un entier pair, et le minimum
de $\abs{\Delta H(x,y)}$ est atteint lorsque cette somme vaut $0$. 
Il suit que la premi\`ere somme dans~\eqref{eq:LV(x)} est major\'ee par 
$-2h\abs{A_-(x)}$, puisque les termes avec $y\in A_0(x)$ sont nuls. 
La seconde somme s'\'ecrit 
\begin{equation}
  \sum_{y\in A_+(x)}\Delta H(x,y) \e^{-\beta \Delta H(x,y)}\;.
\end{equation} 
Or la fonction $u\mapsto f(u) = u\e^{-\beta u}$ est croissante 
sur $[0,1/\beta]$ et d\'ecroissante sur $[1/\beta, \infty[$.
Comme $\Delta H(x,y)$ est minor\'e par $2h$ dans cette somme, 
et que $2h \geqs 1/\beta$, celle-ci est inf\'erieure ou \'egale \`a 
$2h\e^{-2\beta h}\abs{A_+}$. 
\end{proof}

La majoration~\eqref{eq:majoration_LH} exprime le fait que si $\beta$ est 
assez grand, alors l'\'energie a tendance \`a diminuer, sauf dans le cas 
particulier o\`u $\abs{A_-} = 0$, c'est-\`a-dire si $x$ n'a pas de 
configuration voisine d'\'energie inf\'erieure. Cela n'arrive que  
si $x \in \set{\boxplus, \boxminus}$. 

\begin{proposition}[Condition de d\'erive g\'eom\'etrique pour $V = H - H(\boxplus)$]
Si $\beta \geqs \frac1{2h}$, la fonction de Lyapounov~\eqref{eq:Lyapounov_H} 
satisfait la condition de d\'erive g\'eom\'etrique 
\begin{equation}
\label{eq:majo_LV1} 
 (\cL V)(x) \leqs -c V(x) + d
 \qquad 
 \text{avec $c = \dfrac{h}{2N}$ et $d = 2h\e^{-2\beta h} + h^2$\;.}
\end{equation} 
\end{proposition}
\begin{proof}
Nous allons consid\'erer s\'epar\'ement les cas $x=\boxplus$, 
$x=\boxminus$, et $x\in\cX\setminus\set{\boxplus,\boxminus}$.
\begin{itemize}
\item  Si $x = \boxplus$, alors $\abs{A_-(x)} = 0$, puisque retourner un spin augmente 
toujours l'\'energie. Par cons\'equent, \eqref{eq:majoration_LH} implique 
\begin{equation}
 (\cL V)(\boxplus) \leqs 2h\e^{-2\beta h}\;,
\end{equation} 
o\`u nous avons utilis\'e $\abs{A_+(\boxplus)} = N$ et $q = 1/N$. Comme $V(\boxplus) = 0$, la borne~\eqref{eq:majo_LV1} est bien v\'erifi\'ee. 

\item   Si $x = \boxminus$, alors on a 
\begin{equation}
 (\cL V)(\boxminus) \leqs 2h\e^{-2\beta h}\;,
 \qquad 
 V(\boxminus) = 2hN\;.
\end{equation} 
Par cons\'equent,~\eqref{eq:majo_LV1} est v\'erifi\'e puisque 
$d = 2h\e^{-2\beta h} + 2hNc$. 

\item   Pour tous les autres $x$, on a toujours 
$\abs{A_-(x)} \geqs \frac12 I(x)$.
En effet, pour chaque interface, changer le spin $-1$ qui se trouve 
d'un c\^ot\'e de l'interface en $+1$ diminue l'\'energie 
de la configuration. Le facteur $\frac12$ vient du fait que ce spin $-1$
peut \^etre compris entre deux interfaces. 
Ainsi,
\begin{align}
 (\cL V)(x) &\leqs -qh I(x) + 2h\e^{-2\beta h}\\
 &= -\frac12 qhV(x) + qh^2N_-(x) + 2h\e^{-2\beta h}\\
 &\leqs -\frac12 qhV(x) + qh^2N + 2h\e^{-2\beta h}\;.
\end{align} 
Ceci montre que~\eqref{eq:majo_LV1} est bien v\'erifi\'e, puisque $q = 1/N$. 
\qed
\end{itemize}
\renewcommand{\qed}{}
\end{proof}

Si nous supposons $\beta \geqs \frac{1}{2h}\log(\frac2h)$, alors
on peut prendre $d = 2h^2$. Dans la condition de minoration, il faut 
donc prendre $R > 8hN$, de sorte que $K = \setsuch{x\in\cX}{V(x)<R}$ contient 
beaucoup d'\'etats (selon la valeur de $h$, il peut m\^eme arriver que $K = \cX$). 
Il nous faut donc \`a nouveau acc\'elerer le temps afin de pouvoir appliquer 
le Th\'eor\`eme~\ref{thm:convergence}. 

Nous n'allons pas donner une analyse d\'etaill\'ee, mais un argument heuristique. 
Si l'on prend $\nu = \delta_{\boxplus}$, on aura, pour la \CM\ acc\'el\'er\'ee 
d'un facteur $T$, 
\begin{equation}
 \alpha = \inf_{x\in K} \probin{x}{X_T = \boxplus}\;.
\end{equation} 
La question est de savoir comment choisir $T$ pour que $\alpha$ soit d'ordre $1$, 
disons $\alpha = \frac12$. 

On s'attend \`a ce que la transition la plus difficile soit celle de $\boxminus$ 
vers $\boxplus$. La mani\`ere la plus \'economique de faire cette transition est 
de renverser d'abord un spin quelconque, puis de renverser des spins adjacents, un 
par un, jusqu'\`a atteindre $\boxplus$ (Figure~\ref{fig:Ising_optimal_transition}). 
On v\'erifie que seule la premi\`ere transition fait augmenter l'\'energie. Quitte \`a 
augmenter encore $\beta$, on peut mod\'eliser la transition en n\'egligeant tout renversement de spin non optimal, faisant augmenter l'\'energie plus que n\'ecessaire. 
On aboutit alors \`a la \CM\ de la Figure~\ref{fig:Ising_transition}. 
En effet, la probabilit\'e de la premi\`ere transition est de 
$Nq\e^{-\beta\Delta H(\boxplus,R_k(\boxplus))}$, o\`u $\Delta H(\boxplus,R_k(\boxplus)) = 4-2h$ ne d\'epend pas de $k$, puisqu'on peut retourner n'importe lequel des $N$ spins. Toutes les transitions suivantes on la m\^eme probabilit\'e $2q$, car on peut choisir de quel c\^ot\'e la goutte cro\^it. 

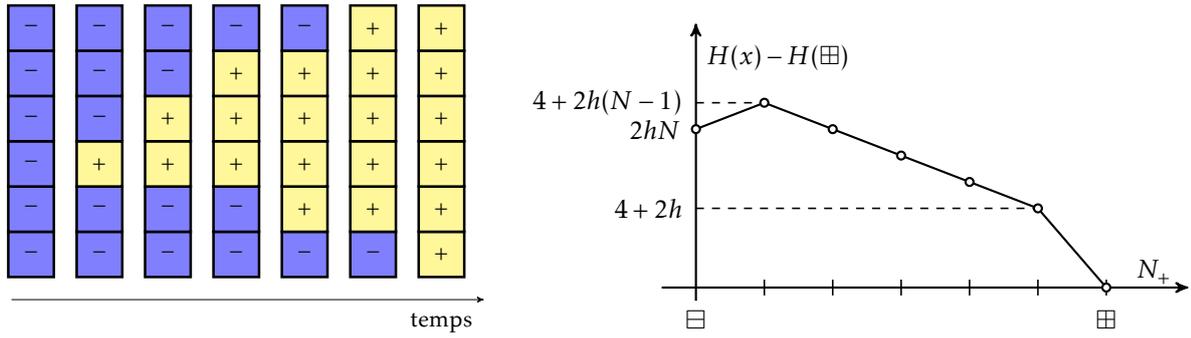
\begin{figure}
\vspace{-3mm}
\begin{center}
\scalebox{0.6}{
\begin{tikzpicture}[->,>=stealth',shorten >=2pt,shorten <=2pt,auto,node
distance=3.0cm,thick,
minus spin/.style={ultra thick,rectangle,scale=1,minimum size=1cm,
fill=blue!50,draw,font=\sffamily\Large},
plus spin/.style={ultra thick,,rectangle,scale=1,minimum size=1cm,
fill=yellow!50,draw,font=\sffamily\Large}]

\node[minus spin] at (0,0) {$-$};
\node[minus spin] at (0,1) {$-$};
\node[minus spin] at (0,2) {$-$};
\node[minus spin] at (0,3) {$-$};
\node[minus spin] at (0,4) {$-$};
\node[minus spin] at (0,5) {$-$};

\node[minus spin] at (1.5,0) {$-$};
\node[minus spin] at (1.5,1) {$-$};
\node[plus spin]  at (1.5,2) {$+$};
\node[minus spin] at (1.5,3) {$-$};
\node[minus spin] at (1.5,4) {$-$};
\node[minus spin] at (1.5,5) {$-$};

\node[minus spin] at (3,0) {$-$};
\node[minus spin] at (3,1) {$-$};
\node[plus spin]  at (3,2) {$+$};
\node[plus spin]  at (3,3) {$+$};
\node[minus spin] at (3,4) {$-$};
\node[minus spin] at (3,5) {$-$};

\node[minus spin] at (4.5,0) {$-$};
\node[minus spin] at (4.5,1) {$-$};
\node[plus spin]  at (4.5,2) {$+$};
\node[plus spin]  at (4.5,3) {$+$};
\node[plus spin]  at (4.5,4) {$+$};
\node[minus spin] at (4.5,5) {$-$};

\node[minus spin] at (6,0) {$-$};
\node[plus spin]  at (6,1) {$+$};
\node[plus spin]  at (6,2) {$+$};
\node[plus spin]  at (6,3) {$+$};
\node[plus spin]  at (6,4) {$+$};
\node[minus spin] at (6,5) {$-$};

\node[minus spin] at (7.5,0) {$-$};
\node[plus spin]  at (7.5,1) {$+$};
\node[plus spin]  at (7.5,2) {$+$};
\node[plus spin]  at (7.5,3) {$+$};
\node[plus spin]  at (7.5,4) {$+$};
\node[plus spin]  at (7.5,5) {$+$};

\node[plus spin]  at (9,0) {$+$};
\node[plus spin]  at (9,1) {$+$};
\node[plus spin]  at (9,2) {$+$};
\node[plus spin]  at (9,3) {$+$};
\node[plus spin]  at (9,4) {$+$};
\node[plus spin]  at (9,5) {$+$};

\path[->,>=stealth', semithick] (-0.5,-1) edge (10,-1);

\node at (9,-1.5) {\Large temps};
\end{tikzpicture}
}
\hspace{2mm}
\begin{tikzpicture}[-,scale=0.5,auto,node
distance=1.0cm, thick,main node/.style={draw,circle,fill=white,minimum
size=3pt,inner sep=0pt}, yscale=0.7, xscale=1.8]

\path[->,>=stealth'] 
     (-0.5,0) edge (7.2,0)
     (0,-0.5) edge (0,10)
  ;

\node at (6.7,0.6) {\small $N_+$};
\node at (1.2,8.7) {\small $H(x) - H(\boxplus)$};

\draw[semithick,dashed] (0,7) -- (1,7);  
\draw[semithick,dashed] (0,3) -- (5,3);  

\foreach \i in {1,...,6} \draw[semithick] (\i,-0.3) -- (\i,0.3);

\draw (0,6) node[main node] {} 
  -- (1,7) node[main node] {} 
  -- (2,6) node[main node] {} 
  -- (3,5) node[main node] {} 
  -- (4,4) node[main node] {} 
  -- (5,3) node[main node] {} 
  -- (6,0) node[main node] {} 
  ;

\node[] at (-0.6,6) {\small $2hN$};
\node[] at (-0.7,3) {\small $4+2h$};
\node[] at (-1.3,7) {\small $4+2h(N-1)$};

\node[] at (0,-1.2) {\small $\boxminus$};
\node[] at (6,-1.2) {\small $\boxplus$};
\end{tikzpicture}
\end{center}
\vspace{-4mm}
 \caption[]{\`A gauche, exemple d'une transition optimale de l'\'etat $\boxminus$ vers l'\'etat $\boxplus$ par croissance d'une goutte, pour $N=6$. \`A droite, valeur de la diff\'erence d'\'energie $H(x) - H(\boxplus)$ en fonction de $N_+$.}
 \label{fig:Ising_optimal_transition}
\end{figure}

Soit alors $f(y) = \expecin{y}{\tau_N}$. Pour $y\in\set{2,\dots,N-1}$, cette fonction satisfait 
\begin{equation}
 f(y) = 2q f(y+1) + (1-2q)f(y) + 1\;.
\end{equation} 
Avec la condition initiale $f(y) = 0$, on trouve 
\begin{equation}
\label{eq:f(2)} 
 \expecin{2}{\tau_N} = \frac{N-2}{2q} = \frac{N(N-2)}{2}\;.
\end{equation} 
Par ailleurs, pour $y\in\set{0,1}$ on obtient les \'equations 
\begin{align}
f(0) &= \e^{-\beta\Delta H} f(1) + (1 - \e^{-\beta\Delta H}) f(0) + 1 \\
f(1) &= qf(0) + (1-3q)f(1) + 2qf(2) + 1\;.
\end{align}
En r\'esolvant ce syst\`eme pour $f(0)$ et $f(1)$ (ce qui revient \`a calculer 
la matrice fondamentale de la \CM\ absorb\'ee en $2$), on obtient 
\begin{equation}
 f(0) = \frac32 \e^{\beta\Delta H} + \frac1{2q} + f(2)\;.
\end{equation} 
En combinant ceci avec~\eqref{eq:f(2)}, on aboutit finalement,  
dans cette approximation, \`a 
\begin{equation}
 \expecin{\boxminus}{\tau_{\boxplus}} \simeq \frac32\e^{2\beta(2-h)} + \frac{N(N-1)}{2}\;.
\end{equation} 
L'in\'egalit\'e de Markov implique alors  
\begin{equation}
 \probin{\boxminus}{\tau_{\boxplus} \geqs k } \leqs \frac{\expecin{\boxminus}{\tau_{\boxplus}}}{k}\;,
\end{equation} 
Par cons\'equent, en choisissant $T = 2 \expecin{\boxminus}{\tau_{\boxplus}}$, 
on aura $\alpha = \frac12$. 
Soient alors $c_T$ et $d_T$ les constantes donn\'ees par le Lemme~\ref{lem:derive_itere}. 
Comme
\begin{equation}
 \log(1-c_T) 
 = T\log\biggpar{1 - \frac{h}{2N}}
 = -\frac{hT}{2N} \biggpar{1 + \biggOrder{\frac{h^2}{N}}}\;,
\end{equation} 
on a $c_T = 1 - \Order{\e^{-hT/(2N)}}$. 
Un choix possible de param\`etres est 
\begin{equation}
 \alpha_0 = \frac14\;, \qquad 
 R = 16hN\;, \qquad 
 \gamma_0 = 1 - \frac{1}{4}c_T\;, 
\end{equation} 
ce qui conduit \`a $R\beta = 1/c_T$ et 
\begin{equation}
 \bar\gamma = \frac{11}{12} + \Order{\e^{-hT/(2N)}}\;.
\end{equation} 
Le point important est que $1-\bar\gamma$ est minor\'e par une 
quantit\'e ind\'ependante de $N$. On s'attend donc \`a une convergence de la forme 
\begin{equation}
\label{eq:convergence_Ising2} 
 \bigabs{\expecin{x}{f(X_n)} - \pi(f)} 
 \leqs C(1+V(x)) \bar\gamma^{n/T} \norm{f - \pi(f)}_{1+V}
\end{equation} 
avec $T = N(N-1) + 2\e^{\beta(2-h)}$. Pour atteindre une pr\'ecision $\delta$, 
il faut choisir $n$ d'ordre $T\log(1/\delta)$. Si $\beta$ n'est pas trop grand, 
ce temps varie comme $N^2$. Toutefois, si $\e^{\beta(2-h)}$ d\'epasse $N^2$, 
c'est ce terme qui d\'etermine le temps de convergence. L'algorithme converge 
donc moins rapidement \`a tr\`es faible temp\'erature, en raison du temps 
n\'ecessaire \`a renverser le premier spin de la configuration $\boxminus$.

\begin{figure}
\vspace{-3mm}
\begin{center}
\begin{tikzpicture}[->,>=stealth',shorten >=2pt,shorten <=2pt,auto,node
distance=3.0cm, thick,main node/.style={circle,scale=0.7,minimum size=1.2cm,
fill=blue!20,draw,font=\sffamily\Large}]

  \node[main node] (0) {$0$};
  \node[main node] (1) [right of=0] {$1$};
  \node[main node] (2) [right of=1] {$2$};
  \node[main node] (3) [right of=2] {$3$};
  \node[circle,scale=0.7,minimum size=1.2cm] (dots) [right of=3,distance=2cm] {\dots};
  \node[main node] (N-1) [right of=dots] {\small $N-1$};
  \node[main node] (N) [right of=N-1] {$N$};

  \path[every node/.style={font=\sffamily\small}]
    (0) edge [bend left, above] node {$\e^{-\beta\Delta H}$} (1)
    (1) edge [bend left, above] node {$2q$} (2)
    (1) edge [bend left, below] node {$q$} (0)
    (2) edge [bend left, above] node {$2q$} (3)
    (3) edge [bend left, above] node {$2q$} (dots)
    (dots) edge [bend left, above] node {$2q$} (N-1)
    (N-1) edge [bend left, above] node {$2q$} (N)
    (0) edge [loop right, above,distance=1.5cm,out=120,in=60] node {$1 - \e^{-\beta\Delta H}$} (0)
    (1) edge [loop right, above,distance=1.5cm,out=120,in=60] node {$1-3q$} (1)
    (2) edge [loop right, above,distance=1.5cm,out=120,in=60] node {$1-2q$} (2)
    (3) edge [loop right, above,distance=1.5cm,out=120,in=60] node {$1-2q$} (3)
    (N-1) edge [loop right, above,distance=1.5cm,out=120,in=60] node {$1-2q$} (N-1)
    (N) edge [loop right, above,distance=1.5cm,out=120,in=60] node {$1$} (N)
  ;
\end{tikzpicture}
\end{center}
\vspace{-2mm}
 \caption[]{\CCM\ mod\'elisant une transition optimale de l'\'etat $\boxplus$ vers l'\'etat 
 $\boxminus$, par croissance d'une goutte de spins $+1$. La valeur de $\Delta H$ pour 
 la transition entre les \'etats $0$ et $1$ est $\Delta H(\boxplus, R_k(\boxplus)) = 4 - 2h$.
 Les \'etats sont num\'erot\'es selon le nombre $N_+$ de spins valant $+1$.}
 \label{fig:Ising_transition}
\end{figure}
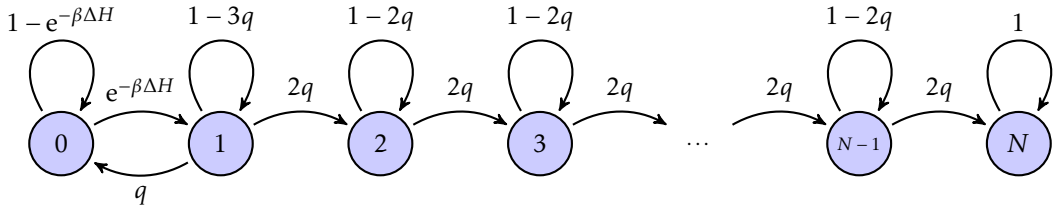

Dans le cas du mod\`ele d'Ising sur $\Lambda\subset\Z^2$, la situation est moins 
favorable. En effet, en partant de la configuration $\boxminus$, il faut d'abord 
cr\'eer une goutte de spins $+1$ d'une certaine taille avant que l'\'energie se 
mette \`a diminuer en approchant $\boxplus$. Dans ce cas, il
existe des algorithmes alternatifs, tels que l'algorithme dit de
Swendsen--Wang, qui convergent beaucoup mieux. Au lieu de retourner 
un seul spin \`a la fois, cet algorithme retourne des groupes de spins bien choisis. 


\part[Cha\^ines de Markov \`a espace continu]{Cha\^ines de Markov\\ \`a espace continu}
\label{part:cm_continu} 


\chapter{D\'efinition et exemples de \CMs\ \`a espace continu}
\label{chap:cont_ex} 

Dans ce chapitre, nous examinons comment on peut \'etendre la th\'eorie des \CMs\ 
sur un ensemble $\cX$ d\'enombrable \`a des ensembles infinis non d\'enombrables, 
plus pr\'ecis\'ement des sous-ensembles ouverts de $\R^d$. Une grande partie des 
concepts du cas discret (\'evolution de la loi de $X_n$, probabilit\'e invariante) 
peuvent \^etre transpos\'es \`a cette situation de mani\`ere assez directe, essentiellement 
en \myquote{rempla\c cant les sommes par des int\'egrales}. 
Il faut \^etre un peu prudent, toutefois, en g\'en\'eralisant les notions de 
r\'ecurrence et de r\'ecurrence positive. Nous aborderons cette question dans le 
chapitre suivant. 


\section{D\'efinitions et notations}
\label{sec:cont_def} 

Soit $\cX\subset\R^d$ un ouvert. Cet ensemble est muni de la \defwd{tribu 
des bor\'eliens}, qui contient en particulier tous les ouverts de $\cX$. 
Voici d'abord la g\'en\'eralisation de concept de matrice stochastique \`a cette 
situation. 

\begin{definition}[Densit\'e de probabilit\'e, noyau markovien \`a densit\'e]
\label{def:noyau_markovien} 
\begin{itemize}
\item  Une \defwd{densit\'e de probabilit\'e} $\nu$ sur $\cX$ est une 
application $\nu:\cX\to\R_+ = [0,\infty[$, continue par morceaux, et 
satisfaisant 
\begin{equation}
\label{eq:dproba} 
 \int_{\cX} \nu(x)\6x = 1\;.
\end{equation} 
\item   Un \defwd{noyau markovien \`a densit\'e} sur $\cX$ est une application 
$p:\cX\times\cX\to\R_+$, continue par morceaux, satisfaisant 
\begin{equation}
\label{eq:dstoch} 
 \int_{\cX} p(x,y)\6y = 1 \qquad \forall x\in\cX\;.
\end{equation} 
\end{itemize}
\end{definition}

Dans la suite, nous utiliserons la m\^eme notation pour la mesure de probabilit\'e
associ\'ee \`a la densit\'e $\nu$. Cela revient \`a poser 
\begin{equation}
 \nu(A) = \int_A \nu(x)\6x
\end{equation} 
pour tout bor\'elien $A\subset\cX$. 
La g\'en\'eralisation naturelle de la notion de \CM\ est alors la suivante. 

\begin{definition}[Cha\^ine de Markov sur un ouvert $\cX$ de $\R^d$]
\label{def:CM_continu} 
Soit $\nu$ une densit\'e de probabilit\'e sur $\cX$, et $p$ un noyau markovien 
\`a densit\'e. 
Une \defwd{\CM} (homog\`ene en temps) sur $\cX$, de loi initiale $\nu$ et de noyau de transition $p$, est une suite $(X_n)_{n\geqs0}$ de variables al\'eatoires \`a valeurs dans $\cX$, telles que $\prob{X_0 \in A} = \nu(A)$ pour tout bor\'elien $A\subset\cX$, et satisfaisant la \defwd{propri\'et\'e de Markov}
\begin{align}
\pcond{X_n \in A}{X_0 = x_0, X_1 = x_1, \dots, X_{n-1} = x_{n-1}}
&= \pcond{X_n \in A}{X_{n-1} = x_{n-1}} \\
&= \int_A p(x_{n-1},x_n) \6x_n
\label{eq:Markov_cont} 
\end{align}
pour tout $n\geqs1$, tout choix de $x_0, \dots, x_{n-1}\in\cX$, 
et tout bor\'elien $A\subset \cX$. 
\end{definition}

Comme la probabilit\'e qu'une variable al\'eatoire \`a densit\'e prenne une valeur particuli\`ere 
vaut $0$, il n'est pas imm\'ediatement \'evident que les probabilit\'es conditionnelles dans~\eqref{eq:Markov_cont} sont bien d\'efinies. Il faut en fait les interpr\'eter \`a l'aide de 
densit\'es conditionnelles. Pour ce faire, soit 
\begin{equation}
 \cB_\eps(x_0) = \bigsetsuch{x\in\cX}{\norm{x-x_0} < \eps} 
\end{equation} 
la boule ouverte de centre $x_0$ et de rayon $\eps$ (o\`u $\norm{\cdot}$ est la norme Euclidienne). 
On d\'efinit alors 
\begin{align}
\bigpcond{X_1 \in A}{X_0 = x_0}
&= \lim_{\eps\to0} \bigpcond{X_1 \in A}{X_0 \in\cB_\eps(x_0)} \\
&= \lim_{\eps\to0} \frac{\bigprob{X_1\in A, X_0 \in\cB_\eps(x_0)}}
{\bigprob{X_0 \in\cB_\eps(x_0)}}\;.
\end{align}
Si $f(x_0,x_1)$ d\'esigne la densit\'e jointe de $X_0$ et $X_1$, alors on a 
\begin{align}
 \bigprob{X_0 \in\cB_\eps(x_0)}
 &= \int_{\cB_\eps(x_0)} \nu(x)\6x\;, \\
 \bigprob{X_1\in A, X_0 \in\cB_\eps(x_0)}
 &= \int_A \int_{\cB_\eps(x_0)} f(x,x_1) \6x\6x_1\;, 
\end{align}
de sorte que 
\begin{equation}
 \bigpcond{X_1 \in A}{X_0 = x_0}
 = \int_A \lim_{\eps\to0} 
 \frac{\displaystyle\int_{\cB_\eps(x_0)} f(x,x_1) \6x}
 {\displaystyle\int_{\cB_\eps(x_0)} \nu(x)\6x} \6x_1
 = \int_A \frac{f(x_0,x_1)}{\nu(x_0)} \6x_1\;.
\end{equation} 
La derni\`ere \'egalit\'e suit du th\'eor\`eme de la valeur moyenne, qui montre que  
\begin{equation}
 \lim_{\eps\to0} \frac{1}{\abs{\cB_\eps(x_0)}}
 \int_{\cB_\eps(x_0)} \nu(x) \6x
 = \nu(x_0)\;, \qquad
 \lim_{\eps\to0} \frac{1}{\abs{\cB_\eps(x_0)}}
 \int_{\cB_\eps(x_0)} f(x,x_1) \6x
 = f(x_0,x_1)\;.
\end{equation} 
En comparant avec~\eqref{eq:Markov_cont} pour $n=1$, il vient 
\begin{equation}
\label{eq:f_x0x1} 
\frac{f(x_0,x_1)}{\nu(x_0)} = p(x_0,x_1)\;.
\end{equation} 
Le noyau markovien $p(x_0,x_1)$ s'interpr\`ete donc comme la \defwd{densit\'e 
conditionnelle de $X_1$ sachant que $X_0 = x_0$}. Par un raisonnement analogue, 
pour tout $n\in\N^*$, la densit\'e jointe de $(X_0,\dots,X_n)$ vaut 
\begin{equation}
\label{eq:f_x0xn} 
 f(x_0,\dots,x_n) = \nu(x_0)p(x_0,x_1)\dots p(x_{n-1},x_n)\;.
\end{equation} 
C'est l'analogue continu de la relation~\eqref{eq:proba_traj} pour la probabilit\'e 
des trajectoires dans le cas discret.

La loi de chaque $X_n$ est obtenue en calculant la marginale ad\'equate de la loi jointe. 
Ainsi, \eqref{eq:f_x0x1} implique 
\begin{equation}
 \bigprob{X_1\in A} 
 = \int_A \int_{\cX} f(x_0,x_1) \6x_0 \6x_1
 = \int_A \int_{\cX} \nu(x_0)p(x_0,x_1) \6x_0 \6x_1\;.
\end{equation} 
De mani\`ere analogue, \eqref{eq:f_x0xn} montre que 
\begin{align}
 \bigprob{X_2\in A}
 &= \int_A \int_{\cX} \int_{\cX} \nu(x_0)p(x_0,x_1)p(x_1,x_2) \6x_0 \6x_1 \6x_2 \\
 &= \int_A \int_{\cX} \nu(x_0)p^2(x_0,x_2)\6x_0\6x_2\;,
\end{align}
o\`u $p^2$ est un noyau markovien d\'efini par 
\begin{equation}
 p^2(x_0,x_2) = \int_{\cX} p(x_0,x_1)p(x_1,x_2)\6x_1\;.
\end{equation} 
Plus g\'en\'eralement, pour tout $n\geqs2$ on a 
\begin{equation}
 \prob{X_n\in A} = \int_A \int_{\cX} \nu(x_0)p^n(x_0,x_n)\6x_0\6x_n\;,
\end{equation} 
o\`u $p^n$ est un noyau markovien d\'efini par r\'ecurrence par 
\begin{equation}
\label{eq:Chapman-Kolmogorov} 
 p^n(x_0,x_n) = \int_{\cX} p^{n-1}(x_0,x_{n-1})p(x_{n-1},x_n)\6x_{n-1}\;,
\end{equation}
avec $p^1 = p$. 
Cette relation est appel\'ee \defwd{relation de Chapman--Kolmogorov}. 

Il sera commode d'utiliser les notations suivantes, o\`u $A\subset\cX$ est un 
bor\'elien, et $f:\cX\to\R$~:
\begin{align}
 \probin{\nu}{X_1\in A} &= (\nu\cP)(A) := \int_A\int_{\cX} \nu(x_0)p(x_0,x_1)\6x_0\6x_1\;, \\
 \expecin{x_0}{f(X_1)} &= (\cP f)(x_0) := \int_{\cX} p(x_0,x_1) f(x_1)\6x_1\;, \\
 \expecin{\nu}{f(X_1)} &= (\nu\cP)(f) := \nu(\cP f)
 = \int_{\cX} \int_{\cX} \nu(x_0) p(x_0,x_1) f(x_1)\6x_0\6x_1\;.
 \label{eq:P_cont} 
\end{align}
La relation de Chapman--Kolmogorov permet de d\'efinir $\cP^n$ pour tout $n\geqs1$ en rempla\c cant 
$p$ par $p^n$ dans~\eqref{eq:P_cont}. 
On a \'egalement des concepts de mesure sign\'ee et fonction test 
tout \`a fait analogues \`a ceux du cas discret. 

\begin{definition}[Mesures sign\'ees finies \`a densit\'e]
\label{def:mesure_cont} 
Soit $\mu:\cX\to\R$ une application continue par morceaux telle que 
\begin{equation}
 \norm{\mu}_1 := \int_{\cX} \abs{\mu(x)}\6x < \infty\;.
\end{equation} 
Elle d\'efinit une \defwd{mesure sign\'ee finie \`a densit\'e}, qui associe 
\`a tout bor\'elien $A$ le nombre 
\begin{equation}
 \mu(A) := \int_A \mu(x)\6x\;.
\end{equation} 
On notera $\cE_1$ l'espace de Banach des mesures sign\'ees finies \`a densit\'e. 
Si $\mu:\cX\to\R_+$ et $\norm{\mu}_1 = 1$, alors $\mu$ est une 
mesure de probabilit\'e. 
\end{definition}

\begin{definition}[Fonctions test]
\label{def:fct_test_cont} 
Une \defwd{fonction test} sur $\cX$ est une application 
$f:\cX\to\R$, continue par morceaux, telle que 
\begin{equation}
 \norm{f}_\infty := \sup_{x\in\cX} \abs{f(x)} < \infty\;.
\end{equation} 
On notera $\cE_\infty$ l'espace de Banach des fonctions test. 
\end{definition}

De mani\`ere analogue au cas discret, nous utiliserons la notation 
\begin{equation}
 \mu(f) = \int_{\cX} \mu(x)f(x)\6x\;.
\end{equation} 
Cette int\'egrale est bien d\'efinie pour tout $\mu\in\cE_1$ et tout $f\in\cE_\infty$, 
et on a 
\begin{equation}
 \abs{\mu(f)} \leqs \int_{\cX} \abs{\mu(x)} \abs{f(x)} \6x 
 \leqs \norm{\mu}_1 \norm{f}_\infty\;.
\end{equation} 

\begin{remark}[Continuit\'e par morceaux]
L'hypoth\`ese de continuit\'e par morceaux n'est pas vraiment n\'ecessaire. 
Toutes les int\'egrales ci-dessus peuvent \^etre interpr\'et\'ees comme des int\'egrales 
de Lebesgue, et alors on peut remplacer \myquote{continue par morceaux}\ par 
\myquote{mesurable}. Toutefois, quand nous \'etudierons les 
question de r\'ecurrence, l'hypoth\`ese de continuit\'e par morceaux simplifiera 
nettement la th\'eorie. Cette hypoth\`ese est amplement suffisante pour les applications.  
\end{remark}


\section{Exemples de \CMs\ \`a espace continu}
\label{sec:cont_ex} 

Voici quelques exemples simples de \CMs\ \`a espace d'\'etats continu. 

\begin{example}[Variables i.i.d.]
Soit $\mu$ une densit\'e de probabilit\'e sur $\cX$, et soit 
\begin{equation}
 p(x,y) = \mu(y) \qquad \forall x, y\in\cX\;.
\end{equation} 
Il est imm\'ediat de v\'erifier que $p$ est un noyau markovien \`a densit\'e. 
Pour tout $x\in\cX$ et tout bor\'elien $A\subset\cX$, on a 
\begin{equation}
 \bigprobin{x}{X_1\in A}
 = \int_A \mu(x_1)\6x_1 
 = \mu(A)\;,
\end{equation}
ce qui montre que $X_1$ a la densit\'e $\mu$. De plus, on trouve
\begin{equation}
 p^2(x,x_2) = \int_\cX p(x,x_1)p(x_1,x_2)\6x_1 
 = \int_\cX \mu(x_1)\mu(y) \6x_1 
 = \mu(y)\;.
\end{equation}
Plus g\'en\'eralement, on v\'erifie par r\'ecurrence que pour tout $n\geqs1$, on a 
\begin{equation}
 p^n(x,x_n) = \mu(x_n) \qquad \forall x, x_n\in\cX\;.
\end{equation} 
Par cons\'equent, les variables al\'eatoires $X_1, X_2, \dots$ ont toutes la m\^eme loi,
de densit\'e $\mu$. De plus, on v\'erifie facilement que 
\begin{equation}
 \bigprobin{x}{X_1\in A_1,\dots, X_n\in A_n} 
 = \mu(A_1)\dots \mu(A_n)
\end{equation} 
pour tout $n\geqs1$, et tout choix de bor\'eliens $A_1$, \dots, $A_n$. Les variables 
$X_1, X_2, \dots$ sont donc ind\'ependantes. 
\end{example}

\begin{example}[Marche al\'eatoire \`a pas Gaussiens]
Prenons $\cX = \R$. Supposons que $X_0 = 0$ et que 
\begin{equation}
 X_{n+1} = X_n + Y_{n+1} \qquad \forall n\in\N\;,
\end{equation} 
o\`u les $Y_n$ sont i.i.d., de loi normale centr\'ee et de variance $\sigma^2$. 
En d'autres termes, on a 
\begin{equation}
 X_n = \sum_{m=1}^n Y_m \qquad \forall n\in\N\;.
\end{equation} 
On dit que $(X_n)_{n\geqs0}$ est une \defwd{marche al\'eatoire \`a pas Gaussiens sur $\R$}. 
Montrons que c'est une \CM\ sur $\R$. La propri\'et\'e de Markov suit de l'ind\'ependance 
des $Y_n$. Les probabilit\'es de transition sont donn\'ees par 
\begin{align}
\bigpcond{X_{n+1}\in A}{X_n = x}
&= \bigpcond{X_n + Y_{n+1}\in A}{X_n = x} \\
&= \bigpcond{x + Y_{n+1}\in A}{X_n = x} \\
&= \bigprob{Y_{n+1}\in A - x} \\
&= \int_{A-x} \frac{\e^{-y^2/(2\sigma^2)}}{\sqrt{2\pi\sigma^2}} \6y \\
&= \int_{A} \frac{\e^{-(z-x)^2/(2\sigma^2)}}{\sqrt{2\pi\sigma^2}} \6z\;, 
\end{align}
o\`u $A-x = \setsuch{y-x}{y\in A}$. 
Il suit que le noyau de transition de la \CM\ est donn\'e par 
\begin{equation}
 p(x,y) = \frac{\e^{-(y-x)^2/(2\sigma^2)}}{\sqrt{2\pi\sigma^2}}\;.
\end{equation} 
Plus g\'en\'eralement, si les $Y_n$ ont une densit\'e $\mu$, alors 
$p(x,y) = \mu(y-x)$. 
\end{example}

\begin{example}[Mod\`ele auto-r\'egressif AR(1)]
\label{ex:AR1} 
Les \defwd{mod\`eles auto-r\'egressifs} sont couramment utilis\'es en statistiques, 
en \'econom\'etrie et en traitement du signal. Le mod\`ele AR(1) en est un cas particulier,  
o\`u le param\`etre $1$, appel\'e \defwd{ordre}, d\'esigne le temps de m\'emoire. 
Il est d\'efini par la relation de r\'ecurrence 
\begin{equation}
 X_{n+1} = aX_n + Y_{n+1}\;,
\end{equation} 
o\`u $a\in\R$, et les $Y_n$ sont i.i.d., de loi normale centr\'ee et de variance $\sigma^2$. 
Il suit alors d'un calcul analogue \`a celui de l'exemple pr\'ec\'edent que 
\begin{align}
\bigpcond{X_{n+1}\in A}{X_n = x}
&= \bigprob{Y_{n+1}\in A - ax} \\
&= \int_{A} \frac{\e^{-(z-ax)^2/(2\sigma^2)}}{\sqrt{2\pi\sigma^2}} \6z\;.
\end{align}
Le noyau de transition du mod\`ele AR(1) est donc donn\'e par 
\begin{equation}
 p(x,y) = \frac{\e^{-(y-ax)^2/(2\sigma^2)}}{\sqrt{2\pi\sigma^2}}\;.
\end{equation} 
Plus g\'en\'eralement, le mod\`ele autor\'egressif d'ordre $p$, AR($p$), est d\'efini par 
\begin{equation}
 X_{n+1} = \sum_{i=1}^p a_iX_{n-i} + Y_{n+1}\;.
\end{equation} 
Si $p\geqs2$, la suite des $X_n$ n'est pas une \CM, puisque la valeur de $X_{n+1}$ 
d\'epend des valeurs \`a $p$ temps pr\'ec\'edents. Toutefois, les vecteurs 
$Z_n = (X_n, X_{n+1}, \dots, X_{n+p-1})$ d\'efinissent une \CM\ sur $\R^p$. 
\end{example}

\begin{example}[Applications it\'er\'ees bruit\'ees]
Une autre g\'en\'eralisation du mod\`ele AR(1) est donn\'ee par la relation de 
r\'ecurrence 
\begin{equation}
 X_{n+1} = F(X_n) + Y_{n+1}\;,
\end{equation} 
o\`u $F:\R\to\R$, et les $Y_n$ sont \`a nouveau i.i.d., de loi normale centr\'ee 
et de variance $\sigma^2$. Il s'agit d'une \CM\ de noyau de transition 
\begin{equation}
 p(x,y) = \frac{\e^{-(y-F(x))^2/(2\sigma^2)}}{\sqrt{2\pi\sigma^2}}\;.
\end{equation} 
On peut \'evidemment consid\'erer d'autres lois pour les $Y_n$ que la loi normale. 
Ce genre de mod\`ele appara\^it par exemple en dynamique des populations, ou en 
\'epid\'emiologie. Sa dynamique d\'epend fortement des propri\'et\'es de $F$ 
(points fixes, stabilit\'e). 
\end{example}


\chapter{Probabilit\'es invariantes et vitesse de convergence}
\label{chap:cont_conv} 

La principale difficult\'e des \CMs\ \`a espace continu, par rapport aux \CMs\ 
\`a espace d\'enombrable, est que l'on a $\probin{x}{X_n = y} = 0$ pour tout 
choix de $x, y\in\cX$ et de $n\in\N^*$. Par cons\'equent, l'esp\'erance du temps 
de passage en un point diff\'erent du point de d\'epart est en g\'en\'eral 
infinie. La solution consiste \`a ne pas consid\'erer les temps de premier passage 
en des points, mais en des ensembles ouverts. C'est ce que nous \'etudierons plus 
en d\'etail dans la Section~\ref{sec:cont_rec}. Avec cette modification, la 
th\'eorie des fonctions de Lyapounov s'applique sans grandes modifications, comme 
nous allons le voir dans la Section~\ref{sec:cont_conv}. 


\section{Irr\'eductibilit\'e et r\'ecurrence de Harris}
\label{sec:cont_rec} 

Nous consid\'erons dans cette section une \CM\ $(X_n)_{n\geqs0}$ 
sur un ouvert $\cX$ de $\R^d$, de noyau de transition \`a densit\'e $p$. 
Le d\'efinition du temps de premier passage est la m\^eme 
que dans le cas d\'enombrable, mais nous la rappelons n\'eanmoins ici. 

\begin{definition}[Temps de premier passage]
Soit $A\subset\cX$ un bor\'elien. Alors le 
\defwd{temps de premier passage en $A$} de la \CM\ $(X_n)_{n\geqs0}$ est 
la variable al\'eatoire
\begin{equation}
 \tau_A = \inf\setsuch{n\geqs1}{X_n \in A}
 \in\N^*\cup\set{\infty}\;.
\end{equation} 
\end{definition}

La d\'efinition de l'irr\'eductibilit\'e est en revanche l\'eg\`erement 
diff\'erente de celle du cas discret, en raison du fait que les probabilit\'es 
de transition vers des points sont nulles. 

\begin{definition}[Irr\'eductibilit\'e d'une \CM\ \`a espace continu]
La \CM\ $(X_n)_{n\geqs0}$ est dite \defwd{irr\'eductible} si pour 
tout $x\in\cX$ et tout ouvert $A\subset\cX$, il existe un $n\in\N^*$ 
tel que $\probin{x}{X_n\in A} > 0$. 
De mani\`ere \'equivalente, pour tout $x\in\cX$ et $A\subset\cX$ ouvert, 
il existe un $n\in\N^*$ tel que $\probin{x}{\tau_A \leqs n} > 0$. 
\end{definition}

Remarquons que si $p(x,y) > 0$ pour tout $x,y\in\cX$, alors la \CM\ 
est irr\'eductible. C'est le cas pour tous les exemples du chapitre 
pr\'ec\'edent faisant intervenir des variables Gaussiennes. 
Nous pouvons maintenant donner les analogues continus des d\'efinitions 
de r\'ecurrence et de r\'ecurrence positive. 

\begin{definition}[R\'ecurrence (positive) au sens de Harris]
\begin{itemize}
\item  La \CM\ $(X_n)_{n\geqs0}$ est \defwd{Harris--r\'ecurrente} si 
\begin{equation}
 \bigprobin{x}{\tau_A < \infty} = 1
\end{equation} 
pour tout $x\in\cX$ et tout ouvert $A\subset\cX$. 

\item   La \CM\ $(X_n)_{n\geqs0}$  est \defwd{Harris--r\'ecurrente positive} 
si de plus 
\begin{equation}
\expecin{x}{\tau_A} < \infty
\end{equation}
pour tout $x\in\cX$ et tout ouvert $A\subset\cX$. 
\end{itemize}
\end{definition}

Remarquons que contrairement au cas discret, la d\'efinition 
fait intervenir le temps de passage en tout ensemble ouvert $A$.  
Par cons\'equent, une \CM\ Harris--r\'ecurrente est automatiquement 
irr\'e\-ducti\-ble, puisque $\bigprobin{x}{\tau_A < \infty} = 1$
implique $\probin{x}{\tau_A \leqs n} > 0$ pour un $n$ fini. 
L'int\'er\^et principal de cette d\'efinition est li\'e aux mesures et 
probabilit\'es invariantes, d\'efinies comme suit.

\begin{definition}[Mesure et probabilit\'e invariantes]
Une mesure $\mu$ sur $\cX$ est \defwd{invariante} si $\mu\cP = \mu$, 
c'est-\`a-dire si  
\begin{equation}
 \int_{\cX} \mu(x) p(x,y) \6x = \mu(y)
 \qquad 
 \forall y\in\cX\;.
 \label{def:proba_inv_cont} 
\end{equation} 
Si $\mu$ est une mesure de probabilit\'e, alors on dit que c'est une 
\defwd{probabilit\'e invariante}. 
\end{definition}

\begin{theorem}[R\'ecurrence, mesures invariantes et probabilit\'es invariantes]
\label{thm:Harris} 
Si la \CM\ $(X_n)_{n\geqs0}$ est Harris--r\'ecurrente, alors elle admet
une mesure invariante $\mu$. Si elle est de plus Harris--r\'ecurrente positive, 
alors elle admet une probabilit\'e invariante $\pi$. De plus, $\pi$ est 
essentiellement unique, c'est-\`a-dire que si $\pi'$ est une autre probabilit\'e 
invariante, alors $\pi'(A) = \pi(A)$ pour tout ouvert $A\subset\cX$. 
\end{theorem}

Afin de pr\'eparer la d\'emonstration de ce r\'esultat, nous introduisons la notion 
de processus tu\'e en touchant un sous-ensemble de $\cX$. 

\begin{definition}[Noyau du processus tu\'e en touchant $B\subset\cX$]
Soit $B$ un bor\'elien de $\cX$, $B^c = \cX\setminus B$, 
et soit $p^\dagger = p^\dagger_B$ la fonction d\'efinie par 
\begin{equation}
p^\dagger(x,y) =
p^\dagger_B(x,y) = 
\begin{cases}
  p(x,y) & \text{si $y\in B^c$\;,} \\
  0 & \text{si $y\in B$\;.}
\end{cases}
\label{eq:def_pkilled} 
\end{equation} 
On d\'efinit par r\'ecurrence des noyaux $p^\dagger_n$ par 
$p^\dagger_1 = p^\dagger$ et 
\begin{equation}
 p^\dagger_{n+1}(x,y) = \int_{B^c} p^\dagger_n(x,z) p^\dagger(z,y) \6z
 \qquad \forall n\in\N^*\;.
\label{eq:def_pkilledn} 
\end{equation} 
\end{definition}

Notons que les noyaux $p^\dagger_n$ ne sont pas en g\'en\'eral 
markoviens, car leur int\'egrale par rapport \`a $y$ est en g\'en\'eral 
strictement inf\'erieure \`a $1$. On dit que ce sont des noyaux 
\defwd{sous-markoviens}. Leur int\'er\^et pour nous est le lemme suivant.

\begin{lemma}[Processus tu\'e et loi de $\tau_B$]
\label{lem:processus_tue} 
Pour tout $n\geqs1$, tout $x\in\cX$, et tout bor\'elien $A\subset \cX$ 
tel que $A\cap B = \varnothing$, on a
\begin{equation}
 \bigprobin{x}{X_n\in A, \tau_B > n} = \int_A p^\dagger_n(x,y)\6y\;.
\end{equation} 
\end{lemma}
\begin{proof}
Cela suit du fait que 
\begin{align}
\bigprobin{x}{X_n\in A, \tau_B > n} 
&= \bigprobin{x}{X_1\notin B, X_2\notin B, \dots, X_{n-1}\notin B, X_n\in A\setminus B} \\
&= \int_{B^c} \int_{B^c} \dots \int_{B^c} \int_{A\setminus B}
p(x,x_1) p(x_1,x_2) \dots p(x_{n-1},x_n) \6x_n \6x_{n-1} \dots \6x_2 \6x_1 \\
&= \int_{\cX} \int_{\cX} \dots \int_{\cX} \int_A
p^\dagger(x,x_1) p^\dagger(x_1,x_2) \dots p^\dagger(x_{n-1},x_n) \6x_n \6x_{n-1} \dots \6x_2 \6x_1 
\end{align}
en vertu de~\eqref{eq:def_pkilled}, et puisque $A\setminus B = A$.  
Par~\eqref{eq:def_pkilledn} et une r\'ecurrence sur $n$, ceci est bien \'egal \`a l'int\'egrale sur $A$ de $p^\dagger_n$ par rapport \`a sa seconde variable. 
\end{proof}

\begin{remark}[Processus de Markov tu\'e]
On peut associer \`a $(X_n)_{n\geqs0}$ un \defwd{processus tu\'e en touchant $B$}, 
not\'e $(X^\dagger_n)_{n\geqs0}$, de la mani\`ere suivante. On ajoute \`a $\cX$ un 
\defwd{\'etat cimeti\`ere} $\dagger$, qui est absobant, et on pose 
\begin{equation}
 X^\dagger_n = 
 \begin{cases}
  X_n & \text{si $n<\tau_B$\;,}\\
  \dagger & \text {si $n\geqs \tau_B$\;.}
 \end{cases}
\end{equation}
Son noyau restreint \`a $B^c$ est alors $p^\dagger$, et on a 
$\bigprobin{x}{X_n\in A, \tau_B > n} = \bigprobin{x}{X_n^\dagger\in A}$
si $A\cap B = \varnothing$.
\end{remark}

Un objet important li\'e au processus tu\'e est le noyau de potentiel, qui joue un 
r\^ole similaire \`a celui de la matrice fondamentale d'une \CM\ absorbante.

\begin{definition}[Noyau de potentiel]
Soit $B\subset\cX$ un ouvert. Le \defwd{noyau de potentiel} de la \CM\ $(X_n)_{n\geqs0}$ 
relatif \`a $B$ est l'application qui associe \`a chaque $x\in\cX$ et chaque bor\'elien 
$A\subset\cX$ le nombre 
\begin{equation}
\label{eq:def_GB} 
 G_B(x,A) = \biggexpecin{x}{\sum_{n=0}^{\tau_B-1}\indicator{X_n\in A}}
 \in [0,\infty]\;.
\end{equation} 
\end{definition}

Le lien entre noyau de potentiel et processus tu\'e est le suivant. 

\begin{proposition}[Densit\'e du noyau de potentiel]
Pour $x\in B$, le noyau de potentiel $G_B(x,\cdot)$ est une mesure (pas n\'ecessairement finie), 
qui admet sur $B^c$ la densit\'e 
\begin{equation}
\label{eq:densite_GB} 
 g_B(x,y) = \sum_{n=1}^\infty p_n^\dagger(x,y)\;,
\end{equation} 
pour tous les $y$ tels que cette s\'erie converge. 
\end{proposition}
\begin{proof}
Notons tout d'abord que $\indicator{X_n\in A_1\cup A_2} \leqs \indicator{X_n\in A_1} 
+ \indicator{X_n\in A_2}$, avec \'egalit\'e si $A_1\cap A_2 = \varnothing$. 
Ceci montre que $G_B(x,\cdot)$ est une mesure, puisque $G_B(x,A_1 \cup A_2) 
\leqs G_B(x,A_1) + G_B(x,A_2)$, avec \'egalit\'e si $A_1\cap A_2 = \varnothing$.

Soit maintenent $A\subset\cX$ un bor\'elien tel que $A\cap B = \varnothing$. 
Alors on a 
\begin{align}
 \int_A g_B(x,y) \6y 
 &= \sum_{n=1}^\infty \int_A p_n^\dagger(x,y) \6y \\
 &= \sum_{n=1}^\infty \bigprobin{x}{X_n\in A, \tau_B > n} \\
 &= \biggexpecin{x}{\sum_{n=1}^\infty \indicator{X_n\in A, \tau_B > n}} \\
 &= \biggexpecin{x}{\sum_{n=1}^{\tau_B-1} \indicator{X_n\in A}} 
 = G_B(x,A)\;.
\label{eq:demo_densite_muB} 
\end{align}
Pour obtenir la derni\`ere ligne, nous avons utilis\'e le fait que 
$A\cap B = \varnothing$ (et donc $x\notin A$). Ceci montre que $g_B(x,y)$ est bien 
la densit\'e de $G_B(x,\cdot)$, du moins sur $B^c$.
\end{proof}

Nous sommes maintenant en mesure de donner une d\'emonstration (au moins partielle) 
du Th\'eor\`eme~\ref{thm:Harris} (nous admettrons l'unicit\'e essentielle). 

\begin{proof}[\textit{D\'emonstration du Th\'eor\`eme~\ref{thm:Harris}}]
Le d\'emonstration est inspir\'ee de la construction du cas discret, 
reposant sur les mesures donn\'ees par~\eqref{eq:gamma(y)}. Fixons un 
$x\in\cX$, et soit $B_\eps = \cB_\eps(x)$ la boule de centre $x$ et de rayon 
$\eps>0$. Pour tout bor\'elien $A\subset\cX$, nous posons 
\begin{equation}
\label{eq:mu_eps} 
 \mu_\eps(A) = G_{B_\eps}(x,A)\;.
\end{equation}
Notre but est de montrer que $\mu_\eps$ converge vers une mesure invariante 
lorsque $\eps\to0$. 
Commen\c cons par remarquer que 
\begin{equation}
 \mu_\eps(B_\eps) = 1\;.
\end{equation} 
En effet, l'hypoth\`ese de r\'ecurrence de Harris implique que 
$\probin{x}{\tau_{B_\eps} < \infty} = 1$, et par cons\'equent la somme~\eqref{eq:def_GB} 
a presque s\^urement un nombre fini de termes, dont seul le dernier contribue 
\`a $\mu_\eps(B_\eps)$. D'autre part, $\mu_\eps$ admet sur $B^c$ la densit\'e donn\'ee par~\eqref{eq:densite_GB}.
Nous observons maintenant que pour tout $y\in B_\eps^c$, on a 
\begin{align}
 \mu_\eps(y) 
 &= p^\dagger(x,y) + \sum_{n=2}^\infty p_n^\dagger(x,y) \\
 &= p(x,y) + \sum_{n=2}^\infty \int_{B_\eps^c} p_{n-1}^\dagger(x,z) p^\dagger(z,y) \6z \\
 &= p(x,y) + \int_{B_\eps^c} \sum_{m=1}^\infty p_{m}^\dagger(x,z) p(z,y) \6z \\
 &= p(x,y) + \int_{B_\eps^c} \mu_\eps(z) p(z,y) \6z\;.
\end{align}
Nous avons utilis\'e \`a deux reprises le fait que $p^\dagger(z,y) = p(z,y)$ pour tout 
$z\in\cX$, puisque $y\in B_\eps^c$. Il suit que 
\begin{equation}
 \mu_\eps(y) - \int_{\cX} \mu_\eps(z) p(z,y) \6z
 = p(x,y) - \int_{B_\eps} \mu_\eps(z) p(z,y) \6z\;.
\end{equation} 
Or, comme $\mu_\eps(B_\eps) = 1$, $\mu_\eps$ est une mesure de probabilit\'e 
sur $B$. Par cons\'equent, 
\begin{equation}
\inf_{z\in B_\eps} p(z,y) \leqs 
\int_{B_\eps} \mu_\eps(z) p(z,y) \6z
\leqs \sup_{z\in B_\eps} p(z,y)\;.
\end{equation} 
Il suit que pour tout $x$ en lequel $x\mapsto p(x,y)$ est continue, on a 
\begin{equation}
 \lim_{\eps\to0} \int_{B_\eps} \mu_\eps(z) p(z,y) \6z = p(x,z)\;.
\end{equation} 
Par cons\'equent, on a pour ces $x$
\begin{equation}
 \lim_{\eps\to0} 
 \biggbrak{\mu_\eps(y) - \int_{\cX} \mu_\eps(z) p(z,y) \6z} = 0\;.
\end{equation} 
Ceci montre que la limite de $\mu_\eps$ lorsque $\eps\to0$ est invariante 
pour presque tout $x$ (en tout point de continuit\'e de $p$, mais la valeur 
de $\mu_\eps$ en des points isol\'es n'influe pas sur les probabilit\'es). 

Consid\'erons finalement le cas Harris--r\'ecurrent positif. Alors on a 
\begin{equation}
 \mu_\eps(\cX) = \biggexpecin{x}{\sum_{n=0}^{\tau_{B_\eps}-1} \indicator{X_n\in\cX}}
 = \bigexpecin{x}{\tau_{B_\eps}}\;.
\end{equation} 
On peut alors prendre 
\begin{equation}
 \pi_\eps(y) = \frac{1}{\bigexpecin{x}{\tau_{B_\eps}}} \mu_\eps(x)\;.
\end{equation} 
C'est la densit\'e d'une mesure de probabilit\'e, qui converge vers une probabilit\'e 
invariante lorsque $\eps\to0$. 
\end{proof}

\begin{remark}[Hypoth\`eses de r\'ecurrence]
\label{rem:hypo_recurrence} 
Dans la d\'emonstration, nous n'avons pas utilis\'e les hypoth\`eses de 
r\'ecurrence (positive) de Harris dans toute leur g\'en\'eralit\'e. En fait, 
nous avons seulement suppos\'e qu'il existe un point particulier $x\in\cX$ 
et un $\eps_0 > 0$ tels que pour tout $\eps\in]0,\eps_0[$, le temps de 
passage dans la boule de rayon $\eps$ centr\'ee en $x$, partant de $x$, 
est presque s\^urement fini, respectivement d'esp\'erance finie. 
\end{remark}

Le r\'esultat suivant permet d'exprimer des esp\'erances sous la 
probabilit\'e invariante en termes d'excursions vers un ensemble fix\'e $B$. 

\begin{proposition}[Esp\'erance de fonctions test]
\label{prop:Nummelin} 
Soit $(X_n)_{n\geqs0}$ une \CM\ Harris--r\'ecurrente positive, $\pi$ 
son unique probabilit\'e invariante, et $f:\cX\to\R_+$. Alors pour tout
ouvert $B\subset\cX$, on a 
\begin{equation}
\label{eq:Nummelin} 
 \pi(f) 
 = \int_B \pi(x) \biggexpecin{x}{\sum_{n=0}^{\tau_B-1} f(X_n)} \6x 
 = \int_B \pi(x) \biggexpecin{x}{\sum_{n=1}^{\tau_B} f(X_n)} \6x \;.
\end{equation} 
\end{proposition}
\begin{proof}
Montrons par r\'ecurrence que pour tout $N\in\N$, on a 
\begin{equation}
 \pi(x) = \pi(x)\indicator{x\in B}
 + \sum_{n=1}^N \int_B \pi(y) p_n^\dagger(y,x)\6y 
 + \int_{\cX} \pi(y) p^\dagger_{N+1}(y,x)\6y\;.
\label{eq:rec_pi} 
\end{equation} 
L'initialisation suit de la d\'ecomposition 
$\pi(x) = \pi(x)\indicator{x\in B} + \pi(x)\indicator{x\in B^c}$ et du fait que 
\begin{equation}
 \pi(x)\indicator{x\in B^c}
 = (\pi\cP)(x)\indicator{x\in B^c}
 = \int_\cX \pi(y)p(y,x)\indicator{x\in B^c} \6y 
 = \int_\cX \pi(y)p^\dagger(y,x) \6y\;. 
\end{equation} 
L'h\'er\'edit\'e vient de 
\begin{align}
\int_\cX \pi(y) p^\dagger_{N+1}(y,x)\6y 
&= \int_\cX \biggbrak{\pi(y)\indicator{y\in B} 
+ \int_\cX \pi(z)p^\dagger(z,y)\6z} p^\dagger_{N+1}(y,x)\6y \\
&= \int_B \pi(y) p_{N+1}^\dagger(y,x)\6y 
 + \int_{\cX} \pi(y) p^\dagger_{N+2}(y,x)\6y\;.
\end{align}
Faisons alors tendre $N$ vers l'infini dans~\eqref{eq:rec_pi}. Le Lemme~\ref{lem:processus_tue}
montre que 
\begin{equation}
 \lim_{N\to\infty} \int_A p^\dagger_{N+1}(y,x)\6x 
 = \lim_{N\to\infty} \bigprobin{y}{X_{N+1}\in A, \tau_B > N+1} 
 = 0
\end{equation} 
pour tout bor\'elien $A$, par r\'ecurrence du processus. Il suit que $p^\dagger_{N+1}(y,x)$
tend vers $0$, d'o\`u 
\begin{equation}
 \pi(x) = \pi(x)\indicator{x\in B}
 + \int_B \pi(y) g_B(y,x) \6y\;, 
\end{equation} 
o\`u $g_B$ est la densit\'e~\eqref{eq:densite_GB} du noyau de potentiel $G_B$. 
En int\'egrant cette relation contre $f$, il vient, en permutant les variables $x$ et $y$, 
\begin{align}
 \pi(f) 
 &= \int_B \pi(x)f(x) \6x 
 + \int_{B^c} \int_B \pi(y) g_B(y,x)\6y f(x) \6x \\
 &= \int_B \pi(x) \biggbrak{f(x) + \int_{B^c} g_B(x,y) f(y)\6y} \6x\;.
\end{align}
Or on a 
\begin{align}
 f(x) + \int_{B^c} g_B(x,y) f(y)\6y
 = \bigexpecin{x}{f(X_0)} + \biggexpecin{x}{\sum_{n=1}^{\tau_B-1}f(X_n)}\;.
\end{align}
Ceci montre la premi\`ere \'egalit\'e dans~\eqref{eq:Nummelin}. La seconde 
\'egalit\'e vient du fait que le terme $n=0$ de la premi\`ere somme est \'egal 
au terme $n=\tau_B$ de la seconde. 
\end{proof}

\begin{remark}[Lien entre probabilit\'e invariante et temps de r\'ecurrence moyen]
En prenant $f = 1$ dans~\eqref{eq:Nummelin} avec $B = B_\eps = \cB_\eps(x)$, il vient 
\begin{equation}
 \int_B \pi(x) \bigexpecin{x}{\tau_B} \6x = 1\;.
\end{equation}  
Il suit de~\eqref{def:proba_inv_cont} que $\pi$ est continue presque partout 
(en tout point de continuit\'e de $y\mapsto p(x,y)$). En ces points, le th\'eor\`eme 
de la valeur moyenne implique 
\begin{equation}
\pi(x) 
 = \lim_{\eps\to0} \frac{1}{\abs{\cB_\eps(x)}\bigexpecin{x}{\tau_{B_\eps}}}\;.
\end{equation} 
C'est l'analogue de la relation ~\eqref{eq:piEtau} du cas discret. 
\end{remark}

\begin{example}[Processus auto-r\'egressif AR(1)]
Nous avons d\'ej\`a observ\'e que le processus AR(1) \'etait Harris--r\'ecurrent, 
puisque sa densit\'e de transition est minor\'ee par une constante strictement 
positive sur tout compact. Comme les $Y_n$ sont Gaussiennes, et que toute somme 
de Gaussiennes est encore Gaussienne, on s'attend \`a avoir une probabilit\'e 
invariante Gaussienne. En fait, on a la relation de r\'ecurrence 
\begin{equation}
 \variance(X_{n+1}) = a^2 \variance(X_n) + \sigma^2\;.
\end{equation} 
On v\'erifie par r\'ecurrence que 
\begin{equation}
 \variance(X_n) = a^{2n}\variance(X_0) + \frac{\sigma^2}{1-a^2}\;. 
\end{equation} 
Ainsi, si $\abs{a} < 1$, la loi de $X_n$ converge vers une loi normale centr\'ee 
de variance $\sigma^2/(1-a^2)$. C'est aussi la probabilit\'e invariante. On notera 
que dans le cas $a=1$, on obtient la marche al\'eatoire \`a pas Gaussiens. Dans ce cas, 
la loi de $X_n$ ne converge pas, et la \CM\ n'est pas Harris--r\'ecurrente positive. 
\end{example}


\section{Fonctions de Lyapounov et vitesse de convergence}
\label{sec:cont_conv} 

Dans cette section, nous consid\'erons une \CM\ sur $\cX\subset\R^d$, admettant 
une densit\'e de transition $p$ continue par morceaux. Par simplicit\'e, nous 
la supposerons \'egalement irr\'eductible, m\^eme si certains r\'esultats peuvent 
\^etre \'etendus \`a des situations plus g\'en\'erales. 

Il s'av\`ere que l'approche par fonctions de Lyapounov \`a l'\'etude de propri\'et\'es 
de r\'ecurrence et de convergence des lois se transpose assez facilement au cas 
d'un espace continu. La d\'efinition de fonction de Lyapounov est la m\^eme que dans 
le cas discret. 

\begin{definition}[Fonction de Lyapounov]
Une \defwd{fonction de Lyapounov} est une fonction 
$V: \cX\to \R_+ = [0,\infty[$ satisfaisant
\begin{equation}
\label{eq:gen} 
 V(x) \to +\infty 
 \qquad \text{pour $\norm{x}\to\infty$\;.}
\end{equation} 
\end{definition}

La d\'efinition de g\'en\'erateur s'adapte aussi tr\`es facilement au cas continu. 

\begin{definition}[G\'en\'erateur]
Le \defwd{g\'en\'erateur} $\cL$ d'une \CM\ sur un ensemble ouvert $\cX\subset\R^d$ est 
d\'efini, pour toute fonction $f:\cX\to\R$, par 
\begin{equation}
 (\cL f)(x) 
 = (\cP f)(x) - f(x) 
 = \int_{\cX} p(x,y)f(y)\6y - f(x)\;.
\end{equation} 
\end{definition}

Les trois r\'esultats suivants, concernant la formule de Dynkin, la croissance sous-exponen\-tielle 
et la non-explosion, restent inchang\'es par rapport au cas discret, avec essentiellement 
les m\^emes d\'emonstrations. Nous r\'ep\'etons donc simplement ici leurs \'enonc\'es. 

\begin{proposition}[Formule de Dynkin]
\label{prop:Dynkin_cont} 
Pour toute fonction de Lyapounov $V$, on a 
\begin{equation}
\label{eq:Dynkin_cont} 
 \bigexpecin{x}{V(X_n)} 
 = V(x) + \biggexpecin{x}{\sum_{m=0}^{n-1} (\cL V)(X_m)}\;, 
\end{equation} 
De plus, si $\tau$ est un temps d'arr\^et tel que $\expecin{x}{\tau} < \infty$, alors 
\begin{equation}
 \bigexpecin{x}{V(X_\tau)} 
 = V(x) + \biggexpecin{x}{\sum_{m=0}^{\tau-1} (\cL V)(X_m)}\;.
\end{equation} 
\end{proposition}

\begin{theorem}[Croissance sous-exponentielle]
\label{thm:sous_exp_cont} 
Supposons qu'il existe une fonction de Lyapounov $V$ et $c > 0$, $d\geqs0$ tels que 
\begin{equation}
 (\cL V)(x) \leqs c V(x) + d
 \qquad \forall x\in\cX\;.
\end{equation} 
Alors on a 
\begin{equation}
 \bigexpecin{x}{V(X_n)} \leqs (1+c)^n V(x) + \frac{(1+c)^n-1}{c}d
\end{equation} 
pour tout $n\in\N$ et tout $x\in\cX$. 
\end{theorem}

\begin{theorem}[Non-explosion]
\label{thm:non_explosion_cont} 
Supposons qu'il existe $d \geqs 0$ et un ensemble born\'e $K\subset\cX$ tel que 
pour tout $x\in\cX$, on ait 
\begin{equation}
 (\cL V)(x) 
 \leqs d \indicator{K}(x) 
 = 
 \begin{cases}
  d    & \text{si $x\in K$\;,} \\
  0    & \text{sinon\;.}
 \end{cases}
\end{equation} 
Alors 
\begin{equation}
 \biggprobin{x}{\lim_{n\to\infty} \norm{X_n} = \infty} = 0
 \qquad \forall x\in\cX\;.
\end{equation} 
\end{theorem}

Le r\'esultat de r\'ecurrence positive et sa d\'emonstration doivent \^etre tr\`es l\'eg\`erement adapt\'es. La principale diff\'erence est que l'ensemble $K$ doit \^etre un ouvert born\'e. 
On pourrait \'egalement prendre un compact d'int\'erieur non vide, le point important \'etant que $K$ doit \^etre born\'e et contenir un ensemble ouvert, afin de pouvoir appliquer la Harris--r\'ecurrence. 

\begin{theorem}[R\'ecurrence positive]
\label{thm:rec_pos_cont} 
Soit $f: \cX\to[1,\infty[$ et $V$ une fonction de Lyapounov telle que 
\begin{equation}
 (\cL V)(x) \leqs -cf(x) + d\indicator{K}(x) 
 \qquad \forall x\in \cX\;,
\end{equation} 
pour un ouvert born\'e $K\subset \cX$ et des constantes $c>0$ et $d\geqs0$. 
Supposons de plus qu'il existe $\delta > 0$ tels que 
\begin{equation}
\label{eq:proba_lb_cont}
p(x,y) \geqs \delta    
\qquad \forall x, y\in K\;.
\end{equation} 
Alors la \CM\ est Harris--r\'ecurrente positive, et admet donc une mesure de 
probabilit\'e invariante $\pi$. De plus, 
\begin{equation}
 \pi(f) < \infty\;.
\end{equation} 
\end{theorem}
\begin{proof}
Nous allons consid\'erer d'abord le passage en $K$, puis celui en $A\subset K$, 
puis celui en un $A$ g\'en\'eral. 
\begin{enumerate}
\item  Fixons $x\in\cX$, et soit $T\in\N^*$. Nous noterons $\tau_K\wedge T = 
\min\set{\tau_K, T}$. Alors la formule de Dynkin implique 
\begin{align}
 0 \leqs \bigexpecin{x}{V(X_{\tau_K\wedge T})}
 &= V(x) + \biggexpecin{x}{\sum_{m=0}^{\tau_K\wedge T -1} (\cL V)(X_m)} \\
 \label{eq:EVXfm} 
 &\leqs V(x) - c \biggexpecin{x}{\sum_{m=0}^{\tau_K\wedge T -1} f(X_m)} + d \\
 &\leqs V(x) - c \bigexpecin{x}{\tau_K\wedge T} + d\;.
\end{align}
Par cons\'equent, on a 
\begin{equation}
 \bigexpecin{x}{\tau_K\wedge T} \leqs \frac{V(x)+d}{c}
\end{equation} 
pour tout $T\in\N^*$. Comme le membre de droite ne d\'epend pas de $T$, on 
obtient, en faisant tendre $T$ vers l'infini, 
\begin{equation}
\label{eq:EtauK} 
 \bigexpecin{x}{\tau_K} \leqs \frac{V(x)+d}{c}
 \qquad\forall x\in\cX\;.
\end{equation} 

\item   Soit $\tau_{K,n}$ le temps du $n$i\`eme passage de la \CM\ en 
$K$, qui est d'esp\'erance finie en vertu de~\eqref{eq:EtauK}. Soit $(Y_n)_{n\geqs0}$ 
la \CM\ trace sur $K$, d\'efinie par $Y_n = X_{\tau_{K,n}}$. 
Alors, pour tout $x_0\in K$ et tout ouvert $A\subset K$, on a 
\begin{equation}
 \bigprobin{x_0}{Y_1 \notin A}
 = 1 - \bigprobin{x_0}{Y_1 \in A}
 \leqs 1 - \bigprobin{x_0}{X_1\in A}
 \leqs 1 - \delta\abs{A}
\end{equation} 
par l'hypoth\`ese~\eqref{eq:proba_lb_cont}. Notons que cette borne est ind\'ependante 
du $x_0\in K$ choisi. 
Si $\hat\tau_A = \inf\setsuch{n\geqs1}{Y_n\in A}$ est le temps du premier passage 
de $Y_n$ en $A$, alors on a pour tout $n\in\N$ 
\begin{equation}
 \bigprobin{x_0}{\hat \tau_A \geqs n+1}
 = \bigprobin{x_0}{\hat \tau_A \geqs n, Y_n\notin A} 
 \leqs \bigpar{1 - \delta\abs{A}} \bigprobin{x_0}{\hat\tau_A \geqs n}\;.
\end{equation} 
Par r\'ecurrence sur $n$, on obtient alors 
\begin{equation}
 \bigprobin{x_0}{\hat\tau_A \geqs n}
 \leqs \bigpar{1 - \delta\abs{A}}^n
 \qquad \forall n\in\N\;.
\end{equation} 
Il suit que 
\begin{equation}
 \bigexpecin{x_0}{\hat \tau_A} 
 = \sum_{n\geqs0} n \bigprobin{x_0}{\hat\tau_A \geqs n}
 \leqs \sum_{n\geqs0} n \bigpar{1 - \delta\abs{A}}^n 
 < \infty\;.
\end{equation} 
Par cons\'equent, $\bigexpecin{x_0}{\tau_A < \infty}$, en vertu de~\eqref{eq:EtauK}
et du fait que $\expecin{x_0}{V(X_1)}$ est born\'e par la formule de Dynkin. 

Ici, nous pouvons invoquer la Remarque~\ref{rem:hypo_recurrence} pour conclure 
que la \CM\ admet une probabilit\'e invariante $\pi$. En effet, il suffit d'appliquer 
le r\'esultat que nous venons d'obtenir aux $A$ donn\'es par des boules de centre $x_0$ 
et de rayon $\eps$ assez petit. 

\item   Afin de montrer que la \CM\ est Harris--r\'ecurrente positive, il faut 
encore v\'erifier que $\bigexpecin{x}{\tau_A}$ est fini pour tout $x\in\cX$ et tout 
ouvert $A\subset\cX$. La majoration~\eqref{eq:EtauK} nous permet de nous limiter 
aux $x\in K$. On peut alors adapter l'argument du point pr\'ec\'edent. 
Soit $(Z_n)_{n\geqs0}$ le processus trace sur $K\cup A$. Il suit de l'hypoth\`ese 
d'irr\'eductibilit\'e que $Z_n$ va visiter $A$ avec probabilit\'e strictement 
positive au bout d'un temps assez long. Comme $K$ est born\'e, on peut 
trouver un entier $n_1$ tel que $\probin{x}{\hat\tau_A > n_1}$ soit major\'e 
par un $p<1$ pour tout $x\in K$. On peut alors proc\'eder comme au point 
pr\'ec\'edent pour montrer que $\bigexpecin{x}{\tau_A}$ est fini. 

\item   Afin de majorer $\pi(f)$, nous utilisons le Lemme~\ref{lem:majo_LH} 
avec $B = K$ et~\eqref{eq:EVXfm} pour obtenir 
\begin{equation}
 \pi(f) 
 = \int_K \pi(x) \biggexpecin{x}{\sum_{n=0}^{\tau_K-1} f(X_n)\6x} 
 \leqs \pi(K) \sup_{x\in K}\frac{V(x)+d}{c}\;,
\end{equation} 
ce qui est fini puisque $K$ est born\'e. 
\qed
\end{enumerate}
\renewcommand{\qed}{}
\end{proof}

Afin de formuler un r\'esultat de convergence, nous travaillerons \`a nouveau avec 
des normes et des distances \`a poids. Les d\'efinitions suivantes sont des adaptations 
naturelles de celles du cas discret.

\begin{definition}[Normes et distances \`a poids]
\begin{itemize}
\item   Un \defwd{poids} sur $\cX$ est une application $W:\cX\to[1,\infty[$.

\item   La norme \`a poids d'une fonction test $f$ est d\'efinie par 
\begin{equation}
 \norm{f}_W = \sup_{x\in\cX} \frac{\abs{f(x)}}{W(x)}\;.
\end{equation} 
On note $\cE_\infty^W$ l'espace de Banach des fonctions test de norme 
$\norm{f}_W$ finie. 

\item   Pour deux mesures sign\'ees finies \`a densit\'e, on d\'efinit 
\begin{align}
\rho_W(\mu,\nu)
&= \sup_{f\colon\norm{f}_W\leqs 1} \int_\cX f(x) \abs{\mu(x)-\nu(x)} \6x \\
&= \sup_{f\colon\norm{f}_W\neq 0} \frac{1}{\norm{f}_W} 
\int_\cX f(x) \abs{\mu(x)-\nu(x)} \6x \\
&= \int_x W(x) \abs{\mu(x)-\nu(x)} \6x\;.
\end{align}
\end{itemize}
\end{definition}

L'analogue continu du Th\'eor\`eme~\ref{thm:convergence} prend alors la forme suivante.
Sa d\'emonstration est la m\^eme que celle du cas discret. 

\begin{theorem}[Ergodicit\'e g\'eom\'etrique]
\label{thm:convergence_cont} 
Supposons que les deux conditions suivantes soient satisfaites.
\begin{enumerate}
\item  \textbf{Condition de d\'erive g\'eom\'etrique\,:}
Il existe $d\geqs0$, $c>0$ et une fonction de Lyapounov $V$ tels que 
\begin{equation}
\label{eq:derive_geom_cont} 
 (\cL V)(x) \leqs -c V(x) + d
 \qquad \forall x\in\cX\;.
\end{equation} 

\item   \textbf{Condition de minoration\,:}
Pour un $R > 2d/c$, soit $K = \setsuch{x\in\cX}{V(x) < R}$.
Alors il existe $\alpha\in]0,1[$ et une mesure de probabilit\'e $\nu$ telle que 
\begin{equation}
\label{eq:minoration_cont} 
 \inf_{x\in K} p(x,y) =
 \inf_{x\in K} \bigprobin{x}{X_1 = y} \geqs \alpha \nu(y)
 \qquad \forall y\in\cX\;.
\end{equation} 
\end{enumerate}

\noindent
Alors il existe des constantes $M>0$ et $\bar\gamma < 1$ telles que 
\begin{equation}
\label{eq:borne_cv_expec_cont} 
 \norm{\expecin{\cdot}{f(X_n)} - \pi(f)}_{1+V}
 \leqs M\bar\gamma^n \norm{f - \pi(f)}_{1+V}
\end{equation} 
pour toute fonction test $f\in\cE_\infty^{1+V}$. 
\end{theorem}

Les constantes $\bar\gamma$ et $M$ sont \`a nouveau d\'etermin\'ees par les 
relations~\eqref{eq:cond_gamma0}, \eqref{eq:beta_gammabar} et \eqref{eq:M}. 

\begin{example}[Processus auto-r\'egressif AR(1)]
Nous avons vu dans l'Exemple~\ref{ex:AR1} que le noyau de transition du mod\`ele 
AR(1) \'etait donn\'e par 
\begin{equation}
 p(x,y) = \frac{\e^{-(y-ax)^2/(2\sigma^2)}}{\sqrt{2\pi\sigma^2}}\;.
\end{equation} 
Prenons comme fonction de Lyapounov $V(x) = x^2$. Alors on a 
\begin{align}
 (\cL V)(x) 
 &= \frac{1}{\sqrt{2\pi\sigma^2}}
 \int_{-\infty}^{\infty} y^2 \e^{-(y-ax)^2/(2\sigma^2)} \6y - V(x) \\
 &= \frac{1}{\sqrt{2\pi}} \int_{-\infty}^{\infty} (ax+\sigma z)^2 \e^{-z^2/2} \6z 
 - x^2 \\
 &= -(1-a^2)x^2 + \sigma^2\;.
\end{align}
Pour obtenir la deuxi\`eme ligne, nous avons utilis\'e le changement 
de variables $y=ax+\sigma z$. La derni\`ere ligne suit des propri\`et\'es 
de la densit\'e d'une loi gaussienne standard. La condition de d\'erive 
g\'eom\'etrique~\eqref{eq:derive_geom_cont} est donc v\'erifi\'ee 
avec $c = 1 - a^2$ et $d = \sigma^2$. 

Pour v\'erifier la condition de minoration, nous devons choisir un 
$R > 2d/c = 2\sigma^2/(1-a^2)$. Alors nous avons $K=[-\sqrt{R},\sqrt{R}]$ et 
\begin{equation}
 \inf_{x\in K} p(x,y) 
 = \inf_{\abs{x}<\sqrt{R}}  \frac{\e^{-(y-ax)^2/(2\sigma^2)}}{\sqrt{2\pi\sigma^2}}
 = \frac{\e^{-(\abs{y}+a\sqrt{R})^2/(2\sigma^2)}}{\sqrt{2\pi\sigma^2}}\;,
\end{equation} 
l'infimum \'etant atteint pour $x=\pm \sqrt{R}$, selon le signe de $y$. 
Ceci sugg\`ere de prendre pour $\nu$ la mesure de densit\'e 
\begin{equation}
 \nu(y) = \frac{\e^{-(\abs{y}+a\sqrt{R})^2/(2\sigma^2)}}{\cN \sqrt{2\pi\sigma^2}}
 \indicator{\abs{y}\leqs \sqrt{R}}\;,
\end{equation} 
o\`u
\begin{equation}
 \cN = \frac{1}{\sqrt{2\pi\sigma^2}}\int_{-\sqrt{R}}^{\sqrt{R}} 
 \e^{-(\abs{y}+a\sqrt{R})^2/(2\sigma^2)} \6y 
 = \frac{1}{\sqrt{2\pi}}
 \int_{-\sqrt{R}/\sigma}^{\sqrt{R}/\sigma} 
 \e^{-(\abs{z}+a\sqrt{R}/\sigma)^2/2} \6y 
\end{equation} 
est la constante de normalisation assurant que $\nu$ soit une mesure 
de probabilit\'e (nous avons pos\'e $y=\sigma z$ pour obtenir la seconde in\'egalit\'e). 
En effet, la condition de minoration~\eqref{eq:minoration_cont} 
est alors satisfaite en prenant 
\begin{equation}
 \alpha = \cN\;.
\end{equation} 
En choisissant $R=4\sigma^2/(1-a^2)$, on obtient un $\alpha$ ind\'ependant de $\sigma$. 
On v\'erifie alors que cela donne un taux de convergence $1-\bar\gamma$ strictement 
positif si $a^2 < 1$, mais qui tend vers $0$ lorsque $a^2$ tend vers $1$. 
\end{example}


\section{Exercices}
\label{sec:Lyapounov_cont_exo} 

\begin{exercise}
On consid\`ere la \CM\ sur $\R$ donn\'ee par 
\[
X_{n+1} = aX_n + Y_{n+1}
\]
avec $a\in[0,\infty[$, les $Y_n$ \'etant ind\'ependantes, identiquement distribu\'ees, 
de loi uniforme sur $[-\delta,\delta]$ pour un $\delta > 0$.

\begin{enumerate}
\item   Donner les probabilit\'es de transition $p(x,y)$ de la cha\^ine.

\item   Soit la fonction de Lyapounov $V(x) = x^2$. 
Calculer $(\cL V)(x)$. 

\item   Pour quelles valeurs de $a$ la cha\^ine est-elle \`a croissance sous-exponentielle~?

\item   \`A l'aide de la formule de Dynkin, calculer la variance de $X_n$ lorsque $a=1$. 

\item   Pour quelles valeurs de $a$ la cha\^ine satisfait-elle une condition de d\'erive g\'eom\'etrique~? Quels en sont les param\`etres~?

\item   Lorsque la condition de d\'erive g\'eom\'etrique est satisfaite, trouver $\alpha\in]0,1[$,  
une mesure de probabilit\'e $\nu$, et une condition sur $p$ tels que la condition de minoration soit satisfaite. Que peut-on en d\'eduire~?
\end{enumerate} 
\end{exercise}

\begin{exercise}
On consid\`ere la cha\^ine de Markov sur $\R$ donn\'ee par 
\[
X_{n+1} = aX_n + Y_{n+1}
\]
avec $a\in[0,\infty[$, les $Y_n$ \'etant ind\'ependantes, identiquement distribu\'ees, 
de loi de Cauchy de param\`etre $c>0$. 

Soit la fonction de Lyapounov $V(x) = \abs{x}^\beta$. Pour quelles valeurs de 
$\beta>0$ la quantit\'e $(\cL V)(x)$ est-elle finie~?  
\end{exercise}

\begin{exercise}
On consid\`ere la cha\^ine de Markov sur $\R_+ = [0,\infty[$ donn\'ee par 
\[
X_{n+1} = aX_n + Y_{n+1}
\]
avec $a\in[0,\infty[$, les $Y_n$ \'etant ind\'ependantes, identiquement distribu\'ees, de loi exponentielle de param\`etre $\lambda>0$.

\begin{enumerate}
\item   Donner les probabilit\'es de transition $p(x,y)$ de la cha\^ine.

\item   Calculer 
\[
\int_0^\infty x^k\e^{-\lambda x}\6x
\]
pour $k\in\set{0,1}$. 

\item   Soit la fonction de Lyapounov $V(x) = x$. 
Calculer $(\cL V)(x)$. 

\item   Pour quelles valeurs de $a$ la cha\^ine est-elle \`a croissance sous-exponentielle~?

\item   \`A l'aide de la formule de Dynkin, calculer l'esp\'erance de $X_n$ lorsque $a=1$.
Que se passe-t-il lorsque $n\to\infty$~?

\item   Pour quelles valeurs de $a$ la cha\^ine satisfait-elle une condition de d\'erive g\'eom\'etrique~? Quels en sont les param\`etres~?

\item   Lorsque la condition de d\'erive g\'eom\'etrique est satisfaite, trouver $\alpha\in]0,1[$ et une mesure $\nu$ tels que la condition de minoration soit satisfaite. Que peut-on en d\'eduire~?
\end{enumerate}
\end{exercise}


\bibliographystyle{plain}
\bibliographystyle{alpha}
\bibliography{KESM}

\def\cprime{$'$}
\begin{thebibliography}{1}

\bibitem{Durrett1}
R.~Durrett.
\newblock {\em The Essentials of Probability}.
\newblock Duxbury, 1994.

\bibitem{Hairer_Mattingly_11}
Martin Hairer and Jonathan~C. Mattingly.
\newblock Yet another look at {H}arris' ergodic theorem for {M}arkov chains.
\newblock In {\em Seminar on {S}tochastic {A}nalysis, {R}andom {F}ields and
  {A}pplications {VI}}, volume~63 of {\em Progr. Probab.}, pages 109--117.
  Birkh\"{a}user/Springer Basel AG, Basel, 2011.

\bibitem{Meyn_Tweedie_92}
Sean~P. Meyn and R.~L. Tweedie.
\newblock Stability of {M}arkovian processes. {I}. {C}riteria for discrete-time
  chains.
\newblock {\em Adv. in Appl. Probab.}, 24(3):542--574, 1992.

\bibitem{Nummelin84}
Esa Nummelin.
\newblock {\em General irreducible {M}arkov chains and nonnegative operators},
  volume~83 of {\em Cambridge Tracts in Mathematics}.
\newblock Cambridge University Press, Cambridge, 1984.

\end{thebibliography}

\nocite{Nummelin84,Durrett1}

\vfill

\bigskip\bigskip\noindent
{\small
Nils Berglund \\
Institut Denis Poisson (IDP) \\ 
Universite d'Orleans, Universite de Tours, CNRS -- UMR 7013 \\
B\^atiment de Mathematiques, B.P. 6759\\
45067~Orleans Cedex 2, France \\
{\it E-mail address: }
{\tt nils.berglund@univ-orleans.fr} \\
{\tt https://www.idpoisson.fr/berglund} 

\end{document}